\documentclass[11pt]{amsart}
\usepackage[margin=1in]{geometry}

\usepackage{mathrsfs}
\usepackage{amsfonts}
\usepackage{latexsym,epsfig}
\usepackage{mathabx, stmaryrd}
\usepackage{amsmath, amscd, amsthm, amssymb, centernot}
\usepackage[pdftex]{color}
\usepackage{comment}
\usepackage{marginnote}
\usepackage[colorlinks,citecolor=cyan,linkcolor=magenta]{hyperref}
\usepackage[nameinlink]{cleveref}
\usepackage{enumitem}
\usepackage[normalem]{ulem}

\usepackage{tikz}
\usepackage{tikz-cd}
\usetikzlibrary{matrix,decorations.pathmorphing,arrows}
\tikzset{commutative diagrams/.cd,arrow style=tikz,diagrams={>=stealth'}}
\usepackage{todonotes}

\newcommand{\Z}{\mathbb{Z}}

\newcommand{\R}{\mathbb{R}}

\newcommand{\wt}{\widetilde}
\newcommand{\mc}{\mathcal}

\newcommand{\cone}{\mathrm{cone}}

\newcommand{\ol}{\overline}
\newcommand{\wh}{\widehat}

\newcommand{\mr}{\mathring}
\newcommand{\del}{\partial}

\newcommand{\hbs}{\tau^{(2)}}
\newcommand{\onto}{\twoheadrightarrow}

\newcommand*\tet{{\raisebox{-.3ex}{\text{\includegraphics[width=.9em]{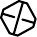}}}}}

\DeclareMathOperator{\fr}{fr}
\DeclareMathOperator{\cl}{cl}

\usepackage{hyperref}

\newtheorem{thm}{Theorem}

\usepackage{thmtools}
\numberwithin{equation}{section}
\newtheorem{theorem}[equation]{Theorem}
\newtheorem{lemma}[equation]{Lemma}
\newtheorem{proposition}[equation]{Proposition}
\newtheorem{corollary}[equation]{Corollary}

\theoremstyle{definition}

\declaretheorem[style=definition,qed=$\lozenge$,sibling=theorem]{remark}

\newcommand{\orb}{\mathcal{O}}
\newcommand{\define}[1]{\textbf{#1}}

\newcommand{\intersect}{\cap}
\newcommand{\boundary}{\partial}

\newcommand\tsim{\kern-.4em\sim}

\newcommand\ssm{\smallsetminus}

\newcommand{\ind}{\mathrm{index}}
\newcommand{\algcusps}{\mathrm{algcusps}}

\newcommand{\cut}{\! \bbslash \!}
\newcommand{\front}{\text{front}}
\newcommand{\back}{\text{back}}

\renewcommand{\phi}{\varphi}
\renewcommand{\epsilon}{\varepsilon}

\DeclareMathOperator{\closure}{cl}
\DeclareMathOperator{\intr}{int}
\DeclareMathOperator{\brloc}{brloc}

\DeclareMathOperator{\cusps}{cusps}
\DeclareMathOperator{\sign}{sign}
\DeclareMathOperator{\core}{core}
\DeclareMathOperator{\prongs}{prongs}

\begin{document}

\title[Transverse surfaces and flows]{Transverse surfaces and pseudo-Anosov flows}
\author[M.P. Landry]{Michael P. Landry}
\address{Department of Mathematics\\
Saint Louis University}
\email{\href{mailto:michael.landry@slu.edu}{michael.landry@slu.edu}}
\author[Y.N. Minsky]{Yair N. Minsky}
\address{Department of Mathematics\\ 
Yale University}
\email{\href{mailto:yair.minsky@yale.edu}{yair.minsky@yale.edu}}
\author[S.J. Taylor]{Samuel J. Taylor}
\address{Department of Mathematics\\ 
Temple University}
\email{\href{mailto:samuel.taylor@temple.edu}{samuel.taylor@temple.edu}}
\thanks{Landry was partially supported by NSF postdoctoral fellowship DMS-2013073.
Minsky was partially supported by DMS-2005328.
Taylor was partially supported by DMS-2102018 and a Sloan Research Fellowship.}

\begin{abstract}
Let $\phi$ be a transitive pseudo-Anosov flow on an oriented, compact $3$-manifold $M$, possibly with toral boundary. We characterize the surfaces in $M$ that are (almost) transverse to $\phi$. When $\phi$ has no perfect fits (e.g. $\phi$ is the suspension flow of a pseudo-Anosov homeomorphism), we prove that any Thurston-norm minimizing surface $S$ that pairs nonnegatively with the closed orbits of $\phi$ is almost transverse to $\phi$, up to isotopy. This answers a question of Cooper--Long--Reid. Our main tool is a correspondence between surfaces that are almost transverse to $\varphi$ and those that are relatively carried by any associated veering triangulation. The correspondence also allows us to investigate the uniqueness of almost transverse position, to extend Mosher's Transverse Surface Theorem to the case with boundary, and more generally to characterize when relative homology classes represent Birkhoff surfaces.
\end{abstract}
\maketitle

\setcounter{tocdepth}{1}
\tableofcontents

\section{Introduction}
\subsection{Overview}
Let $\phi$ be a transitive pseudo-Anosov flow on a compact, oriented $3$-manifold $M$, possibly with toral boundary. Important examples include the geodesic flow on the unit tangent bundle of a closed hyperbolic surface and the suspension flow on the mapping torus of a pseudo-Anosov homeomorphism. A pseudo-Anosov flow is called \define{circular} if it is such a suspension flow, up to isotopy and reparametrization.

A basic problem 
is to classify the surfaces in $M$ that are {transverse} to a pseudo-Anosov flow, up to isotopy. Mosher's Transverse Surface Theorem \cite{Mos92,mosher1992dynamical} from the early 1990s
 answers the related question of which homology classes have representatives almost transverse to the flow $\varphi$. Here, \textit{almost transverse} means that 
 a mild operation
 must first be performed on the singular orbits of flow; see \Cref{sec:flows} for details.
When $M$ is closed, Mosher proves that if a surface $S$ has nonnegative algebraic intersection with each closed orbit of $\varphi$, then $S$ is homologous to a surface almost transverse to $\varphi$.

Although Mosher's theorem is quite powerful, there are situations in which one needs to know whether a given surface $S$ is transverse to $\varphi$ up to {isotopy} rather than homology (see e.g. the discussion following \Cref{A}). Thirty years ago Cooper--Long--Reid asked, in the special case where $\phi$ is circular and $M$ is closed: \emph{which surfaces in $M$ can be made transverse to $\phi$ up to isotopy \cite[Sec. 4, Ex. 2]{cooper1994bundles}?}

Among our results, we give a complete answer to this question in the more general setting of transitive pseudo-Anosov flows on compact manifolds. 
An obvious necessary condition for $S$ to be almost transverse to $\phi$ up to isotopy is that $S$ have nonnegative algebraic intersection with each closed orbit. A more subtle one is that $S$ be \textbf{taut}, meaning it realizes the Thurston norm of its homology class and has no nullhomologous collection of components (\cite[Sec. 2]{mosher1992dynamical} and \Cref{lem:taut_from_flow}).  When specialized to the setting of Cooper--Long--Reid's question, our results say these two necessary conditions are in fact sufficient for almost transversality (\Cref{A}); we also describe exactly when almost transversality can be promoted to transversality (\Cref{E}).

\smallskip
Our characterization follows from a general treatment of almost transversality for surfaces with respect to an arbitrary transitive pseudo-Anosov flow $\phi$, which can be informally summarized as follows. See the next section for the formal statements.

\begin{itemize}
\item \textbf{Characterization via veering triangulations:} Our primary technical result is that, up to isotopy, a surface $S$ is almost transverse to $\phi$ if and only if it is relatively carried by any veering triangulation associated to $\phi$ (\Cref{E}). The point is that both the veering triangulation and the relatively carried condition allow us to study transversality in a way which is purely combinatorial. Moreover, the surface is honestly transverse to $\phi$ if and only if it has a specific type of carried position with respect to the veering triangulation.
\item \textbf{Transverse surface theorems:} Using the correspondence above, we prove a general transverse surface theorem that states that after isotopy $S$ is almost transverse if and only if it is taut, algebraically nonnegative on closed orbits, and can be isotoped to have no negative intersections with a specific finite collection of orbits (\Cref{B}). This last condition is vacuous for a large class of pseudo-Anosov flows (including circular flows), but we show by example that it is necessary in general (\Cref{thm:example}).
\item \textbf{Uniqueness of transverse position:} We show that the position of an almost transverse surface is essentially unique, up to isotopy along the flow. This includes the nontrivial fact that there is a unique way to minimally modify the flow (in a certain combinatorial sense) to make the surface transverse (\Cref{D}). 
\item \textbf{Existence of Birkhoff surfaces:} 
Plenty of pseudo-Anosov flows do not have (almost) transverse surfaces (e.g. geodesic flow of a hyperbolic surface). However, all pseudo-Anosov flows admit \emph{Birkhoff surfaces}. By extending Mosher's transverse surface theorem to manifolds with toral boundary (\Cref{th:mosher_with_boundary}) we give a general characterization for which relative homology classes represent (almost) Birkhoff surfaces for $\phi$ (\Cref{C}). This also extends Fried's characterization of relative classes representing Birkhoff sections (i.e. Birkhoff surfaces meeting every closed orbit in bounded time). 
\end{itemize}

As mentioned above, our results include the case of pseudo-Anosov flows on manifolds with (toral) boundary. See \Cref{sec:flows} for a precise definition. There are two important reasons for this, beyond simply generalizing from the closed case. First, the application to Birkhoff surfaces stated above is fairly straightforward after proving the transverse surface theorem (\Cref{th:mosher_with_boundary}) for manifolds with boundary. Second, such flows can be used to understand the structure of the Thurston norm on manifolds with boundary in a way that has only previously been studied in the closed case. See \Cref{th:flows_rep_faces} for details.

\subsection{Results on transverse surfaces}
Throughout the introduction we work with a pseudo-Anosov flow $\varphi$ on a compact, oriented $3$--manifold $M$. We allow $\partial M \neq \emptyset$ in which case each component of $\partial M$ is a torus. See \Cref{sec:flows} for the definitions.  

Our first result characterizes almost transverse surfaces for the class of pseudo-Anosov flows \emph{without perfect fits}. Such flows were first studied by Fenley \cite{fenley1998structure, fenley1999foliations} and include the class of circular (i.e. suspension) flows mentioned above. For our purposes here, it suffices to know that $\varphi$ has no perfect fits if and only if there are no \define{anti-homotopic orbits}, i.e. closed orbits $\gamma_1,\gamma_2$ with $\gamma_1$ homotopic to $-\gamma_2$.
See \Cref{sec:pfits}, and particularly \Cref{rmk:antihom}, for the precise definition and further discussion.

The following lemma says that everything ``seen" by these flows at the level of homology is also seen at the level of isotopy:

\begin{thm}[Strong transverse surface theorem]
\label{A}
Let $\phi$ be a pseudo-Anosov flow with no perfect fits on $M$, possibly with boundary. Then a properly embedded, oriented surface $S$ is almost transverse to $\phi$, up to isotopy, if and only if $S$ is taut and pairs nonnegatively with the closed orbits of $\phi$. 
\end{thm}

See \Cref{lem:taut_from_flow} and \Cref{th:stst}. We remark that in general the classes in $H_2(M,\partial M)$ that pair nonnegatively with the closed orbits of $\phi$ form a closed cone, which under the hypotheses of \Cref{A}, is exactly the cone over a closed face of the Thurston norm ball. This was proved by Mosher when $\partial M = \emptyset$ and we give a proof in the case with boundary (\Cref{th:flows_rep_faces}). Hence, \Cref{A} implies that the surface $S$ is almost transverse to $\varphi$ up to isotopy if and only if $e_\varphi([S]) = \chi(S)$, where $e_\varphi$ is the Euler class of the flow, suitably defined, and $\chi$ is Euler characteristic. See \Cref{sec:flowsrepresentfaces} for additional details.

\Cref{A} is a key ingredient in \cite[Theorem A]{landry2023endperiodic} where we prove that every atoroidal endperiodic map is isotopic to the first return map of a circular pseudo-Anosov flow to a leaf of a transverse depth one foliation.
\smallskip

To complement \Cref{A}, we show by explicit example in \Cref{thm:example} that the hypothesis of no perfect fits cannot be entirely dropped. However, we \emph{can} characterize transverse surfaces for general transitive pseudo-Anosov flows using an auxiliary (noncanonical) collection of closed orbits. For the statement, a finite collection of closed orbits $\kappa$ \emph{kills the perfect fits of $\varphi$} if there are no anti-homotopic orbits in the punctured manifold $M \ssm \kappa$; again see \Cref{sec:pfits}. The following is a restatement of \Cref{th:general_stst}.

\begin{thm} \label{B}
Let $\phi$ be a transitive pseudo-Anosov flow on $M$, and let $\kappa$ be any finite collection of closed regular orbits of $\phi$ that kills its perfect fits. 
Then an oriented surface $S$ in $M$ is isotopic to a surface that is almost transverse to $\phi$ if and only if
\begin{enumerate}
\item $S$ is taut,
\item $S$ has nonnegative (homological) intersection with the closed orbits of $\phi$,
and 
\item the algebraic and geometric intersection numbers of $S$ with $\kappa$ are equal.
\end{enumerate}
\end{thm}

Given a taut surface $S$ that has nonnegative algebraic intersection with each closed orbit of $\phi$, \Cref{B} gives a simple criterion for whether $S$ is almost transverse to $\phi$ up to isotopy: one need only check whether $S$ can be isotoped so that all its intersection points with the oriented link $\kappa$, if any, are positive.

\smallskip

Our general technique, which handles pseudo-Anosov flows on manifolds with boundary, has applications to the more general situation of Birkhoff surfaces. A \define{Birkhoff surface} (sometimes called a partial section) of $\phi$ is an immersed surface $\Sigma$ in $M$ whose boundary components cover closed orbits of $\phi$ and 
whose interior is embedded and transverse to $\phi$. More generally, a Birkhoff surface of a dynamic blowup of $\phi$ will be called an \emph{almost} Birkhoff surface of $\phi$.

The next theorem generalizes Fried's result on Birkhoff sections \cite[Theorem N]{fried1982geometry} and seems new even for geodesic flow on hyperbolic surfaces. See \Cref{th:Birk} for statement that includes conditions for when the `almost' can be dropped.

\begin{thm}[Representing relative classes with Birkhoff surfaces]
\label{C}
Let $\phi$ be a transitive pseudo-Anosov flow and let $\kappa$ be any collection of closed orbits. Then a class $\eta \in H_2(M, \kappa) = H^1(M \ssm \kappa)$ is represented by an almost Birkhoff surface $\Sigma$ in $M$ with $\partial \Sigma \subset \kappa$ if and only if it is nonnegative on closed orbits of $M \ssm \kappa$.

\end{thm}

\smallskip

Finally, we turn to the sense in which `almost transverse' position is essentially unique.
For this, note that the statement 
\[
\text{\emph{``$S$ is almost transverse to the pseudo-Anosov flow $\phi$ up to isotopy"}}
\]
contains two potential ambiguities. The first is that it does not specify a particular dynamic blowup to which $S$ is transverse up to isotopy. The second is that it says nothing about the \emph{way} in which $S$ can be realized transversely to a given dynamic blowup $\phi^\sharp$---can one interpolate between any two such positions by flowing along $\phi^\sharp$, or are there multiple transverse positions such that any isotopy between them must move through non-transverse surfaces?

These ambiguities are resolved with the following result, a combination of \Cref{thm:combinatorial equivalence} and \Cref{thm:same shadow}.

\begin{thm} \label{D}
Let $\phi$ be a pseudo-Anosov flow on $M$ and let $S_1$ and $S_2$ be isotopic properly embedded surfaces which are minimally transverse to dynamic blowups $\phi_1^\sharp$ and $\phi_2^\sharp$, respectively. Then $\phi_1^\sharp$ and $\phi_2^\sharp$ are combinatorially equivalent.

Moreover, if the surfaces are transverse to a single blowup $\phi^\sharp$, then $S_1$ and $S_2$ are isotopic along flowlines of $\phi^\sharp$.
\end{thm}

We note here that `combinatorial equivalence' is a slightly weaker notion than orbit equivalence; see \Cref{global blowups}.
The proof of the theorem requires a detailed analysis of the flow space shadows of transverse surfaces that is carried out in \Cref{sec:generalshadows}.

\subsection{Veering triangulations as the main tool}
\label{sec:motivation}

As before, let $\varphi$ be a transitive pseudo-Anosov flow on $M$ and let $\kappa$ be a finite collection of closed orbits that kills its perfect fits. Such a collection exists by work of Fried \cite{fried1983transitive} and Brunella \cite{brunella1995surfaces}; see also Tsang \cite[Proposition 2.7]{Tsang_geodesic}. 
We let $\kappa_s$ denote the union of $\kappa$ and any singular orbits or boundary components of $M$. 

Our primary technical tool is a canonical \emph{veering triangulation} on the manifold $M \ssm \kappa_s$ which was essentially constructed by Agol and Gueritaud.
 In our previous work (\cite{LMT21}), we show the $2$-skeleton of this triangulation is a branched surface properly embedded in $M \ssm \kappa_s$ that is positively transverse to $\varphi$. This is exactly what makes the veering triangulation useful in understanding transverse surfaces.
 
 Given a properly embedded surface $S$ {in} $M$, one can ask whether $S$ can be put into a `relatively carried' position with respect to the triangulation. In essence, this means $S$ is carried by the branched surface away from $\kappa_s$ and intersects a neighborhood of $\kappa_s$ 
efficiently.
See \Cref{sec:veering_background} for the precise definitions and details. 

In \cite[Introduction]{Landry_norm} it was remarked that ``being [relatively] carried by $\tau^{(2)}$ is a combinatorial version of almost transversality." 
Our main result relating the veering triangulation to transverse surfaces makes this precise:

\begin{thm}[Almost transverse if and only if relatively carried]
\label{E}
		\label{th:honest}
	Let $\varphi$ be a transitive pseudo-Anosov flow on $M$ and let $\tau$ be any dual veering triangulation. For any surface $S$ in $M$, 
\[
	S \text{ is almost transverse to } \varphi \quad \iff \quad S \text{ is relatively carried by } \tau.	
\]

	Moreover, $S$ is \emph{honestly} transverse to $\varphi$ if and only if it is relatively carried $\tau$ without complementary annuli.
\end{thm}

See \Cref{thm:ATtocarried} and \Cref{cor:no_lad}.

\subsection*{Basic conventions}
\label{sec:veering_background}
In this paper all the 3-manifolds we consider are oriented, and all homology and cohomology groups have coefficients in $\R$. If $B$ is a subset of a metric space $A$, we denote the completion of $A-B$ in the induced path metric by $A\cut B$.

\subsection*{Acknowledgments} We thank Chi Cheuk Tsang for his detailed comments on an earlier draft.

\section{Veering triangulations and relatively carried surfaces}

We begin with some background on veering triangulations. Our goal is to establish \Cref{combiSTST}, which itself extends the main theorem from \cite{Landry_norm}.

\subsection{Index and train tracks}

A \textbf{surface with cusps} is a surface $S$ with a discrete collection of points in $\del S$ called \textbf{cusps} which are modeled on the point $(0,0)\in\{(x,y)\in \R^2\mid y\ge\sqrt{|x|}\}$. 
The \textbf{index} of a compact surface with cusps $S$ is 
\[
\ind(S)=2\chi(S)-\#\cusps(S).
\]

Recall that a \textbf{train track} in a surface $S$ is a 1-complex in $S$ with an everywhere-defined tangent space. This means that the edges meeting a given vertex are ``combed" so as to be tangent. The edges of a train track are called \textbf{branches} and the vertices are called \textbf{switches}. In this paper all the train tracks we consider will lie in the interior of surfaces.

Given a train track $t\subset S$, a \textbf{patch} of $t$ is a component of $S\cut t$. A patch naturally has the structure of a surface with cusps, with the cusps arising from the switches of $t$. The index of these patches is additive, in the sense that $2\chi(S)=\ind(S)=\sum \ind(p)$ where the sum is taken over all patches of $t$.

\subsection{Veering triangulations}

An ideal tetrahedron is a tetrahedron minus its vertices, and an ideal triangulation of a (necessarily noncompact) 3-manifold $Z$ is a cellular decomposition of $Z$ into ideal tetrahedra.

A \textbf{veering tetrahedron} $\Delta$ is an ideal tetrahedron together with the following extra data:
\begin{itemize}
\item Two faces are cooriented outward, and two are cooriented inward. These are called the top and bottom faces of $\Delta$, respectively.
\item The edge along which the top faces meet is labeled $\pi$, and so is the edge along which the bottom faces meet. These are called the top and bottom edges of $\Delta$, respectively.
\item The remaining edges, called equatorial edges, are labeled 0.
\item The equatorial edges are colored red or blue in an alternating fashion (see \Cref{fig:veertet}). In some of the literature the red and blue edges are called right and left veering, respectively, and the property of being right or left veering is called the \textbf{veer} of an edge.
\end{itemize}

Up to oriented equivalence, there are two veering tetrahedra in $\R^3$. The ``standard" veering tetrahedron is distinguished by the property that when viewed from above and projected to the plane as in \Cref{fig:veertet}, the red edges have positive slope and the blue edges have negative slope.

\begin{figure}
\centering
\includegraphics[]{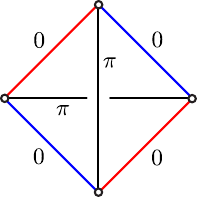}
\caption{The standard veering tetrahedron. Each face is cooriented out of the page.}
\label{fig:veertet}
\end{figure}

A \textbf{veering triangulation} of an oriented 3-manifold $Z$ is an ideal triangulation of $Z$ together with an assignment of veer to each edge and coorientation to each face, such that:
\begin{itemize}
\item each ideal tetrahedron is equivalent to the \emph{standard} veering tetrahedron via an orientation-preserving map, and
\item the sum of the 0 and $\pi$ labels around each edge of the ideal triangulation is $2\pi$.
\end{itemize}

We will always think of the the 0 or $\pi$ label of an edge of a veering tetrahedron as describing the dihedral angle for that edge. Face coorientations, together with the condition that edge labels sum to $2\pi$ around each edge, endow the 2-skeleton with the structure of a cooriented ``branched surface" as depicted in \Cref{fig:edgepinch}. We now briefly discuss branched surfaces.

\begin{figure}
\centering
\includegraphics[height=1in]{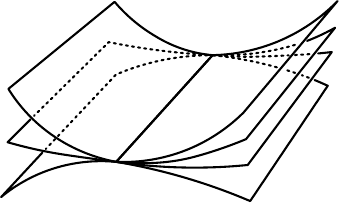}
\caption{A local picture of part of an edge of a veering triangulation, showing the 2-skeleton's branched surface structure.}
\label{fig:edgepinch}
\end{figure}

\subsection{Branched surfaces and carrying}
A \textbf{branched surface} is the 2-dimensional analogue of a train track: it is a 2-complex with a continuously varying tangent plane at each point, locally modeled on the quotient of a finite stack of  disks by identifying half disks of adjacent disks, requiring that the inclusion of each disk be smooth. If $B$ is a branched surface, then the \textbf{branch locus} $\brloc(B)$ is the union of all nonmanifold points of $B$, and the \textbf{sectors} of $B$ are the components of $B\cut \brloc(B)$.

Let $B$ be a branched surface living in a 3-manifold. A \textbf{standard neighborhood} of $B$ is a small tubular neighborhood of $B$ foliated in a standard way by intervals (see \cite{floyd1984incompressible, oertel1986homology}). This foliation is called the \textbf{vertical foliation} of $N(B)$. We say that $B$ is \textbf{cooriented} if the vertical foliation is oriented. If $\Lambda$ is a 2-dimensional lamination embedded in $N(B)$ transversely to the vertical foliation, we say $\Lambda$ is \textbf{carried} by $B$ (note $\Lambda$ could simply be a compact surface). If $B$ is cooriented and $\Lambda$ is cooriented, when we say that $\Lambda$ is \textbf{carried} by $B$ we additionally assume that each leaf of the vertical foliation passes from the negative side to the positive side of $\Lambda$ at each intersection point. We refer the reader to \cite{floyd1984incompressible, oertel1986homology} for more details on branched surfaces and carrying. 

There is a homotopy equivalence $N(B)\onto B$, called the \textbf{collapsing map}, that collapses the leaves of the vertical foliation. If $S$ is a surface carried by $B$, the \textbf{collapse} of $S$ is the image of $S$ under the collapsing map.

\begin{figure}
\centering
\includegraphics[height=1.25 in]{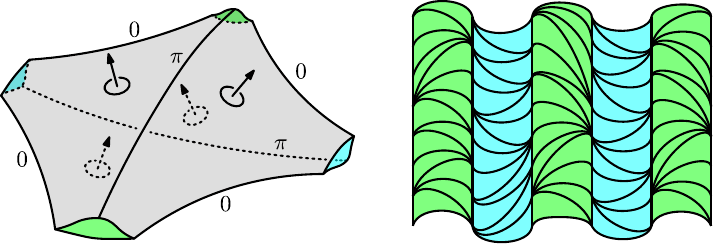}
\caption{Left: a truncated veering tetrahedron. Right: the tips of truncated tetrahedra divide the boundary of $M_U$ into upward and downward ladders (green and blue, respectively) separated by ladderpole curves.}
\label{fig:trunc}
\end{figure}

\subsection{Relative branched surfaces and relative veering triangulations}
\label{sec:reltube}
Let $M$ be a compact oriented 3-manifold with toral boundary. A \textbf{tube system} $U$ for $M$ is the union of a small closed tubular neighborhood of $\del M$ with a small closed tubular neighborhood of a link $L$ in $M$. Here it is possible that $\del M=\varnothing$ or $L=\varnothing$. The solid torus components of $U$ are called \textbf{solid tubes} and the components homeomorphic to $T^2\times I$ are called \textbf{hollow tubes}.
A \textbf{veering triangulation of $M$ relative to $U$} is a veering triangulation of $Z=\intr(M)-L$ such that 
\begin{itemize}
\item the 2-skeleton $\hbs$ of $\tau$ intersects $\del U$ transversely, and 
\item $\hbs\cap U$ is diffeomorphic to $(\hbs\cap \del U)\times [0,\infty)$.
\end{itemize}
We remark that for any veering triangulation on $Z$, a tube system of $M$ can be chosen to satisfy these conditions.
We also use the terminology ``relative veering triangulation of $M$ with tube system $U$."

\smallskip
Let $\tau$ be a relative veering triangulation of $M$ with tube system $U$. We set the notation
\[
\tau_U=\hbs\cap M\cut U.
\]
We see that $\tau_U$ is a branched surface in $M_U:=M\cut U$. There is an induced cellular decomposition of $M_U:=M\cut U$ whose cells are truncated veering tetrahedra, as shown on the left side of \Cref{fig:trunc}.
Each tip of a truncated veering tetrahedron, corresponding to an ideal vertex, is combinatorially a triangle with cooriented sides. However, the veering structure endows each tip with the smooth structure of a disk with 2 cusps. The 1-skeleton of the tessellation of $\del M$ is a train track $\del \tau_U\colon= \tau_U\cap \del M_U$ whose patches are the tips of truncated veering tetrahedra. 

A tip is called \textbf{upward} if it has two sides cooriented outward, and \textbf{downward} otherwise. The definition of veering triangulations implies, with some work, that the union of all upward tips meeting $\del M_0$ is a collection of annuli called \textbf{upward ladders}, and symmetrically for downward tips and \textbf{downward ladders}. See the right side of \Cref{fig:trunc}. The upward and downward ladders meet along a collection of curves called \textbf{ladderpole curves}. For a given component $T$ of $\del M_U$, the slope of the ladderpole curves on $T$ is called the \textbf{ladderpole slope} for $T$. This structure was first observed in \cite[Section 2]{futer2013explicit}; see that paper for more details.

A properly embedded surface $S\subset M$ is \textbf{relatively carried by $\tau$} if $S\cap (M\cut U)$ is carried by the branched surface $\tau_U$ in the standard sense of branched surfaces, and each component of $S\cap U$ is either a meridional disk in a solid tube or an annulus with at most one boundary component on $\del M$. 
There are two types of possible annuli in $S \cap U$; those that meet $\partial M$ and those that do not. We call the latter type \define{ladderpole annuli} since both their boundary components have ladderpole slope in $\partial M_U$.

Let $N(\tau_U)$ be a standard neighborhood of $\tau_U$ in $M\cut U$. The collapsing map $N(\tau_U)\onto \tau_U$ extends over $U$ by a map that is the identity away from a small neighborhood of $\del M_U\subset M$. Hence we can speak of the collapse of a surface relatively carried by $\tau_U$---this will be the union of some sectors of $\tau_U$ with some disks and annuli lying in $U$.

\subsection{Cusped tori}

Let $D$ be a compact disk with $n>0$ cusps in its boundary, and let $f\colon D\to D$ be an orientation-preserving diffeomorphism. The mapping torus $T$ of $f$ is called a \textbf{cusped solid torus} (see \Cref{fig:cuspedtori}), and the suspensions of the cusps of $D$ under $f$ are called the \textbf{cusp curves} of $T$. The \textbf{index} of $T$ is defined to be the index of $D$, that is $2(1)-n$. We say that $n$ is the \textbf{number of prongs} of $T$, denoted $\prongs(T)$. In particular $\ind(T)=2-\prongs(T)$. Note that $\prongs(T)$ is not necessarily the same as the number of cusp curves; in general the number of cusp curves is equal to $\prongs(T)$ divided by the order of the permutation obtained by restricting $f$ to the cusps of $D$.

\begin{figure}
\centering
\includegraphics[height=1.25in]{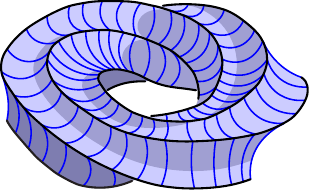}
\caption{An example of a cusped solid torus. To obtain a cusped torus shell, we would drill out a neighborhood of the core circle.}
\label{fig:cuspedtori}
\end{figure}

A \textbf{cusped torus shell} is the object obtained by modifying the definition of a cusped solid torus by replacing the disk $D$ with an annulus with one smooth boundary component and one with $n$ cusps. The cusp curves of cusped torus shells are defined as for cusped solid tori. We do not define the index of a cusped torus shell because there is no canonical choice of meridian.

\subsection{Stable and unstable branched surfaces}\label{sec:bs}

Let $M$ be a compact oriented 3-manifold, and let $\tau$ be a veering triangulation of $M$ relative to a tube system $U$.
There are two branched surfaces with generic branch locus associated to $\tau$. The \textbf{stable branched surface}, denoted $B^s$, intersects each veering tetrahedron as shown the top row of \Cref{fig:stableunstable}. The \textbf{unstable branched surface}, denoted $B^u$, intersects each veering tetrahedron as shown in the bottom row of \Cref{fig:stableunstable}. In this paper we are not concerned with the positions of $B^s$ and $B^u$ relative to each other---it is enough to know that they intersect each tetrahedron as described. We refer the reader to \cite[Section 4.2]{LMT20} for additional details.

\begin{figure}
\begin{center}
\includegraphics[height=2.5in]{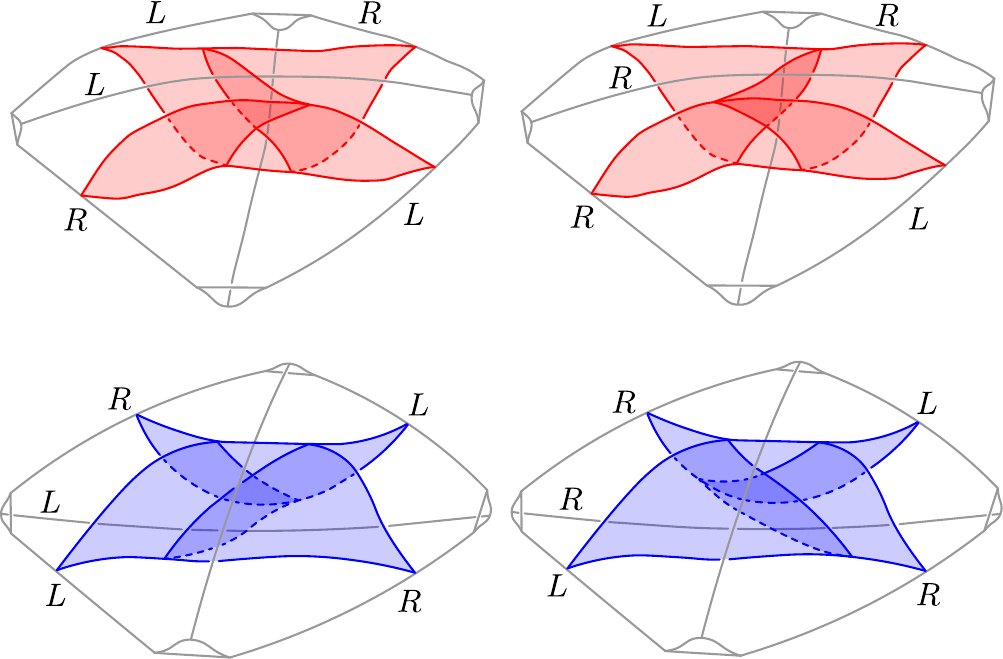}
\caption{The stable and unstable branched surfaces of a veering triangulation intersect veering tetrahedra as shown, where $L$ and $R$ labels denote left and right veering edges, respectively. Note the veers of the bottom edges in the top row are irrelevant, as are the veers of the top edges in the bottom row.}
\label{fig:stableunstable}
\end{center}
\end{figure}

The components of $M\cut B^s$ and $M\cut B^u$ are cusped solid tori and cusped torus shells, each of which contains a single solid or hollow tube, respectively. We will refer to these as the \textbf{cusped tori of $B^s$ and $B^u$}.

Let $U_0$ be a component of $U$, and let $T_0^s$ and $T_0^u$ be the cusped tori of $B^s$ and $B^u$ containing $U_0$. Let $t_0$ be the tessellation induced by the intersection of $\tau$ with $\del U_0$. Then as described in \cite[Section 5.1.1]{LMT20}, the cusp curves of $T_0^s$ are in canonical bijection with the cores of upward ladders of $t_0$. Similarly the cusp curves of $T_0^u$ are in canonical bijection with the cores of downward ladders. In particular $T_0^s$ and $T_0^u$ are diffeomorphic.

\subsection{Spheres, disks, tori, and annuli}

We continue to assume $\tau$ is a relative veering triangulation of $M$ relative to $U$.

The existence of $\tau$ restricts the types of properly embedded surfaces in $M$ with nonnegative Euler characteristic that can exist.
As observed in \cite{Schleimer2019veering}, the branched surfaces $B^s$ and $B^u$ are laminar and therefore carry essential laminations; it follows that $M$ is irreducible \cite{Li02, GO89}. (For a given boundary component of $M$, the degeneracy slope of both laminations is equal to the ladderpole slope of the veering triangulation). The next lemma lets us analyze disks and tori.

\begin{lemma}\label{lem:noposindex}
Let $M$ be a 3-manifold admitting a relative veering triangulation $\tau$. Let $S$ be any 
properly embedded surface in $M$. Then $S$ can be isotoped rel $\del M$ so that $S\cap B^s$ and $S\cap B^u$ have no patches of positive index disjoint from $\del S$.
\end{lemma}

\begin{proof}
This is \cite[Lemma 5.1]{Landry_norm} in the additional generality of our setting, i.e. $S$ may have boundary. The same proof (which uses irreducibility of $M$) applies.
\end{proof}

Now suppose that $D$ is a properly embedded disk in $M$. Applying \Cref{lem:noposindex}, we can isotope $D$ to have no patches of positive index disjoint from $\del D$. Euler characteristic forces $D\cap B^s=\varnothing$, so $D$ lies in a cusped torus shell of $B^s$ and is isotopic into $\del M$.

Now we consider annuli and tori. We say that the relative veering triangulation $\tau$ is \textbf{strict} if each cusped solid torus of $B^s$ has negative index. By the discussion above this is equivalent to requiring the same of $B^u$, or requiring that the meridional curve for every solid tube has geometric intersection number $\ge6$ with the union of all ladderpole curves.

Suppose that $\tau$ is strict, and that $A$ is a properly embedded annulus in $M$. Note that no patch of $A\cap B^s$ meeting $\del A$ can have positive index. After an application of \Cref{lem:noposindex}, all patches of $A\cap B^s$ disjoint from $\del A$ will have nonpositive index. This implies that all patches of $A\cap B^s$ have index $0$. By strictness, none of the patches disjoint from $\del A$ is a meridional disk for a solid tube; hence $A$ may be isotoped to lie outside of $U$ away from a regular neighborhood of $\del A$, unless $A$ already lies entirely inside a component of $U$. Since $M- U$ admits a complete hyperbolic metric by \cite{hodgson2011veering, futgue11}, $A$ must be homotopic rel $\del A$ into $\del M$.

Continuing to assume $\tau$ is strict, suppose that $T$ is a torus embedded in $M$. As with annuli, we can apply \Cref{lem:noposindex} to arrange for $T\cap B^s$ to have no patches of positive index. In this case, strictness implies that $T$ can be isotoped to lie outside of $U$. As above, $M-U$ admits a complete hyperbolic metric. Hence, $T$ is either parallel to a component of $\del M$ or $T$ is compressible in $M-U$, hence also in $M$.

We summarize the above discussion in the following proposition.

\begin{proposition} \label{prop:top}
If $\tau$ is a relative veering triangulation of a 3-manifold $M$, then:
\begin{enumerate}[label=(\alph*)]
\item $M$ is irreducible and $\del$-irreducible, and
\item if $\tau$ is strict, then $M$ is also atoroidal and anannular.
\end{enumerate}
\end{proposition}

\subsection{The dual graph}

Let $M$ be a compact 3-manifold, and let $\tau$ be a veering triangulation of $M$ relative to a tube system $U$.

We construct a graph by placing a vertex in each tetrahedron of $\tau$ and placing an edge between two vertices for each face identification. There is an obvious embedding of $\Gamma$ into $M$ such that each edge of $\Gamma$ passes through one face of $\tau$, and we can orient each edge so that these intersections are compatible with the coorientation of the $\tau$-faces. This directed graph (together with its embedding into $M$) is called the \textbf{dual graph}, and denoted $\Gamma$. Since each vertex of $\Gamma$ has two incoming and two outgoing edges, $\Gamma$ is a 1-cycle representing a class in $H_1(M)$ that (by a slight abuse) we also call $\Gamma$.

It is clear from \Cref{fig:stableunstable} that the branch loci of both $B^s$ and $B^u$ can be canonically identified with $\Gamma$. We will always think of $\brloc(B^s)$ and $\brloc(B^u)$ as directed graphs using this identification with $\Gamma$. Note that this endows each cusp curve of each cusped torus of $B^s$ and $B^u$ with an orientation. It turns out that these orientations are coherent, in the sense that any two cusp curves of a cusped torus are isotopic preserving orientation.

If $T$ is a cusped solid torus, there is a natural orientation of its core curve, defined by the property that any cusp curve is homotopic to a positive multiple of the core curve.

\subsection{Removable annuli and efficient carrying.}

Let $\tau$ be a veering triangulation of $M$ relative to a tube system $U$, and let $S$ be a surface in $M$ relatively carried by $\tau$. This position is not (combinatorially) unique in general, and it will be important for us (e.g. in \Cref{prop:efficientblowup}) to discuss a particular simplification that can eliminate components of $S\cap U$.

For this, suppose that $A$ is an annulus component of $S\cap U$ disjoint from $\del M$, such that after collapsing $S$, both of its boundary components lying on \emph{adjacent} ladderpoles of some ladder $\ell$. If this collapsing of $A$ is homotopic in $U$ to $\ell$ fixing $\del A$, then we say that $A$ is a \textbf{removable annulus}. 
If $S$ is relatively carried by $\tau$ without any removable annuli, then we say that $S$ is \textbf{efficiently relatively carried by $\tau$}. 
The next lemma justifies this terminology by saying that ``removable annuli can be removed":

\begin{lemma}\label{lem:noladders}
Let $\tau$ be a veering triangulation of a 3-manifold $M$ relative to a tube system $U$. Let $S$ be a surface relatively carried by $\tau$. Then $S$ is isotopic to a surface efficiently relatively carried by $\tau$.
\end{lemma}
\begin{proof}
Suppose $A$ is a removable annulus, with both components of $\del A$ carried by the ladderpoles of a ladder $L$. If $A$ is innermost, meaning that the component of $U\cut A$ containing $L$ is disjoint from $S$, then we may apply an \emph{annulus move} as defined in \cite[Sec. 4.5]{Landry_norm} and shown in \Cref{fig:annulusmove} (technically, this is the reverse of what that paper calls an annulus move) to eliminate $A$ while preserving the fact that $S$ is relatively carried by $\tau$. Applying this argument finitely many times to innermost such annuli, we are done.
\end{proof}

\begin{figure}
\begin{center}
\includegraphics[width=4.25in]{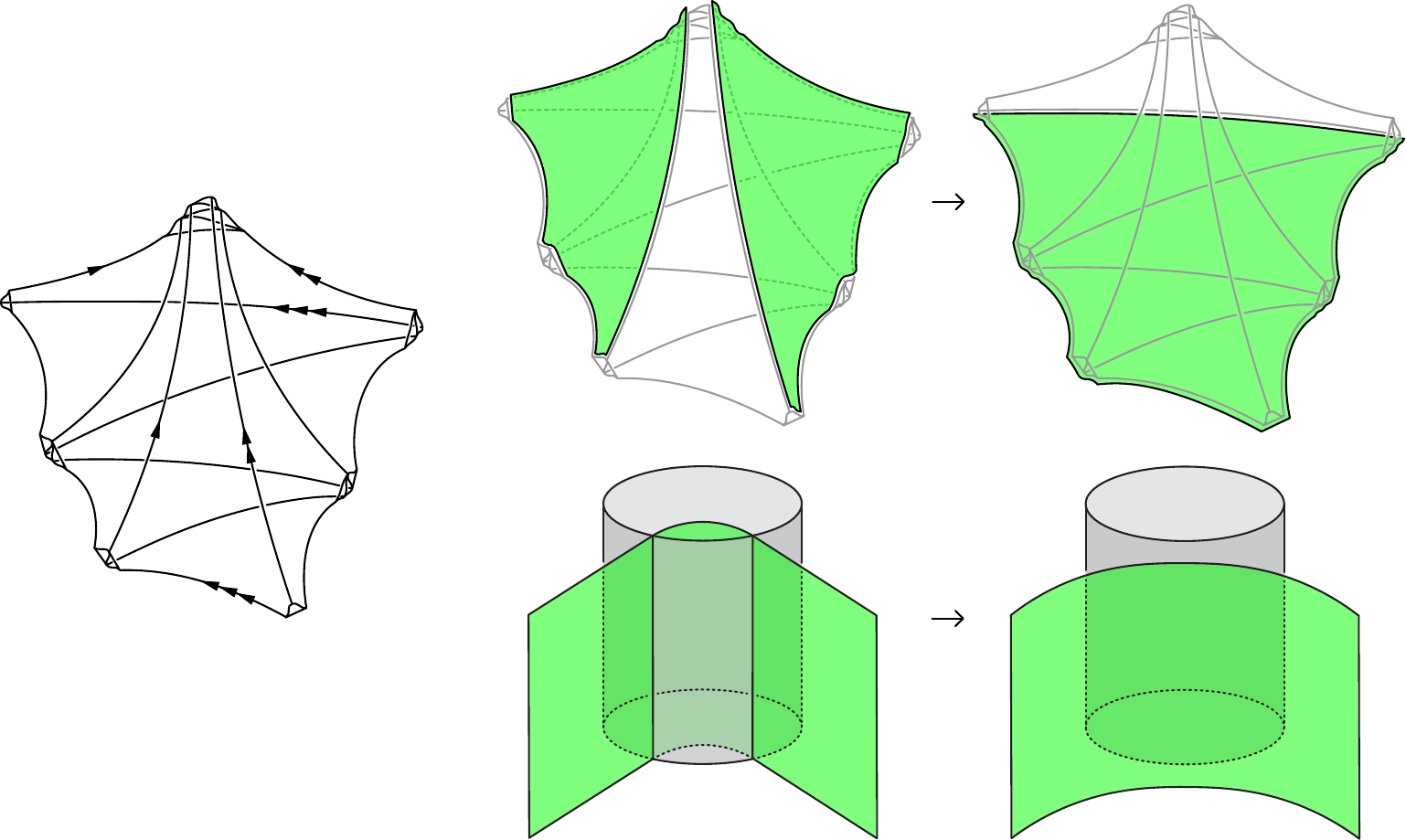}
\caption{Left: part of a veering triangulation with some edge identifications specified. Right, top row: before and after an annulus move corresponding to an upward ladder, looking only at $M\cut U$. Right, bottom row: a cartoon of what this annulus move accomplishes---it pulls an annulus contained in a relatively carried surface out of a tube. Note the tube could be hollow, although a solid one is pictured here.}
\label{fig:annulusmove}
\end{center}
\end{figure}

\subsection{Thurston's norm on homology and its dual}\label{sec:tnorm}

We now briefly define the Thurston norm on $H_2(M,\del M)$ and mention some of its properties. See \cite{thurston1986norm} for more details and proofs.
Let $S$ be an oriented surface that is properly embedded in $M$, and define
\[
\chi_-(S)=-\chi \left(S-(\text{sphere, disk, torus, and annulus components})\right).
\]
For an integral homology class $\alpha\in H_2(M, \del M)$, define 
\[
x(\alpha)=\min_{[S]=\alpha}\{\chi_-(S)\},
\]
where the minimum is taken over all oriented properly embedded surfaces representing $\alpha$.

Thurston showed that this function on the integral lattice of $H_2(M,\del M)$ extends to a continuous convex function $H_2(M,\del M)\to \R$ that is linear on rays through the origin, and vanishes precisely on the subspace spanned by the classes of properly embedded surfaces with nonnegative Euler characteristic. If there are no such surfaces representing nonzero homology classes, then $x$ is a vector space norm on $H_2(M,\del M)$; otherwise $x$ is merely a pseudonorm. In a slight abuse of terminology, $x$ is known as the \textbf{Thurston norm} regardless of whether it is a \emph{bona fide} norm. We often think of $x$ as a function on $H^1(M)$ via Poincar\'e/Lefschetz duality.

The \textbf{dual Thurston norm} $x^*\colon H^2(M,\del M)\to \R\cup{\infty}$ is defined by 
\[
x^*(u)=\sup\{ u\mid x(u)\le 1\}.
\]
Regardless of whether $x$ is a norm, the function $x^*$ restricts to a norm on the subspace on which it takes finite values.

\subsection{The Euler class}\label{sec:eulerclass}

Let $\gamma$ be the homology class of the union of the oriented cores of the solid torus components of $U$. Define $e_\tau\in H^2(M, \del M)$ by the formula
\[
2e_\tau=\langle 2\gamma- \Gamma, \cdot \rangle.
\]
Here $\langle \cdot,\cdot \rangle$ is the intersection pairing between $H_1(M)$ and $H_2(M,\del M)$. We call $e_\tau$ the \textbf{Euler class} of $\tau$.

We now explain that this generalizes previous definitions in the literature of ``Euler classes" for veering triangulations.
If $U$ has no solid torus components then the formula reduces to $2e_\tau=-\langle \Gamma, \cdot\rangle$, which is the formula for combinatorial Euler class given in \cite{LMT20}, based on Lackenby's definition for arbitrary taut ideal triangulations from \cite{lackenby2000taut}.

If $T$ is a cusped solid torus of $B^s$, let $\beta_T$ be the union of cusp curves of $T$. Then $\beta_T$ is homologous to the 1-chain $\prongs(T)\cdot \core(T)$, which can be seen by homotoping $\beta_\tau$ onto $\core(T)$.
The Euler class defined in \cite{Landry_norm} is an element of the first homology of a closed manifold defined as $\frac{1}{2}\sum \ind(T)\cdot [\core(T)]$, where the sum is taken over all cusped solid tori of $B^s$ (all cusped tori are solid here because the manifold is closed). We have
\begin{align*}
\sum_{} \ind(T)\cdot [\core(T)]
&=\sum(2-\prongs(T))\cdot [\core(T)]\\
&=\sum(2[\core(T)]-[\beta_T])\\
&=2\gamma-\Gamma,
\end{align*}
where in the last line we have used that $\Gamma$ is homologous to the union of $\beta_T$ over all tubes $T$. Therefore in the case when $B^s$ has no cusped torus shells, our Euler class is Poincar\'e dual to the Euler class from \cite{Landry_norm}. If cusped torus shells are present, we can still apply the above analysis to each cusped solid torus. This shows the following.

\begin{lemma}\label{lem:partialeulerclassformula}
Let $c$ be the union of all 
cusp curves of cusped torus shells of $B^s$. Then
\[
2e_\tau=\left\langle\left( \sum \ind(T_i)\cdot\core(T_i)\right)-c,\cdot\right\rangle,
\]
where the sum is over all cusped solid tori $T_i$ of $B^s$.
\end{lemma}

The next two lemmas will be key tools for connecting veering triangulations to the Thurston norm. They are analogous to \cite[Lemma 5.3, Lemma 5.4]{Landry_norm} and the proofs are adaptations of the proofs of those lemmas to the more general setting here. Before we state and prove them, we make a few definitions and set some notation.

Let $S$ be a surface properly embedded in $M$ transverse to $B^s$ such that each patch $p$ of the train track $t^s:=S\cap B^s$ is $\pi_1$-injective in its cusped torus, which we denote by $T(p)$. 
If $p$ is a patch such that $T(p)\cut p$ is disconnected, or if $p\cap \del M$ is nonempty and has null pairing with the cusp curves or $T$, then we say that $p$ is \textbf{superfluous}. Note that if $p$ is superfluous, then $p$ does not contribute to the algebraic intersection of $S$ with $\Gamma$ or $\gamma$, hence does not contribute to the pairing of $S$ with $e_\tau$.

Let $\text{NSD}$ and $\text{NSA}$ denote the sets of nonsuperfluous disks and annuli, respectively. If $p$ is a nonsuperfluous disk then $T(p)$ is a cusped solid torus and $p$ is a meridional disk of $T(p)$. Moreover $p$ has pairing $+1$ or $-1$ with the core of $T(p)$; we denote this pairing by $\sign(p)$. For any patch $p$, we say a cusp of $p$ is {positive} or {negative} according to whether the corresponding point of $S\cap \Gamma$ is positive or negative, respectively.
Define $\algcusps(p)$ to be the number of positive cusps of $p$ minus the number of negative cusps of $p$. 
Note that $\algcusps(p)$ is always less than or equal to $\cusps(p)$, the number of cusps of $p$.

Having set all this notation, \Cref{lem:partialeulerclassformula} now gives
\begin{align}\label{euler_patch}
2e_\tau([S])=\sum_{\text{NSD}}\sign(p)\cdot \ind(T(p)) +\sum_{\text{NSA}}(-\algcusps(p)).
\end{align}

Recall that a properly embedded surface $S$ in $M$ is \define{taut} if $x([S])= \chi_-(S)$ and no collection of its components is nullhomologous.

\begin{lemma}\label{lem:dualnorm}
The dual Thurston norm of $e_\tau$ is $\le 1$. 

As a consequence, if $Y$ is a surface such that $\chi_-(Y)=-e_\tau([Y])$, then 
\begin{enumerate}[label=(\alph*)]
\item $x([Y])=-e_\tau([Y])$ and
\item if no collection of components of $Y$ is nullhomologous, then $Y$ is taut.
\end{enumerate}
\end{lemma}

\begin{proof}
Let $S$ be a taut surface in $M$. 
We can isotope $S$ so that each patch $\pi_1$-injects into its cusped torus, and so that, by \Cref{lem:noposindex} combined with the fact that $S$ is not a disk homotopic into $\del M$, there are no patches of $S\cap B^s$ with positive index.

We then have the following equalities and inequalities (justifications are below):
\begin{align*}
2e_\tau([S])&= \sum_{\text{NSD}}\sign(p)\cdot \ind(T(p)) +\sum_{\text{NSA}}(-\algcusps(p))\\
&\ge \sum_{\text{NSD}}\ind(T(p)) +\sum_{\text{NSA}}-\cusps(p)\\ 
&\ge \sum_{\text{NSD}}\ind(p) +\sum_{\text{NSA}}\ind(p)\\
&\ge\sum_{\text{all patches}} \ind(p)\\
&=2\chi(S)
\end{align*}
The first line is a restatement of \Cref{euler_patch}.
The second uses that all cusped solid tori have index $\le 0$. 
The third line uses that $\ind(p)\le \ind (T(p))$ for $p\in \text{NSD}$ and $\ind(p)=-\cusps(p)$ for $p\in \text{NSA}$.
For the fourth line we have used that all patches of $S\cap B^s$ have nonpositive index. Hence $\chi(S)\le e_\tau([S])$, so 
\[
-e_\tau([S])\le -\chi(S)=x([S]). 
\]
This shows that $e_\tau$ has dual Thurston norm $\le 1$ as claimed.

Now suppose that $Y$ is as in the statement of the lemma. Then $x([Y])\le\chi_-(Y)=-e_\tau([Y])\le x([Y])$ so all these quantities are equal. In particular (a) holds, and if no collection of components of $Y$ is nullhomologous then $Y$ is taut.
\end{proof}

The next lemma is a generalization of \cite[Lemma 5.4]{Landry_norm}.

\begin{lemma}\label{lem:carriedtaut}
Let $S$ be a surface properly embedded in $M$ that is relatively carried by a relative veering triangulation $\tau$ of $M$. Then $S$ is taut and $x([S])=-e_\tau([S])$.
\end{lemma}

\begin{proof}

By \cite[Lemma 5.8]{LMT20} the dual graph $\Gamma$ is strongly connected, so every $\Gamma$-edge is part of a directed cycle. Hence each component of $S$ pairs positively with some directed curve in $M$, so no collection of components is nulhomologous. In particular, using \Cref{prop:top}, we have $\chi_-(S)=-\chi(S)$.

Because $S$ is relatively carried by $\tau$, each patch $p$ of $S\cap B^s$ is either 
\begin{enumerate}[label=(\alph*)]
\item a meridional disk of a cusped solid torus,
\item a topological annulus with one boundary component on $\del M$ and another boundary component having a nonzero number of positive cusps, or
\item a topological annulus with no cusps in its boundary.
\end{enumerate}
In case (a), $\ind(p)=2-\cusps(p)$. In case (b), $\ind(p)=-\cusps(p)$. In case (c), $\ind(p)=0$. In each case $\ind(p)$ is precisely the contribution of $p$ to $2e_\tau([S])$, so we have $\chi_-(S)=-e_\tau([S])$. \Cref{lem:dualnorm} finishes the proof. 
\end{proof}

\subsection{Combinatorial transverse surface theorems}

We define $\cone_1(\Gamma)$ to be the smallest closed convex cone in $H_1(M)$ containing the homology class of each directed cycle in the dual graph $\Gamma$. Its dual cone $\cone_1^\vee(\Gamma)$ is the cone of all classes in $H^1(M)$ that pair nonnegatively with each element of $\cone_1(\Gamma)$.

The following is the generalization of \cite[Theorem 5.5]{Landry_norm} to our setting.

\begin{theorem}[Combinatorial TST]
Let $\tau$ be a veering triangulation of $M$ relative to a tube system $U$. Let $u$ be an integral cohomology class in $H^1(M)$. Then $u\in \cone_1^\vee(\Gamma)$ if and only if there exists a surface $S$, necessarily taut, such that $S$ is relatively carried by $\tau$ and $[S]$ is Poincar\'e dual to $u$. 
\end{theorem}

\begin{proof}
The proof of \cite[Theorem 5.5]{Landry_norm} goes through almost without modification, but we describe it here in broad strokes. 

If $u$ is dual to a carried surface, then that surface is taut by \Cref{lem:carriedtaut} and it is clear that $u\in \cone_1^\vee(\Gamma)$ since the surface is positively transverse to $\Gamma$.

Conversely, let $u\in \cone_1^\vee(\Gamma)$, and let $u_U$ be the pullback to $H^1(M_U)$. We can use $u_U$ to produce a surface $S_U$ in $M_U$ that is carried by $\tau_U$ and dual to $u_U$ \cite[Proposition 5.11]{LMT20}. Let $M_0$ be the union of $ M_U$ with the solid tubes. The surface $S_U$ can be capped off by annuli and disks inside the solid tubes to produce a surface dual to $u_0$, the pullback of $u$ to $M_0$. This is essentially because $u\in H^1(M)$ (see \cite[Lemma 3.3]{landry2018taut}; that proof goes through here, guaranteeing that capping off is possible).

The only mild novelty in the setting of this theorem is the possibility that $M_0\ne M$; that is, our tube system may have hollow tubes. At a component of $\del M_U$ to which a hollow tube is glued, we simply extend $S_U$ to $\del M$ by gluing in annuli using the $T^2\times I$ product structure of the hollow tubes.
The result is a relatively carried surface $S$ that is taut by \Cref{lem:carriedtaut}. Since $S_U$ is dual to $u_U$, $S$ is dual to a class mapping to $u_U$ under $H^1 (M)\to H^1(M_U)$. Since this map is injective, $S$ is Poincar\'e dual to $u$.
\end{proof}

\begin{corollary}
If $\tau$ is a relative veering triangulation in $M$, then the dual cone $\cone_1^\vee(\Gamma)$ is contained in the cone in $H^1(M)$ on which $-e_\tau$ and $x$ agree as maps $H^1(M)\to \R$.
\end{corollary}

\begin{corollary}\label{cor:easycontainment2}
Let $S$ be a taut surface with $[S]\in \cone_1^\vee(\Gamma)$. Then $x([S])=\chi_-(S)=-e_\tau([S])$. 
\end{corollary}

We now aim to prove a version of the main result in \cite{Landry_norm} when $\tau$ is strict. We will use the following, which is a stronger version of \cite[Theorem 8.1]{Landry_norm}:

\begin{lemma}\label{lem:export}
Assume that $\tau$ is a strict relative veering triangulation of $M$ and
let $S$ be an incompressible and boundary incompressible surface in $M$. Further suppose that $S$ has the property that for any surface $S'$ isotopic to $S$ that is transverse to $B^u$ and $B^s$, either: 
\begin{itemize}
\item one of $S' \cap B^u$ or $S' \cap B^s$ has a patch of positive index, or
\item every negative switch of $S' \cap B^u$ and every negative switch of $S' \cap B^s$ belongs to a bigon.
\end{itemize}
Then $S$ is isotopic to a surface relatively carried by $\tau$.
\end{lemma}

\begin{remark}\label{rmk:export}
\Cref{lem:export} is stated in greater generality than originally appeared in \cite{Landry_norm}. There it was assumed that all tubes of $U$ are hollow, so that $\tau$ is automatically strict.
However, the proof given in \cite{Landry_norm} goes through as written, modulo slight differences in terminology. 
We note the additional fact that the individual ``moves" comprising the isotopy in \Cref{lem:export} can be performed some distance away from $\del M$, so the isotopy can be performed fixing $\del S$ pointwise.
\end{remark}

\begin{theorem}[Combinatorial Strong TST]\label{combiSTST}
Let $\tau$ be a strict relative veering triangulation of $M$. Let $S$ be a taut surface in $M$. The following are equivalent:
\begin{enumerate}[label=(\alph*)]
\item $[S]\in \cone_1^\vee(\Gamma)$
\item $x([S])=-e_\tau([S])$
\item $S$ is relatively carried by $\tau$ up to an isotopy, which can be chosen to fix $\del S$.
\end{enumerate}
\end{theorem}

\begin{proof}
Condition (a) implies (b) by \Cref{cor:easycontainment2}. Also (c) implies (a) because any carried surface has only positive intersections with $\Gamma$.

It remains to prove that (b) implies (c), so suppose that $x([S])=-e_\tau([S])$. In light of \Cref{lem:export}, our strategy will be to show that for any surface isotopic to $S$, one of the two conditions in \Cref{lem:export} holds. The proof is analogous to the proof of \cite[Theorem 5.12]{LMT20}, but with our slightly different definition of the Euler class. We present it for the sake of completeness.

By an isotopy, we can assume that $S$ is transverse to both $B^s$ and $B^u$ and that each patch of $t^s=B^s\cap S$ and $t^u=B^u\cap S$ is $\pi_1$-injective in its cusped torus. We will show that if $t^s$ has no patches of positive index, then every negative switch of $t^s$ belongs to a bigon.

If $p$ is a superfluous patch, note that $p$ has $\algcusps(p)=0$ and the number of intersections of $p$ with $\gamma$ counted with sign is 0.

By (b) and the fact that $S$ is taut, we have $2\chi(S)=2e_\tau([S])$. By \Cref{euler_patch}
we can write this as
\begin{align*}
2\chi(S)&= \sum_{\text{NSD}} \sign(p)\cdot\ind(T(p)) +\sum_{\text{NSA}} (-\algcusps(p)).
\end{align*}

If we let SP denote the set of superfluous patches, the index sum formula for $\chi(S)$ gives
\[
2\chi(S)=\sum_{\text{all patches}} \ind(p)=\sum_{\text{SP}} \ind(p) +\sum_{\text{NSD}} \ind(p)+\sum_{\text{NSA}}\ind(p).
\]
Subtracting the above two expressions for $2\chi(S)$, we obtain
\[
0=\sum_{\text{SP}} \ind(p) +\sum_{\text{NSD}}(\ind(p)-\sign(p)\cdot\ind(T(p)))+\sum_{\text{NSA}}(\ind(p)+\algcusps(p)).
\]
We claim every term in each of the three sums above is nonpositive:
\begin{itemize}
\item First sum: every term is nonpositive because we are assuming there are no patches of positive index. 
\item Second sum: if $p$ is a nonsuperfluous disk with $\sign(p)=-1$ then clearly $\ind(p)-\sign(p)\cdot\ind(T(p))\le 0$ is a sum of two nonpositive numbers, and if $\sign(p)=1$ then the inequality follows from the fact that a meridional disk of a cusped solid torus has index at most the index of the cusped solid torus.
\item Third sum: if $p$ a nonsuperfluous annulus then $\chi(p)\le0$. Hence $\ind(p)+\algcusps(p)\le \ind(p)+\cusps(p)=2\chi(p)\le 0$, so every term in the second sum is nonpositive.
\end{itemize}

Therefore each term in the sums over SP, NSD, and NSA is equal to 0. This has the following implications:
\begin{itemize}
\item each superfluous patch has index 0, so is either a bigon or an annulus with no cusps,
\item each nonsuperfluous disk patch $p$ has $\ind(p)=\sign(p)\cdot\ind(T(p))$, so $\sign(p)=+1$ and all cusps of $p$ are positive.
\item each nonsuperfluous annulus patch $p$ in a cusped torus shell has $\ind(p)=-\algcusps(p)$, so $p$ is equal to an annulus with no negative cusps,
\end{itemize}
We have shown that any negative switch of $t^s$ must belong to a bigon. A symmetric argument shows the same for $t^u$. Hence (b) implies (c).
\end{proof}

\begin{corollary}\label{cor:combinatorialflowsrepresentfaces}
If $\tau$ is a strict relative veering triangulation in $M$, then $\cone_1^\vee(\Gamma)$ is equal to the cone over a face of the the Thurston norm ball whose codimension is equal to the dimension of the largest linear subspace contained in $\cone_1(\Gamma)$. Moreover, $\cone_1^\vee(\Gamma)$ is the maximal domain on which $x$ and $-e_\tau$ agree as functions $H^2(M,\del M)\to \R$.
\end{corollary}

\section{Pseudo-Anosov flows and dynamic blowups}
\label{sec:flows}

We refer to Mosher \cite{Mos92}, Fenley--Mosher \cite{fenley2001quasigeodesic} and Agol--Tsang \cite{AgolTsang} for thorough discussions of the definition of a
pseudo-Anosov flow on a closed 3-manifold. There are in fact two
definitions, a smooth and a topological one, and their equivalence for transitive flows has
recently been shown by Agol--Tsang \cite{AgolTsang} relying on the Anosov-case established by Shannon \cite{Shannon20}.
It will suffice for our purposes to record some of the main properties associated with (either
type of) pseudo-Anosov flow, which we will do below.

We will then define the notion of a \define{dynamic blowup} 
of such a flow, which is a
generalization of the definition given by Mosher \cite{mosher1990correction}. A special
case of this will yield a definition---basically the expected one---of pseudo-Anosov flows on manifolds with toral
boundary, and the general case will give us the definition of {\em almost} pseudo-Anosov
flows, in both the closed case and the case with boundary.

\smallskip

A \define{pseudo-Anosov flow} $\varphi$ on a closed 3-manifold $M$ has in particular the
following properties: 
\begin{enumerate}
  \item The flow $\varphi$ has a finite collection $\Pi_s$
  of closed orbits (called
    \define{singular orbits}), and a $\varphi$-invariant pair of {\em transverse singular
      foliations} of codimension 1 whose leaves intersect along orbits of $\varphi$. The foliations are nonsingular and transverse in the complement of $\Pi_s$ and have a standard form in a neighborhood of $\Pi_s$, as described below.
  \item The two foliations are called the \define{stable} and the \define{unstable} foliations
    of $\varphi$, denote $W^s$ and $W^u$, respectively. 
    For any two points $x,y$ in a single leaf of the stable foliation $W^s$, the
    trajectories are asymptotic, meaning that for some orientation-preserving
    $s\colon \R_+\to\R_+$, $\lim_{t\to+\infty}d(\varphi_t(x),\varphi_{s(t)}(y)) = 0$. The
    corresponding statement holds for the unstable foliation $W^u$ with $t\to \infty$ replaced by $t \to -\infty$. 
\end{enumerate}

We will typically work with {\em transitive} flows (i.e. possessing a dense
orbit). Mosher \cite[Proposition 2.7]{Mos92} proved that all pseudo-Anosov
flows on closed atoroidal manifolds are transitive. Once we define pseudo-Anosov flows on manifolds with boundary, it will be easy to see that atoroidality implies transitivity in the broader setting, by choosing appropriate Dehn fillings and applying Mosher's result in the closed case.

\smallskip
It remains to describe 
the local dynamics around a periodic orbit $\gamma$ in
$\Pi_s$,  as a suspension of a
surface homeomorphism.
In what follows, we err on the laborious side in order to prepare for the
description of the blowups. 
Our description is equivalent to the \emph{pseudo-hyperbolic orbits} described in e.g. \cite{fenley2001quasigeodesic, AgolTsang}. 

Let $U\subset V$ be nested open 2-disks, and $f\colon U\to V$ an orientation-preserving embedding, with the following
properties: There is a one-vertex tree $T$ with an even number of {\em prongs} (at least 4) properly embedded
in $V$ so that $f^{-1}(T) \subset T\intersect U$ and $f$ fixes the vertex $p$ and no other
points. 
Each component $Q$ of $V\ssm T$ is an {\em open quadrant} and its closure  in $V$,
$\ol Q$, is
a {\em closed quadrant.} The map $f$ permutes the germs of quadrants in the sense that it takes $Q\intersect U$ into
some $Q'$. Thus there is a least power $f^k$ that takes germs of quadrants into themselves (that
is, there is some smaller neighborhood $U'\subset U$ such that $f^k(Q\intersect U') \subset
Q$). 

Finally, there is an identification of $\ol Q$ with a neighborhood of the
origin in a closed quadrant of the Euclidean plane, which conjugates $f^k$ to the map $(x,y)
\mapsto (\lambda x,y/\lambda)$ for some $\lambda> 1$ (on the image of $\ol Q
\intersect U$).

In particular note that along  the interval $\boundary Q \intersect T$ the map $f$ fixes
the point $p$ and translates the rest in the {\em same} direction as viewed from inside $Q$.
Hence, each prong at $p$ is either \emph{stable} (i.e. contracting) or \emph{unstable} (i.e. expanding). Note that our convention is to draw the stable direction as vertical and the unstable direction as horizontal.

We suspend this picture by considering $V\times [0,1]$, with $U\times 1$ identified with
$f(U)\times 0$ using $f$. Vertical flow is then defined, for bounded times, on a small neighborhood of the
closed orbit $p\times[0,1]/\sim$. This is the model for a pseudo-Anosov flow in a
neighborhood of a closed orbit $\gamma$; the orbit is \define{singular} (i.e. $\gamma \in \Pi_s$) if the number of stable (and unstable) prongs is at least $3$. 
The suspended vertical lines of the quadrants are the intersections of leaves of the
stable foliation with the neighborhood, and the suspended horizontal lines are
intersections of leaves of the unstable foliation. Note that prongs thus alternate between
stable and unstable singular leaves. 
Finally, we define the \define{index} of $\gamma$ to be $2$ minus the number of stable prongs.

From this description one sees that each singular periodic orbit is contained in a singular leaf of
the stable foliation, which has the structure of 3 or more annuli glued together along the
orbit (this is true at least locally and it is not hard to see that it is globally
true.) The orbit is the attractor of all flow lines in its stable leaf. The corresponding
statement is true for the unstable singular leaf.

\subsection{Local blowups}
\label{local blowups}

Let $U^\sharp\subset V^\sharp$ be either nested disks as before, or nested annuli
homeomorphic to $S^1\times[0,1)$ sharing the boundary component $\boundary U^\sharp =
  \boundary V^\sharp = S^1\times 0$. An embedding $f^\sharp \colon U^\sharp \to V^\sharp$ is a
 \define{blowup} of $f$ if the following holds:

There is a degree 1 map $\pi \colon V^\sharp\to V$ taking $U^\sharp $ to $U$, which
semiconjugates $f^\sharp$ to $f$. 
The preimage $\pi^{-1}(T)$ is a graph $T^\sharp$ (without valence $2$ vertices) properly embedded in $V^\sharp$, containing
a connected subgraph $G=\pi^{-1}(p)$ in $U^\sharp$. In the \emph{disk case}
 $T^\sharp$ is a tree, and in the
\emph{annulus case} $T^\sharp$ has Euler characteristic $0$, containing $\boundary U^\sharp\subset G$ as its
unique circle each of whose vertices we require to have degree $3$. 
We refer the reader to the right-hand sides of \Cref{fig:arcs-internal} and \Cref{fig:arcs-boundary} 
for examples in the disk and annulus case, respectively.
The complement of $G$ in $T^\sharp$ is a union
of intervals, which correspond to the edges of $T$. 
The restriction of $\pi$ to
$V^\sharp \ssm G$ is a homeomorphism to $V\ssm p$, which conjugates (the
restrictions of) $f^\sharp$ to $f$.

We require that the periodic points of $f^\sharp$ are the vertices of $T^\sharp$, so on the interior of each edge some
power of $f^\sharp$ acts by translation. We again have open quadrants, the components of
$V^\sharp \ssm T^\sharp$, which are taken homeomorphically to open quadrants of
$V\ssm T$. The closure of each quadrant now has a boundary which is a segment in
$T^\sharp$ composed of several edges. Note that the translations of (a suitable
power of) $f^\sharp$ on these edges are all in the same direction, as viewed from the
quadrant. We call such a picture a {\em consistently oriented} blowup quadrant. One implication is that around any vertex $v$ of $T^\sharp$ the adjacent prongs alternate between attracting and repelling at $v$.

In the simplest blowup of a quadrant map, the corner point pulls back to an
interval connecting two fixed points. 
One can easily construct such a blowup, in coordinates where $f$ is
differentiable at the corner with eigenvalues $\lambda > 1 > \mu > 0$, using a ``real
blowup'' where the corner is replaced by the interval of unit tangent vectors pointing
into the quadrant. See \Cref{fig:simple-blowup}, and also the appendix of \cite{cantwell1999isotopies}.

\begin{figure}[htbp]
\begin{center}
\includegraphics{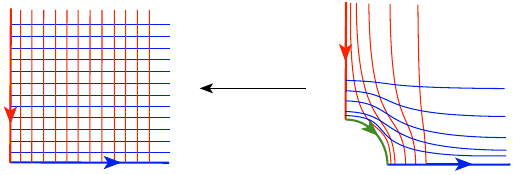}
\caption{Standard blowup of a quadrant.}
\label{fig:simple-blowup}
\end{center}
\end{figure}

A blowup with more fixed points (which still satisfies the consistent orientation
condition) can then be obtained by repeating the construction at the new fixed points; see
\Cref{fig:repeated-blowup}.

\begin{figure}[htbp]
\begin{center}
\includegraphics[height=1.5in]{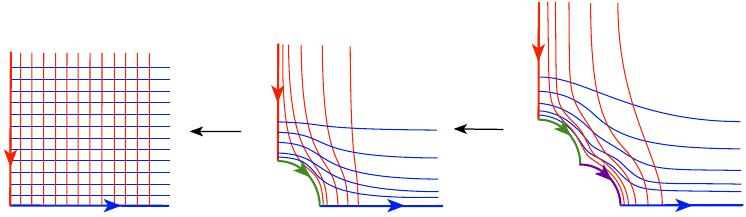}
\caption{A repeated blowup has additional fixed points in the blowup interval}
\label{fig:repeated-blowup}
\end{center}
\end{figure}

We suspend $f^\sharp$ in the same way as before, obtaining a flow defined for bounded
times in a
neighborhood of the suspension of $G$. This is 
what we call a blowup of the
suspension of $f$. Note that in the annulus case we obtain a manifold with torus
boundary. 

The edges of the graph $G$ suspend to flow-invariant annuli which we call \textbf{blown annuli}. Note that in each such annulus the flow lines are asymptotic to one boundary
component in backward time and the other in forward time. The suspension of $G$ itself is called an \define{annular complex}; it is a union of blown annuli.

\begin{remark}[Nearby homeomorphism] \label{nearby homeo} Note that in the disk case, the map $\pi$ is homotopic to a homeomorphism by a
homotopy supported on a small neighborhood of $\pi^{-1}(p)$, since such a neighborhood is
a disk. In the annulus case the two surfaces are not homeomorphic, but if we had two
annulus-type blowups $\pi_1\colon V_1^\sharp\to V$ and $\pi_2\colon V_2^\sharp\to V$, one similarly
sees that there is a homeomorphism $h \colon V_1^\sharp \to V_2^\sharp$ which is equal to
$\pi_2^{-1}\circ \pi_1$ outside a small neighborhood of $\pi_1^{-1}(p)$.  
Similar statements hold for the suspensions. 
\end{remark}
 
\subsection{Global blowups}
\label{global blowups}
A {\bf dynamic blowup} of a pseudo-Anosov flow $\phi$ on a closed 3-manifold $M$ is a flow
$\phi^\sharp$ on a compact 3-manifold $M^\sharp$, together with a map $\pi \colon M^\sharp \to M$ with the
following properties:
\begin{enumerate}
  \item $\pi$ semiconjugates $\phi^\sharp$ to $\phi$. 
\item There is a finite collection $\Pi$ of closed orbits of $\phi$ such that, from
  $\pi^{-1}(M\ssm \Pi)$ to $M\ssm \Pi$, $\pi$ is a homeomorphism and a conjugacy of
  the flows.
\item For each orbit $\gamma$ in $\Pi$ there is a neighborhood $W$ whose preimage under
  $\pi$ is a local blowup in the sense described above. In particular, $\pi^{-1}(W)$ is either a
  neighborhood of the suspension of a tree homeomorphism (the ``disk case'' of \Cref{local blowups}), or a
  neighborhood of a toral boundary component (the ``annulus case'').

\end{enumerate}
 We call the neighborhood $W$ from item $(3)$ above a \define{standard neighborhood} of $\gamma$.

A blowup $\phi^\sharp$ of a pseudo-Anosov flow is called an \define{almost pseudo-Anosov flow}. 
We say that the \define{combinatorial type} of such a blowup is the blowup graph $T^\sharp$
associated to each local blowup, together with its embedding in the disk or annulus $V^\sharp$,
up to  proper isotopy. 
Note that the combinatorial type is completely specified by saying
which pairs of  quadrants of the original return map, in a neighborhood of each orbit, are to be
connected along the blowup graph. We do {\em not} claim that, given the combinatorial
type,  the blowup flow $\phi^\sharp$ is unique up to orbit equivalence. Indeed there are
further choices in the construction (notably the choice of gluing map across the blowup
graph), but these will not matter to our discussion.

\smallskip

There are several important special cases of dynamic blowups. If all the local blowups are in the disk case, then
$M^\sharp$ is still a closed manifold and in fact, using the discussion in \Cref{local
  blowups}, there is a homotopy of $\pi$ to a homeomorphism from $M^\sharp$ to $M$,
supported on a small neighborhood of $\pi^{-1}(\Pi)$. This is exactly the case defined
and used by Mosher \cite{mosher1990correction}.

If a local blowup is in the annulus case then the subgraph $G$ is a circle each of whose vertices has degree $3$ in $T^\sharp$ 
 (i.e. each singularity is $3$ pronged),
so we have blown up a closed orbit to a boundary component in the simplest possible
way. We call this a \define{Fried blowup} (see \Cref{fig:blowup2}) and it gives us a clean way to
define a pseudo-Anosov flow on a manifold with boundary:

  If $M$ is a compact 3-manifold with toral boundary and $\varphi$ is a flow that is
  tangent to the boundary, we say $\varphi$ is \define{pseudo-Anosov} if there is a Dehn filling
  $\pi\colon M \to M^\flat$ such that $\varphi$ is the blowup of a pseudo-Anosov flow on $M^\flat$, for
  which all the local blowups are Fried blowups.

\begin{figure}[htbp]
\begin{center}
\includegraphics{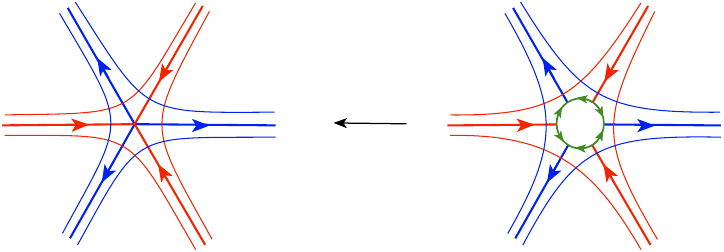}
\caption{The Fried blowup of a singular point yields the dynamics around a boundary component.}
\label{fig:blowup2}
\end{center}
\end{figure}

By construction, each boundary component $P$ of $\partial M$ is a union of (consistently oriented) blown annuli and so the restriction of $\phi$ to $P$ has alternating expanding/contracting periodic orbits. The contracting periodic orbits in $\partial M$ are called \define{unstable prong curves} and the expanding periodic orbits are called \define{stable prong curves}. 
The associated slopes on $\partial M$ are called the \define{prong slopes}; these are also called the degeneracy slopes in the literature.

\begin{remark}[Singular orbits and blowups]
\label{rmk:no_blow}
By definition only singular orbits of $\phi$ can be blown up, except for the case of a Fried blowup at a regular orbit. Hence, when $\pi \colon M^\sharp \to M$ restricts to a homeomorphism on boundary components (or the manifolds are closed), $\Pi \subset \Pi_s$. 
\end{remark}

\subsection{Blowing down and up}
\label{rmk:blowup_on_M}
We now discuss a more general blowup construction with the goal of producing a modified flow on the \emph{same} manifold $M$ in the case where $M$ has nonempty boundary. 
For some simple examples, see the bottom row of the bottom chart in \Cref{fig:shadows} and ignore the green surface; see also \Cref{fig:arcs-internal} and \Cref{fig:arcs-boundary}.

Starting with a manifold with boundary and a pseudo-Anosov flow, we will first ``blow down'' to a closed $M^\flat$
so that $M\to M^\flat$ is a Fried blowup, and then blow up $M^\flat$ to obtain $M^\sharp$
with general annulus type blowups that recover all the boundary components, so that $M^\sharp$
and $M$ are homeomorphic. As in \Cref{nearby homeo}, we obtain a homeomorphism $h \colon M^\sharp \to M$ and an almost pseudo-Anosov flow $\varphi^\sharp$ on $M^\sharp$ so that $h$ pushes $\phi^\sharp$ to a flow on $M$ which, away from
a small neighborhood of blown singular orbits and boundary components of $M$, agrees with
$\phi$. We refer to this combined operation as a
\textbf{generalized dynamic blowup}; however for brevity we will usually just use the term dynamic blowup.
Moving forward, we will use $h$ to  {\em identify} $M^\sharp$ with $M$ and consider $\phi^\sharp$ 
as a small modification of $\phi$ on $M$. With this change, the above setup gives

\begin{equation}\label{eqn:blowupdown}
  \begin{tikzcd}
(M,\phi^\sharp)\arrow{rd}{\pi_\sharp}& & (M,\phi)\arrow{ld}[swap]{\pi_\flat}\\
&(M^\flat,\phi^\flat)&\\
  \end{tikzcd}
\end{equation}
where $\pi_\sharp$ and $\pi_\flat$ are the blowdown maps. Moreover,  
we have a homeomorphism
\[
g \colon M \ssm \left( \text{blowup locus of }\pi_\sharp  \right) \to M \ssm \left(\text{blowup locus of }\pi_\flat \right)
\]
which is an orbit equivalence from the restriction of $\phi^\sharp$ to that of $\phi$, and
equal to the \emph{identity} 
away from a small neighborhood of the blowup loci.

\begin{remark}[Goodman--Fried surgery] \label{remark:FGsurg}
We observe that for any pseudo-Anosov flow $\varphi$ on $M$ and any finite collection of orbits $\kappa$, we can apply a Fried blowup to each of the orbits of $\kappa$ to obtain a manifold $\mr M$ with pseudo-Anosov flow $\mr \phi$. 

For each component $k$ of $\kappa$, let $s_k$ be an essential simple closed curve on the associated component $P_k$ of $\partial \mr M$ with the property that its geometric intersection number with the union of unstable prong curves of $P_k$ is at least $2$. Then Fried shows that there is a blowdown $\mr M \to M'$ to a pseudo-Anosov flow $\phi'$ on the Dehn filled manifold $M' = \mr M(s_k: k \in \kappa)$ that collapses $P_k$ to a closed orbit $k'$ whose index is the above mentioned geometric intersection number. Of course, $\mr M$ is also a Fried blowup of $M'$, illustrating the flexibility in the definition of a pseudo-Anosov flow on a manifold with boundary.

See Fried \cite{fried1982geometry} for the details of the construction and Tsang \cite{tsang2022constructing} for the related history.
\end{remark}

\subsection{Stable/unstable foliations of almost pseudo-Anosov flows}
Let $\phi$ be an almost pseudo-Anosov flow on a compact 3-manifold $M$, obtained by
dynamically blowing up a pseudo-Anosov flow $\phi^\flat$ on a closed 3-manifold $M^\flat$. The
\textbf{stable foliation} of $\phi$ is a singular foliation whose leaves are the
preimages of leaves of the stable foliation of $\phi^\flat$ 
under the semiconjugacy $(M, \phi)\to (M^\flat, \phi^\flat)$. 
It is not hard to see that this definition is independent of which blowdown
$(M^\flat,\phi^\flat)$ we choose. 
We similarly define the
\textbf{unstable foliation} of $\phi$. One sees from the local structure of a
blowup that the preimage of a leaf is is connected, so
there is a one-to-one correspondence between leaves of the stable and unstable foliations
of $\phi^\flat$ and $\phi$. If $H$ is a stable leaf of $\phi^\flat$ containing a singular orbit
$\gamma$ of $\phi^\flat$ which is blown up to obtain $\phi$, then the stable leaf of
$\phi$ corresponding to $H$ will contain every blown annulus collapsing to
$\gamma$. Note that the blown annuli simultaneously belong to both stable and unstable leaves
  of the foliations of $\phi$, and in particular the foliations are not transverse
  even away from the periodic orbits.

We say two orbits $\gamma$ and $\gamma'$ of a flow on 
a manifold $X$ with metric $d$ are \textbf{asynchronously proximal} (in the forward
direction) 
if there are sequences of
times $(t_n), (t_n')$ with $t_n, t_n' \to \infty$ such that $d(\gamma(t_n),\gamma'(t_n'))\to 0$. Note that this is a weaker notion than forward asymptotic: asynchronously proximal orbits get arbitrarily close, while asymptotic orbits get arbitrarily close and \emph{stay} arbitrarily close. 

The definition of asynchronously proximal will be relevant to us when we lift an almost pseudo-Anosov flow on a compact 3-manifold $M$ to its universal cover. 

For the next lemma, a stable \define{half-leaf} is a complementary component of the closed orbits within a stable leaf of $\varphi$.
An unstable half-leaf is defined similarly. Note that a stable/unstable leaf can contain $k$ half-leaves, where $k$ is any integer $k\ge 1$. The following is a key feature of almost pseudo-Anosov flows.

\begin{lemma}\label{lem:proximal}
Let $\phi$ be an almost pseudo-Anosov flow on a compact 3-manifold $M$, and let $\wt\phi$ be its lift to the universal cover $\wt M$. If $\gamma$ and $\gamma'$ lie in the same stable half-leaf 
of $\wt\phi$, then $\gamma$ and $\gamma'$ are asynchronously proximal. If they lie in the same unstable half-leaf, then they are asynchronously proximal in the backward direction.
\end{lemma}

In the case where $M$ is closed, this lemma is an observation of Fenley in \cite[Proof of Theorem 4.1]{Fen09} although he does not use the term asynchronously proximal. The key point there and in this more general setting is that a semiconjugacy collapsing blown annuli can be taken to be an isometry away from a neighborhood of the blown annuli.

\subsection{Flow spaces and  perfect fits}
\label{sec:flowspace}
For a flow $\phi$ on a manifold $M$, its \define{flow space} $\orb$ is obtained by taking the quotient of the universal cover $\wt M$ of $M$ by the orbits of the lifted flow $\wt \phi$. The quotient map $\Theta \colon \wt M \to \orb$ is equivariant with respect to the $\pi_1(M)$--actions, and when $\phi$ is almost pseudo-Anosov, its stable/unstable invariant foliations project to invariant singular foliations of $\orb$ which we continue to call the \define{stable/unstable} foliations of $\mc O$.  

A \define{blown segment} in $\orb$ is a topological arc obtained by lifting a blown annulus to $\wt M$ and projecting to $\mc O$; it is contained in leaves of both foliations. 

\begin{theorem}\label{th:flow_space}
Let $\phi$ be an almost pseudo-Anosov flow on a compact manifold $M$ and let $\orb$ be its flow space. Then $\orb$ is a simply connected surface that is (singularly) foliated by the images of the stable/unstable foliations, which are transverse away from blown segments.
If $M$ is closed, $\orb$ is homeomorphic to $\mathbb{R}^2$. Otherwise, $\orb$ is homeomorphic to a disk minus a closed subset of its boundary.
\end{theorem}

In the case where $\phi$ is pseudo-Anosov and $M$ is closed, \Cref{th:flow_space} was proven by Fenley--Mosher \cite[Proposition 4.1]{fenley2001quasigeodesic}. They also establish the result when $\phi$ is almost pseudo-Anosov (again when $M$ is closed) assuming that $\phi$ is transverse to a taut foliation. 

We begin with some setup with the goal of showing that the general result follows from the special case proven by Fenley--Mosher. By definition (see \Cref{global blowups}), there is a quotient map $\pi \colon M \to M^\flat$, that blows $\phi$ down to a pseudo-Anosov flow $\phi^\flat$ on a closed manifold $M^\flat$. Let $\orb^\flat$ be the flow space of $\phi^\flat$; we will use the fact, referenced above, that $\orb^\flat$ is a bifoliated plane. 

\begin{remark}
The blowdown $M \to M^\flat$ is not unique (see \Cref{remark:FGsurg}) but the following discussion is true for any such blowdown.
\end{remark}

The map $\pi \colon M \to M^\flat$ is $\pi_1$--surjective and we let $M' \to M$ be the
covering space determined by $\ker \pi_*$
so that there is an induced quotient map $\pi' \colon M' \to \wt M^\flat$ to the universal cover $\wt M^\flat$ of $M^\flat$.
Denote by $\orb'$ the quotient of $M'$ obtained by collapsing orbits in $M'$ of the lifted flow. Then there is an induced quotient map $\pi'_\orb \colon \orb' \to \orb^\flat$. Since $\pi' \colon M' \to \wt M^\flat$ collapses blown annular complexes (including boundary components) to singular orbits, $\pi'_\orb \colon \orb' \to \orb^\flat$ collapses connected graphs of blown segments to singularities of $\orb^\flat$. 

By construction, the flow space $\orb = \orb_M$ of $\phi$ is the universal cover $q \colon \orb_M \to \orb'$ and the composition $\pi_\orb = \pi'_\orb \circ q \colon \orb_M \to \orb^\flat$ is equivariant with respect to $\pi_*$.

\begin{proof}[Proof of \Cref{th:flow_space}]
 With the description above, it now suffices to show that $\orb'$ is a surface, possibly
 with boundary; that is, locally two-dimensional, and Hausdorff.  The argument in
 Fenley--Mosher \cite[Proposition 4.2]{fenley2001quasigeodesic} goes through verbatim to show that $\orb'$ is Hausdorff.  
 
 Following Fenley--Mosher, to show that $\orb'$ is locally two-dimensional,
 it suffices to show that every point of $M'$ has a transverse (half-)disk neighborhood
 that injects into $\orb'$. This follows from the fact that $M'$ does not contain $(\epsilon, T)$-cycles with $\epsilon$ arbitrarily small and $T$ arbitrarily large. We recall that an $(\epsilon,T)$ cycle is a flow segment $[x,y]$ with $y$ obtained from $x$ by flowing for time $T$ such that the distance between $x$ and $y$ is no more than $\epsilon$. 

To prove the fact, note that such a cycle will project to an $(\epsilon', T')$-cycle in
$M^\flat$ with $\epsilon' \le O(\epsilon)$, since the map $\pi$ is Lipschitz. Moreover, as in
the claim in the proof of \cite[Proposition 4.2]{fenley2001quasigeodesic}, $T'/T$ is uniformly bounded
above and below.  Hence, we obtain an $(\epsilon', T')$-cycle of $\wt M^\flat$ with $\epsilon'$
arbitrarily small and $T'$ arbitrarily large. However, this is shown to be impossible
in the proof of Proposition 4.1 of \cite{fenley2001quasigeodesic}.

This shows that $\orb'$ is a surface. The description of $\orb'$ given above then implies
that it is obtained topologically by removing the interiors of a countable collection
(zero when $\partial M = \emptyset$)  of disjoint disks from the plane.
This implies that its universal cover $\orb_M$ has the required description; the statement about the singular foliations on $\orb_M$ follows from the corresponding statement for the foliations on $M$.
\end{proof}

\subsubsection{Perfect fits and flow space relations}
\label{sec:pfits}\label{prop:tsang}

Fix an almost pseudo-Anosov flow $\varphi$ on $M$ with flow space $\mc O$. A \define{perfect fit rectangle} in $\mc O$ is a proper embedding of the rectangle with missing corner $[0,1]^2 \setminus \{(1,1)\}$ into $\mc O$ that maps the horizontal/vertical foliations to the unstable/stable foliations of $\mc O$. We note that any perfect fit rectangle contains a perfect fit subrectangle that is disjoint from the singularities and blown segments of $\mc O$. 

For any collection of closed orbits $\kappa$, let $\wh \kappa$ be the $\pi_1$--invariant
discrete collection of points in $\mc O$ obtained by lifting $\kappa$ to the universal
cover $\wt M$ and projecting to $\mc O$. Then we say that $\kappa$ \define{kills perfect
  fits} if $\wh \kappa$ meets the interior of every perfect fit rectangle in $\mc O$.
Note that for the definition, it suffices to only ever consider regular orbits since
singular orbits correspond to points in $\mc O$ that do not lie in the interior of any
rectangle. If $\mc O$ has no perfect fit rectangles (i.e. $\kappa = \emptyset$ kills
perfect fits) then we say that $\phi$ \define{has no perfect fits}. The importance of this condition was first recognized and studied by Fenley; see \cite{fenley1998structure, fenley1999foliations}.

Now fix a collection $\kappa$ of closed \emph{regular} orbits of $\phi$, and let $\pi
\colon M \to M^\flat$ be any blowdown to a pseudo-Anosov flow $\phi^\flat$ on a closed
manifold $M^\flat$. 
Since the orbits in $\kappa$ are regular, $\kappa$ is mapped
homeomorphically to its image in $M^\flat$.
We define $\kappa^\flat$ to be this image together with any
regular orbits of $\phi^\flat$ that are the images of components of $\partial M$.
Using the induced flow space map $\pi_\orb \colon \orb \to \orb^\flat$ (discussed 
before the proof of
\Cref{th:flow_space}), it is clear that $\kappa$ kills perfect fits for $\phi$ if and only
if $\kappa^\flat$ kills perfect fits for $\phi^\flat$.

Still with $\kappa$ fixed, we set $\kappa_s$ to be the union of $\kappa$ with the set of all singular orbits and blown annular complexes of $\phi$ (including those containing the boundary), and we let $\wh \kappa_s$ be obtained by lifting to the universal cover and projecting to $\mc O$. The blowdown $\pi$ sends $\kappa_s$ to $\kappa^\flat_s$, where $\kappa^\flat_s$ is the union of $\kappa^\flat$ with all singular orbits of $\phi^\flat$. 
In particular, we may use the restriction of $\pi$ to \emph{identify} the manifolds $M \ssm \kappa_s = M^\flat \ssm \kappa^\flat_s$, which we call
 the \define{fully-punctured manifold} associated to the pair $(\phi, \kappa)$. 
The flow space of the fully-punctured manifold (for the restriction of $\phi$) is denoted $\mr {\mc P}$, and the natural projections $\mr{\mc P} \to \mc O \ssm \wh \kappa_s \to \mc O^\flat \ssm \wh \kappa^\flat_s$ are covering spaces. 

The associated branched cover $\mc P \to \mc O^\flat$, infinitely branched over 
$\wh \kappa^\flat_s \subset \orb^\flat$,
is called the \define{completed flow space} of the fully-punctured manifold. The added branch points are also called completion points or singularities of $\mc P$. Finally, we observe that $\kappa$ fills perfect fits of $\phi$ if and only if $\mc P$ has no perfect fit rectangles. 

\begin{remark}
\label{rmk:blow_to_no}
From the discussion here, we see that $\kappa$ kills perfect fits of $\phi$ if and only if the associated Fried blowup 
has no perfect fits.  
\end{remark}

To conclude, we recall the fact that for every transitive pseudo-Anosov flow $\phi$ there exists a finite collection of regular orbits that kills its perfect fits. This follows from that fact, due to Fried \cite{fried1983transitive} and Brunella \cite{brunella1995surfaces}, that every such flow has a Fried blowup that blows down to pseudo-Anosov suspension flow. From the discussion above, it is then clear that the closed regular orbits of the Fried blowup kill the perfect fits of $\phi$. In fact, Tsang has recently proven that for any transitive pseudo-Anosov flow on a closed $3$-manifold, there is a single regular orbit that kills its perfect fits \cite[Proposition 2.7]{Tsang_geodesic}. Just as before, this is easily extended to all transitive, almost pseudo-Anosov flows on compact manifolds using any blowdown to a pseudo-Anosov flow on closed manifold.

We conclude with a remark that ties the notion of perfect fits to that of anti-homotopic  orbits from the introduction.

\begin{remark}[Anti-homotopic orbits] \label{rmk:antihom}
Let $M$ be a closed $3$-manifold with a pseudo-Anosov flow $\varphi$. If $\varphi$ has perfect fits, then Fenley proved that $\varphi$ has anti-homotopic orbits (\cite[Theorem B]{Fen16}). Conversely, if $\varphi$ has anti-homotopic orbits $\gamma_1$ and $\gamma_2$, then after taking compatible lifts to the universal cover and projecting to the flow space $\orb$, we obtain points $p_1$ and $p_2$ that are fixed by some $g \in \pi_1 (M)$ in the conjugacy class determined by $\gamma_1,\gamma_2$. Fenley shows that in this case $p_1$ and $p_2$ are connected by a \emph{chain of lozenges} (\cite[Theorem 4.8]{fenley1999foliations}), which in particular implies that $\varphi$ has perfect fits. 
\end{remark}

\subsection{The Agol--Gueritaud construction and transversality to the flow}
\label{sec:AG}
Given a pseudo-Anosov flow $\phi$ on a closed $3$-manifold and a finite collection $\kappa$ of regular orbits that kills its perfect fits, the Agol--Gueritaud construction produces veering triangulation $\tau$ on the manifold $M \ssm \kappa_s$, and our previous work shows that the $2$-skeleton of $\tau$ is positively transverse to $\phi$. From the above discussion, this generalizes as follows:

\begin{theorem} \label{th:AG+LMT}
Suppose that $\phi$ is a transitive almost pseudo-Anosov flow on a compact manifold $M$. Then there is a veering triangulation $\tau$ on $M \ssm \kappa_s$ whose $2$-skeleton is positively transverse to $\phi$.
\end{theorem}

\begin{proof}
We use the notation from \Cref{sec:flowspace}.

The construction in \cite{LMT21} produces a veering triangulation $\tau$ on $M^\flat \ssm \kappa^\flat_s$ whose $2$--skeleton is transverse to the flow $\phi^\flat$ under the assumption that its completed flow space $\mc P$ has no perfect fits. But we saw in \Cref{sec:pfits}, $\mc P$ has no perfect fits if and only if $\kappa$ kills the perfect fits of $\phi$. This, together with the fact that $\pi$ restricts to a flow preserving homeomorphism $M \ssm \kappa_s \to M^\flat \ssm \kappa^\flat_s$, completes the proof.
\end{proof}

We call $\tau$ the veering triangulation \define{associated to} the pair $(\phi, \kappa)$. When $\phi$ has no perfect fits, the veering triangulation associated to $\phi$ is the one associated to $(\phi, \emptyset)$. When $\phi$ has perfect fits, there is significant flexibility in the choice of $\kappa$; we will call $\tau$ \define{an associated} veering triangulation if it is associated to some pair $(\phi, \kappa)$ for some choice of $\kappa$. 

\medskip

Now fix once and for all a \define{standard neighborhood} $U$ of $\kappa_s$. More precisely, if $\pi \colon M \to M^\flat$ is a pseudo-Anosov blowdown with $M^\flat$ closed, we let $U^\flat$ be a standard neighborhood of the closed orbits in $\kappa^\flat_s$ (see \Cref{global blowups}) and take $U = \pi^{-1}(U^\flat)$. Then, as in \Cref{sec:reltube}, $\tau$ is a relative veering triangulation of $M$ with tube system $U$. 
In particular, we have that
\[
\tau_U = \tau^{(2)} \cap M \cut\ U
\]
is a branched surface on $M_U = M \cut U$ that is positively transverse to $\phi$.

\begin{remark}
\label{rmk:same_veer}
The blowdown $\pi\colon M \to M^\flat$ restricts to a homeomorphism $M \cut U \to M^\flat  \cut U^\flat$ sending $\tau_U$ to $\tau_{U^\flat}$. In this sense, the veering triangulation associated to an almost pseudo-Anosov flow is essentially the same as the one associated to any pseudo-Anosov blowdown. 
\end{remark}

\smallskip

Let $U_i$ be a component of $U$, which we recall is a neighborhood in $M$ of either a closed orbit or blown complex (possible containing a single boundary component of $M$). Then the inner component of $\del U_i$ (i.e. the one not meeting $\partial M$) is a torus endowed with some combinatorial data: there are collections of stable and unstable prong curves on $\del U_i$ as well as a tessellation of $\del U_i$ coming from $\del \tau_U$, which divides $\del U_i$ into upward and downward ladders.
 
\begin{lemma}\label{lem:ladderpoleprong}
Let $U_i$ be a component of $U$. Each upward ladder contains exactly one stable prong curve. Each downward ladder contains exactly one unstable prong curve. 
\end{lemma}

See \Cref{fig:ladders} for an illustration in the case where $U_i$ is a neighborhood of a singular orbit.

\begin{proof}
Lemma 2.8 of \cite{landry2019stable} treats the case where $\phi$ is the suspension flow of a pseudo-Anosov map on a compact surface with boundary. However, the proof there applies verbatim since our triangulation $\tau$ is still obtained from the Agol--Gu\'eritaud construction as in \Cref{th:AG+LMT}.
\end{proof}

 \begin{figure}
    \centering
    \includegraphics[height=2in]{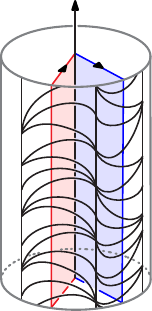}
    \caption{Stable (red) prong curves bisect upward ladders, while unstable (blue) prong curves bisect downward ladders.}
    \label{fig:ladders}
\end{figure}

Given \Cref{lem:ladderpoleprong}, it is now clear from the definitions that the veering triangulation $\tau_U$ on $M$ relative to $U$ is strict if and only if $\kappa =  \emptyset$, whence the flow $\phi$ on $M$ has no perfect fits.

\subsection{Flows represent faces}\label{sec:flowsrepresentfaces}
We have now developed enough combinatorics to easily prove Mosher's beautiful theorem ``Flows Represent Faces" \cite{mosher1992dynamical}, which connects the Thurston norm to dynamics, in the broader setting of pseudo-Anosov flows on manifolds with boundary. 

Let $\phi$ be a pseudo-Anosov flow with no perfect fits on a manifold $M$ with its associated strict relative veering triangulation $\tau$, furnished by \Cref{th:AG+LMT}.
Let $\gamma$ be an orbit of $\phi$. We define the index of $\gamma$ to be 2 minus its number of prongs; in particular periodic orbits in $\del M$ have index $-1$. Define $e_\phi\in H^2(M,\del M)$ by
\[
2e_\phi=\sum_i \langle \ind(\gamma_i)[\gamma_i],\cdot\rangle, 
\]
where the sum is over all singular orbits in $\intr(M)$ and all \emph{stable} periodic orbits in $\del M$ (i.e. the boundaries of stable half-leaves meeting $\del M$). By \Cref{lem:partialeulerclassformula} and \Cref{lem:ladderpoleprong}, we have $e_\tau=e_\phi$. 

Let $\cone_1(\phi)$ denote the cone in $H_1(M)$ generated by the periodic orbits of $\phi$. Its dual, $\cone_1^\vee(\phi) \subset H^1(M) = H_2(M, \partial M)$, is the cone of classes that are nonnegative on the periodic orbits of $\phi$.
It follows from \cite{LMT21} that the collection of periodic orbits of $\phi$ and the collection of directed cycles in the dual graph $\Gamma$ of $\tau$ generate the same cone in $H_1(M)$, so $\cone_1(\phi)=\cone_1(\tau)$. See e.g. \cite[Corollary 6.15]{LMT21}.

Hence the following follows immediately from \Cref{cor:combinatorialflowsrepresentfaces}:

\begin{theorem}[Flows Represent Faces (with boundary)]
\label{th:flows_rep_faces}
Let $\phi$ be a pseudo-Anosov flow with no perfect fits on a compact oriented 3-manifold $M$. Then $\cone_1^\vee(\phi)$ is equal to the cone over a face of the Thurston norm ball whose codimension is equal to the dimension of the largest linear subspace contained in $\cone_1(\phi)$. 

Moreover, this cone is the maximal domain on which $x$ and $-e_\phi$ agree as functions $H^2(M)\to \R$.
\end{theorem}

We remark that the hypothesis that $\phi$ has no perfect fits allows us to choose $\tau$ to be strict, so $M$ admits no essential surfaces of nonnegative Euler characteristic by \Cref{prop:top}. Hence in the above theorem the Thurston norm for $M$ is a norm and not merely a pseudonorm.

We also remark that, as in \cite{mosher1992dynamical}, the hypothesis that $\varphi$ has no perfect fits cannot be removed. See also the example in \Cref{sec:example}.

Finally, in \cite{mosher1992dynamical} Mosher included the requirement that $\phi$ be quasigeodesic in the statement of his theorem, and used the quasigeodesic property in his proof. He wrote ``It would also be nice to find a proof of {Flows Represent Faces} which makes no
use of the quasigeodesic hypothesis, for the reason that hyperbolic geometry seems completely extraneous to the purpose of the theorem." In the time since, Fenley has shown that pseudo-Anosov flows with no perfect fits are quasigeodesic (see \cite{Fen16} for the strongest statement). Note, however, that the proof we give above in the more general setting does not use quasigeodesity.

\section{Relatively carried surfaces are almost transverse}
As before, let $\phi$ be a transitive pseudo-Anosov flow on the compact manifold $M$.  We say that a properly embedded surface $S$ is \define{almost transverse} to $\phi$ if it is transverse to a (generalized) dynamic blowup $\phi^\sharp$ of $\phi$ on $M$ (see \Cref{rmk:blowup_on_M}).

The goal of this section is to prove the following:

\begin{theorem}\label{th:dynamic blowup for carried}
Let $\phi$ be a transitive pseudo-Anosov flow and let $\tau$ be any associated veering triangulation.
If $S$ is relatively carried by $\tau$ then it is almost transverse to $\phi$.
\end{theorem}

In fact, by assuming that $S$ is efficiently carried, we can be more precise about which dynamic blowups are needed:

\begin{proposition}\label{prop:efficientblowup}
Assume in addition to the hypotheses of \Cref{th:dynamic blowup for carried}, that $S$ is carried efficiently by $\tau$. Then the blowup locus consists of the singular orbits and boundary components whose tubes are met by $S$ in ladderpole annuli.
\end{proposition}

Note that if $S$ can be isotoped to be transverse to a dynamic blowup $\phi^\sharp$, then its intersection with each blown annulus can only be along core curves and essential arcs.
We say that the dynamic blowup is \define{minimal} with respect to $S$ if (after isotopy) it intersects every nonboundary blown annulus in core curves. 

\medskip
For the remainder of the section, we fix the surface $S$ and the pseudo-Anosov flow $\phi$ as in \Cref{th:dynamic blowup for carried}. By assumption, there is a collection $\kappa$ of closed orbits that kills the perfect fits of $\phi$ so that the veering triangulation $\tau$ on $M \ssm \kappa_s$ furnished by \Cref{th:AG+LMT} has the following property:
There is a standard neighborhood $U$ of $\kappa_s$ in $M$ so that $S$ is relatively carried by $\tau$. See \Cref{sec:reltube} and \Cref{sec:AG} for definitions. We recall here that since $\phi$ is pseudo-Anosov, $\kappa_s = \kappa \cup \Pi_s \cup \{P_i \}$, where $\partial M = \bigcup P_i$ is the union of boundary blown annuli. 

Since $S$ is relatively carried by $\tau$,  $S \cap M \cut U$ is carried by $\tau_U$ and so
at each component $U_i$ of $U$, the intersection $S \cap \partial U_i$ is a union of parallel curves carried by $\partial \tau_U$. If $S \cap U_i$ contains a ladderpole annulus, then we say that $S$ is \define{ladderpole} at $U_i$. If $S$ meets $U_i$ but is not ladderpole there, then either $S\cap U_i$ is a collection of meridional disks or annuli with one boundary component on $\partial M$.

\subsection{Coherent arc systems and meridional disks and annuli}
\label{sec:arc_sys}
We begin with a general definiton:

Let $D$ be a closed disk or annulus, with points $p_1,\dots, p_n\in \del D$. 
Suppose that $\del D\ssm(p_1\cup\cdots\cup p_n)$ is oriented so that each $p_i$ is either a source or a sink. A \textbf{coherently cooriented arc system} is a collection of cooriented closed arcs properly embedded in $D$ such that for each arc $\alpha$,
\begin{itemize}
\item the interior of $\alpha$ lies in $\intr D$
\item the endpoints of each arc lie in $\del D\ssm(p_1\cup\cdots\cup p_n)$
\item $\alpha\pitchfork \del D$, and the coorientation of $\alpha$ is compatible with the orientation of $\del D\ssm(p_1\cup\cdots\cup p_n)$.
\end{itemize}
See \Cref{fig:coordisk} for an example in a disk. Let $\theta$ denote a permutation
of the $\{p_i\}$ preserving their order around the circle.  We say a coherently
cooriented arc system is \textbf{symmetric under  $\theta$} if $\theta$ preserves the
relation defined by the arcs, as well as their coorientations.

\begin{figure}[h]
    \centering
    \includegraphics[]{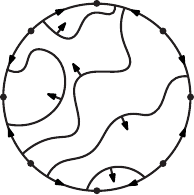}
    \caption{A coherently cooriented arc system.}
    \label{fig:coordisk}
\end{figure}

Coherently cooriented arc systems arise naturally when considering veering triangulations dual to flows; see \Cref{prop:arc system}.

Let $U_i$ be a $T^2\times I$ component of $U$ containing a boundary component $P_i$
of $\del M$. A \textbf{meridional annulus} of $U_i$ is an essential, properly embedded
annulus joining distinct boundary components of $U_i$ whose boundary component on $\del M$
essentially intersects the prong curves. 
In what follows, a meridional disk or annulus will always be taken to be transverse to the flow. That such disks/annuli exist is immediate from the local description of the flow in a neighborhood of its closed orbits and boundary. 

Let $U_i$ be a component of $U$. Let $D$ be a meridional disk or annulus for $U_i$, depending on whether $U_i$ is a solid torus or $T^2\times I$. In either case, we orient the open segments of $\del D\ssm(\text{prongs of $\phi$})$ so that each point of $\del D\cap (\text{stable prongs})$ is a source and each point of $\del D\cap (\text{unstable prongs})$ is a sink. We call this the \textbf{singular $\phi$-orientation} of $\del D$. See \Cref{fig:coorprongs}.

Following the flow around $U_i$ we obtain a permutation 
of the stable and unstable
prongs and their intersection points with $\del D$ preserving the cyclic order. We call this 
permutation the {\em twisting} at $U_i$.

\begin{figure}[h]
    \centering
    \includegraphics[height=1.2in]{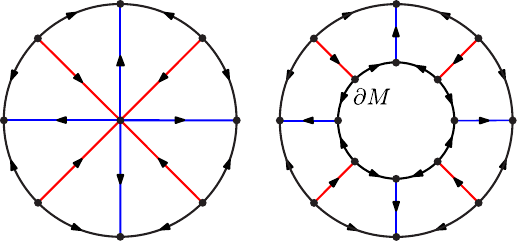}
    \caption{Illustrating the singular $\phi$-orientation on $\del D$ when $D$ is transverse to $\phi$ and either (left) a meridional disk of a solid torus component of $U$ or (right) a meridional annulus in a $T^2\times I$ component of $U$. Recall that red and blue denote stable and unstable, respectively.}
    \label{fig:coorprongs}
\end{figure}

\begin{proposition}
\label{prop:arc system}
Let $U_i$ be a component of $U$, and let $S$ be a surface relatively
carried by $\tau$ which is ladderpole at $U_i$. Then up to an isotopy of $S$ supported in
$U_i$, there exists a meridional disk or annulus $D$ for $U_i$ so that $\mc A = S\cap D$
is a coherently oriented arc system that is symmetric under
$\theta$, where $\theta$ is the twisting at $U_i$. 
Moreover, the
intersection of a ladderpole annulus of $S \cap U_i$ with $D$ corresponds to the orbit of
an arc of $\mc A$ under $\theta$.
\end{proposition}

\begin{proof}
Since the curves $S \cap \partial U_i$ are ladderpole, we can choose a ($\phi$-transverse) meridional disk or annulus $D$ such that $\partial D$ intersects theses curves minimally. Then by an innermost disk argument, we can isotope $S$ in $U_i$ so that $S\cap D$ is a collection of properly embedded arcs. We can further isotope $S$ so that each annulus of $S$ which meets $\del M$ is disjoint from the prongs of $U_i$ (see \Cref{lem:ladderpoleprong}).

Let $c$ be the boundary component of $D$ which does not lie on $\del M$ (if $D$ is a disk this is the only boundary component). It follows from \Cref{lem:ladderpoleprong} and the definition of the singular $\phi$-orientation that the coorientation of $S\cap D$ is compatible with the singular $\phi$-orientation of $\del D$ at each point of $S\cap c$.

Further, for each component $\alpha$ of $S\cap D$ touching both components of $\del D$, $\alpha$ is disjoint from the prongs of $M$. This implies that the coorientation of $\alpha$ is compatible with the singular $\phi$-orientation of $\del D$ at the boundary point of $\alpha$ on $\del M$.

The statement about $\theta$-symmetry follows from the fact that each annulus of $S\cap U_i$ has boundary components parallel to the prongs of $U_i$, which themselves are $\theta$-symmetric.
\end{proof}

We say a meridional disk/annulus satisfying the conclusion of \Cref{prop:arc system} is \define{adapted to} $S$.

\subsection{Blowups adapted to a surface} \label{sec:surfaceblowups}
Let $S$ be a relatively carried surface and fix a collection of adapted meridional disks/annuli (from \Cref{prop:arc system}) for each component of $U$ that $S$ meets in ladderpole annuli. 
Here we construct a (generalized) dynamic blowup of $\phi$ on $M$ that is transverse to $S$ by modifying the flow only in these ladderpole components of $U$.
We begin by describing the necessary blowups at each orbit in $U$ and then describe the blowups along the boundary components in $U$.

\subsubsection*{Blowup at a closed orbit}
First consider 
a closed orbit $\gamma$ of $\kappa_s$ contained in a component $U_\gamma$ of $U$ so that $S$ is ladderpole on $U_\gamma$,
let $D$ be a meridional disk of $U_\gamma$ adapted to $S$, 
and $f$ the return map of the flow. Near $p=D\intersect c$, $f$ has the standard form described in \Cref{local blowups}. 
Note that we allow for the possibility that $\gamma$ is a regular orbit; in this case no blowup will be performed (see \Cref{rmk:no_blow}), but we must still position $S$ in $U_\gamma$ so that it is transverse to the flow.

Recall that $T$ in $D$ is the union of prongs
(see \Cref{fig:arcs-internal} where they are
indicated by blue and red radial arcs), permuted by $f$ and cutting $D$ into quadrants.
The boundary of $D$, divided into quadrant arcs, receives a singular $f$-orientation as in
the previous section (see \Cref{fig:coorprongs}). Let $A$ be the coherently cooriented arc
system properly embedded in $D$, obtained from $S$ as in
\Cref{prop:arc system}.

\begin{figure}[htbp]
\begin{center}
  \includegraphics[height=1.5in]{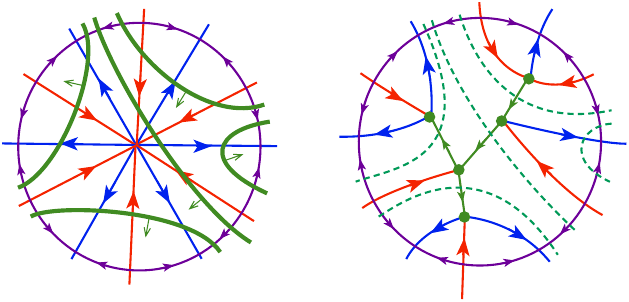}
\caption{A coherently co-oriented arc system (left side, in green) determines a blowup
  tree. On the right, the arc system (dotted) is realized transversely to the tree.}
\label{fig:arcs-internal}
\end{center}
\end{figure}

Given this, we construct the blow up tree $G$ as follows: Let $A'$ be the set of parallel classes
of $A$ (that is, one arc in $A'$ for each set of arcs in $A$ that begin and end in the
same quadrants), excluding the arcs with endpoints in adjacent quadrants. 
Let $G$ be the dual tree of $A'$ in $D$:
it collapses each complementary component to a vertex and has an edge crossing each
component of $A'$. Note that $G$ comes with an embedding into $D$ and the $\theta$--action
of $f$ on the disk and the arcs (where $\theta$ is the twisting around $U_\gamma$) gives
rise to an automorphism of $G$, so far defined
not pointwise but with respect to its action on vertices and edges. We orient each edge
of $G$ consistently with the co-orientation of the arc to which it is dual. 

We may identify $D\ssm p$ with the complement of $G$ in a disk $D^\sharp$,
in such a way that each
prong of $T$ coming into $D$ terminates at the vertex of $G$ associated to the region
where that prong enters $D$.
The union of $G$ with the prongs of $T$ gives the full
blowup tree $T^\sharp$, which we observe has no vertices of valence two. 
See \Cref{fig:arcs-internal}.
We now need to extend the dynamics
of $f$ on $D\ssm p$ to a continuous map  on $D^\sharp$ for which the vertices of $G$ are
periodic, and the return map to each edge has the dynamics (up to topological conjugacy)
prescribed by the orientations. This is done following the discussion in \Cref{local blowups}.

Each quadrant $Q$ of the original picture embeds now with a compactification arc that lies
along the tree and whose edges are consistently oriented.
To see this consistency, denote
by $Q^\sharp$  the compactified quadrant, and consider the arcs of $A'$ that cross the arc
of $\del D$ subtended by $Q$. Their dual orientations are consistent with the
$f$-orientation of that arc, and this implies the corresponding edges of $G$, which lie in
the boundary of $Q^\sharp$, have consistent orientations. The prong arcs adjacent to $Q$
have consistent orientations by definition. See \Cref{fig:quadrant-consistent} for a
typical picture of this.

\begin{figure}[htbp]
\begin{center}
  \includegraphics[height=1.5in]{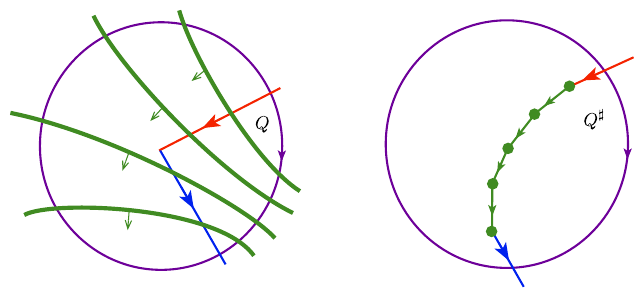}
\caption{Consistency of orientations induced on the boundary of a blown-up quadrant}
\label{fig:quadrant-consistent}
\end{center}
\end{figure}

Thus we may choose the dynamics
of the return map to be given by a standard repeated blowup on each quadrant (\Cref{local blowups}), and match them along the two
sides of each edge to obtain a globally continuous map. 
(The choice for these matchings are one source of the nonuniqueness mentioned in \Cref{global blowups}.).

The arc system now embeds in the new picture so that each arc crosses the tree at the edge
dual to it, and its coorientation agrees with the orientation of the edge.  The arcs from $A$ which were excluded when constructing $A'$ (i.e. those that meet adjacent quadrants) were transverse in the original picture and so remain transverse here (see \Cref{fig:arcs-internal}).

\subsubsection*{Blowup at a boundary component}
Next let $P$ be a boundary torus contained in the component $U_P$ of $U$. The flow near $P$ in $U_P$ can be taken to be the suspension flow of a
map on a meridional annulus $E$, which is a Fried blowup as in
\Cref{local blowups}. Let the inner
boundary of $E$ correspond to its intersection with $P$, and again consider the sectors
into which $E$ is divided: this time each sector already contains a ``blowup arc'' along
the inner boundary.

\begin{figure}[htbp]
\begin{center}
  \includegraphics[height=1.5in]{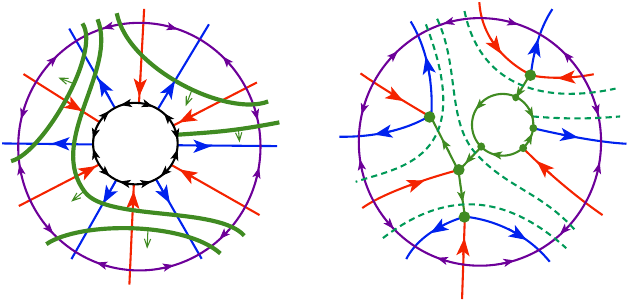}
\caption{A coherently co-oriented arc system near a boundary component determines a blowup
  graph. On the right, the arc system (dotted) is realized transversely to the graph.}
\label{fig:arcs-boundary}
\end{center}
\end{figure}

Again let $A$ be an invariant embedded collection of arcs with endpoints on the 
boundary of $E$ (each arc has at least one endpoint on the outer boundary), and again we require that $A$ inherits a coorientation from the  boundary
arcs (see \Cref{fig:arcs-boundary}). Components of $A$ with an endpoint on the inner
boundary must be contained in a single sector and their coorientation is consistent with
the orientation on both boundaries. 

Let $A'$ contain one arc for each
parallel class of arcs of $A$, 
excluding (as before)  arcs which have endpoints in adjacent quadrants, but now also 
excluding arcs with an endpoint on the inner boundary circle. 

We will now construct a graph $G$ embedded in $E$ which acts as the
``dual'' to $A'$. The graph $G$ contains a circle which we identify with
the inner boundary of $E$ and associate to the annulus complementary region of $E\ssm A'$. 
In each disk complementary region of $E\ssm A'$ we place a vertex of $G$,
and on the circle we place one vertex for each arc of $A'$ adjacent to the annulus region,
and one vertex for each landing point of an arc of $T$ that enters $E$ directly in the
annulus region (these vertices should be arranged in the same circular order as the arcs
and prongs that they correspond to). 
We place an edge in $G$
crossing each arc of $A'$, connecting the 
two vertices on either side if they are both disks, or a vertex in a disk to the
associated vertex on the circle if one of
the regions is the annulus. 
Each prong has a ``landing point'' in $G$: if it enters $E$ in a disk
region of $E\ssm A$ then we attach it to the corresponding vertex, and if it enters $E$ in
the annulus component we attach it to the associated vertex on the circle of $G$. 
We identify $E$ minus its inner boundary with $E\ssm G$ in such a way that the prongs
terminate at their landing points.

Let $T^\sharp$ denote $G$ union the prongs. Its edges inherit an orientation:
Each dual edge is oriented by the coorientation of the 
arc of $A'$ dual to it. 
Each segment $e$ on the circle is contained on the boundary of a component of $E\ssm
T^\sharp$ that meets the outer circle in a unique quadrant (because the arcs of $A'$ and/or
the prongs defining $e$ prevent any prongs from entering through this region). 
We give $e$ an orientation inherited from the boundary arc of this quadrant; see \Cref{fig:arcs-boundary} for examples of this. 

The components of $E\ssm G'$ are again compactifications of the original quadrants, with
multiple fixed points on the compactification arcs and dynamics prescribed by the
orientation. We again use multiple blowups to obtain a self-map which has the prescribed
dynamics on the graph and restricts on the quadrants to be conjugate to the original map. 
The arcs embed in this picture, again with consistent coorientations, by a similar
argument. 

The arc data, and hence the construction, are invariant under the twisting permutation $\theta$, so $f$
can be extended to the blowup graph as before.

\subsection{Suspending the blowups}
After the dynamic blowup associated to a system of arcs $A$, we obtain a system of
positively transverse annuli for the suspension flow as follows.

Consider first the case around a closed orbit.
In each quadrant $Q$  let $Q_0$ be the region bounded by an $f$-invariant hyperbola. A
radial arc $a$ in $Q_0$ connecting the origin to the hyperbola $\boundary Q_0$ is mapped
by $f$ to a disjoint radial arc, on the side of $a$ pointed to by its coorientation. In
$Q\times [0,1]$ we may connect $a\times\{1\}$ to $f(a)\times\{0\}$ by an disk which
(except over the origin) is transverse to the vertical flow. Gluing $Q_0\times\{1\}$ to
$Q_0\times\{0\}$ via $f$ produces an annulus transverse to the flow.

Now in the blown up quadrant $Q^\sharp$, we can adjust $a$ so that it meets one of the new
blowup intervals, and now $f(a)$ lies strictly to one side of $a$ and the same
construction yields an annulus which is also transverse to the flow on its boundary.

To produce a complete picture, we do this for all the arcs of $A$. That is, each arc of
$A$ can first be realized as a pair of radial arcs (in their respective quadrants) passing
through the center of $D$ and terminating in $Q_0$ for each $Q$ (we can choose $Q_0$ small
enough and trim $D$ so that
its boundary contains a large enough arc of $\boundary Q_0$). We blow up each
quadrant and attach the blowup arcs along the tree, in a way that respects the dynamics as
above, and the arcs of $A$ may be perturbed near
the blowup point so that they become embedded again and each one passes through its dual
tree edge. The map $f^\sharp$ created in this way takes each arc to the side pointed to by
its coorientation, and so as before we can create annuli through the solid torus that are
transverse to the flow.

There is one case that is slightly different: If an arc of $A$ has endpoints on adjacent
quadrants, then we can push it away from the singularity so that it still satisfies the
condition that its $f$-image is disjoint from it and on the side determined by its
coorientation. That is, the corresponding annulus can be made transverse to the flow
without performing any blowup. (This is the reason that arcs like this were excluded from
$A'$.) When there are only four quadrants (e.g. the periodic orbit
is in fact regular) this is all that can happen, so no blowup is necessary.

This construction can be made to match the original $S$ along the boundary of the
neighborhood, because $S$ meets the neighborhood in annuli that are transverse to the flow
near its boundary, in the same pattern as what we have created by \Cref{prop:arc system}.
The resulting surface $S^\sharp$ is transverse to the blowup flow $\phi^\sharp$ obtained by suspending the local blowups $f^\sharp$ obtained above.
For a blowup near a boundary component, the argument is similar, except that some of the
arcs terminate in the boundary.

With these constructions in hand, we can complete the proof of \Cref{th:dynamic blowup for carried}.
  
 \begin{proof}[Proof of  \Cref{th:dynamic blowup for carried}]
Suppose that $S$ is relatively carried by the dual veering triangulation $\tau$ relative to $U$, as in our setup. The surface $S$ is transverse to $\phi$ outside of $U$ (\Cref{th:AG+LMT}) and meets each component of $U$ in either a collection of disks, annuli meeting $\partial M$, 
or ladderpole annuli (or the empty set). Each disk intersection can be isotoped in $U$ to be transverse to $\phi$ since $S$ is transverse near the boundary of $U$. The same is true for any annulus meeting $\partial M$ that is not of ladderpole slope since its boundary can be made transverse to the flow in $\partial M$. 

Finally, the discussion above gives a recipe to dynamically blow up $\phi$ within the remaining components of $U$ to produce a flow $\phi^\sharp$ on $M$ that is transverse to the surface $S^\sharp$ so obtained. By construction, $S^\sharp$ and $S$ agree outside of the components of $U$ where $S$ is ladderpole. However, within such components, the description above (relying on \Cref{prop:arc system}) implies that they are isotopic by an isotopy supported in $U$.  Hence, after an isotopy $S$ is  transverse to the (generalized) dynamic blowup $\phi^\sharp$ and this completes the proof.
 \end{proof}

\section{Almost transverse surfaces are relatively carried}

In this section, we prove the converse to \Cref{th:dynamic blowup for carried}: an almost transverse surface to a pseudo-Anosov flows is relatively carried by any associated veering triangulation. 
Together with \Cref{th:dynamic blowup for carried}, this establishes:

\begin{theorem}[Almost transverse$\iff$relatively carried]\label{thm:ATtocarried}
Let $\phi$ be a transitive pseudo-Anosov flow on $M$, possibly with boundary, and let $\tau$ be any associated veering triangulation.
Then, up to isotopy, a surface $S$ is almost transverse to $\phi$ if and only if it is relatively carried by $\tau$.
\end{theorem}

First we need a lemma extending \cite[Section 2]{mosher1992dynamical}
to the case of manifolds with boundary. We remark that our generalization is necessary even in the case where the manifold $M$ from \Cref{thm:ATtocarried} is closed. This is because our strategy for the proof is to first perform a Fried blowup on the original flow and then apply \Cref{lem:taut_from_flow}. Recall the Euler class $e_\tau$ as defined in \Cref{sec:eulerclass}.

\begin{lemma}
\label{lem:taut_from_flow}
Any surface $S$ almost transverse to a transitive pseudo-Anosov flow is taut. If $\tau$ is any associated veering triangulation, 
then $x([S])=-e_\tau([S])$.
\end{lemma}

\begin{proof}
To begin, let $\psi$ be the almost pseudo-Anosov flow transverse to $S$ and
let $\tau$ be a veering triangulation on $M \ssm \kappa_s$ where $\kappa$ is any collection of orbits that kills the perfect fits of $\psi$ (\Cref{th:AG+LMT}). By blowing down blowup annuli of $\psi$ that $S$ meets in proper arcs, we can assume that $S$ is minimally transverse to $\psi$.

Let $U$ be the associated tube system where each component of $U$ is a standard neighborhood of a component of $\kappa_s$, and let $U_m$ be the subcollection of components that $S$ meets in meridional annuli (\Cref{sec:arc_sys}) or disks.  Hence, the components of $\kappa_s$ that are contained in components of $U_m$ are lone
orbits and boundary components of $M$. That is, all blown annuli of $\kappa_s \cap U_m$ are contained in $\partial M$.

Recall from \Cref{lem:partialeulerclassformula} that 
\[
e_\tau=\frac{1}{2}\langle\left( \sum \ind(T_i)\cdot\core(T_i)\right)-c,\cdot\rangle,
\]
where $c$ is the union of the cusp curves of all torus shells of $M\cut B^s$ (corresponding to the boundary components of $M$) and the sum is over solid torus components of $U$ (corresponding to components of $\kappa_s$). 
Because of \Cref{lem:ladderpoleprong}, in each $T^2 \times I$ component of $U_m$, we can isotope $c$, preserving orientation, to the union of stable singular orbits of $\psi$ on $\del M$ (i.e. those associated to stable prong curves). We replace $c$ by its image under this isotopy. Similarly on each solid torus component of $U_m$, we can assume that $\core(T_i)$ agrees on the nose with the associated closed orbit of $\psi$. Further, observe that $S$ meets no singular orbits contained in $U \ssm U_m$. Hence, when computing the pairing $e_\tau([S])$ we can restrict to the summation terms associated to $U_m$ and consider only the curves of $c$ contained in $U_m$.

Next, recall that if $x$ is a $p$-pronged singularity of a singular foliation on a surface with boundary $S$, its index is defined to be $\ind(x)=2-p$. This includes the case when $p$ lies on $\del S$; see \Cref{fig:indices}. With this definition, given a foliation of $S$ with only prong singularities
and leaves tangent to $\partial S$,
we have $\sum \ind(x)=2\chi(S)$ where the sum is over all singularities of the foliation. 

\begin{figure}[htbp]
\begin{center}
\includegraphics[]{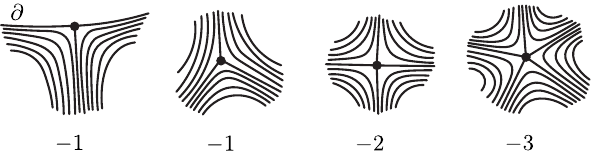}
\caption{Some singularities of singular foliations on surfaces with boundary, and their respective indices.}
\label{fig:indices}
\end{center}
\end{figure}

Now recall that $W^s$ is the stable foliation of $\psi$. Since $S$ is transverse to $\psi$, $S\cap W^s$ defines a singular foliation of $S$, the singularities of which are the points of intersection between $S$ and the singular orbits of $\psi$ that are contained in $U_m$. All these intersections are positive and so comparing formulas we see that $\chi_-(S)=-e_\tau([S])$. 
Since $\psi$ is transitive, we can find a closed curve pairing positively with any union of components of $S$. Hence each such union is homologically nontrivial. By \Cref{lem:dualnorm}, $S$ is taut.
\end{proof}

With this lemma in hand, we turn to the proof of \Cref{thm:ATtocarried}. The proof strategy of first performing a Fried blowup to obtain a pseudo-Anosov flow {without perfect fits} on a manifold with boundary will also be key in the next section.

\begin{proof}[Proof of \Cref{thm:ATtocarried}]
Recall that $\tau$ is a veering triangulation on $M \ssm \kappa_s$ where $\kappa$ is any collection of orbits that kills the perfect fits of $\phi$.
If $S$ is relatively carried by $\tau$, then $S$ is almost transverse to $\phi$ up to isotopy by \Cref{th:dynamic blowup for carried}.

Conversely, suppose that $S$ is a surface almost transverse to $\phi$. 
This means $S$ is transverse to an almost pseudo-Anosov flow $\phi^\sharp$ on $M$ of $\phi$ obtained by (generalized) dynamic blowup. Note that we can naturally think of $\kappa$ as a collection of closed regular orbits of $\phi^\sharp$.

Since $\kappa$ consists of regular orbits, the relative veering triangulation $\tau$ on $M$ is strict if  and only if $\kappa =\emptyset$. To remedy this, 
perform a Fried blowup of $\phi^\sharp$ on each orbit of $\kappa$, obtaining an almost pseudo-Anosov flow $\phi^*$ on a manifold $M^*$ which has a boundary component for each orbit in $\kappa$ (in addition to any boundary coming from $M$). Since each intersection of $S$ with $\kappa$ is transverse and positive, we obtain a surface  $S^*$ properly embedded in $M^*$ that is transverse to $\phi^*$ and maps to $S$ under the blowdown $M^* \to M$ from $\phi^*$ to $\phi^\sharp$.

By construction, the flow $\phi^*$ has no perfect fits (see \Cref{rmk:blow_to_no}). Moreover, its associated veering triangulation $\tau^*$ agrees with $\tau$ under the blowdown $M^* \to M$; said differently, they are the same veering triangulation of $M \ssm \kappa_s$; see \Cref{rmk:same_veer}.

Since $S^*$ is transverse to $\phi^*$ and $\tau^*$ is associated to $\phi^*$, \Cref{lem:taut_from_flow} gives that $S^*$ is taut and $x([S^*])=-e_{\tau^*}([S])$. By \Cref{combiSTST}, we may isotope $S^*$ to be relatively carried by $\tau^*$, fixing $\del S^*$. 
By postcomposing this isotopy with the blowdown $M^*\to M$, we see that $S$ is relatively carried by $\hbs$ up to isotopy.
\end{proof}

\begin{remark}
An alternative approach would be to note that the almost transversality condition implies that no complementary region of $S\cap B^s$ or $S\cap B^u$ can be a bigon with both cusps corresponding to negative intersections of $S$ with $\kappa$. This fact allows one to apply the steps of the algorithm from \cite{Landry_norm}. This algorithm is fairly complex. The approach above is much simper, highlighting the utility of working with pseudo-Anosov flows on manifolds with boundary.
\end{remark}

\section{Transverse surface theorems}
In this section, we prove several transverse surface theorems that realize surfaces as (almost) transverse to pseudo-Anosov flows. We then also use our Fried blowup technique to realize relative homology classes as Birkhoff surfaces (\Cref{th:Birk}).

 In \Cref{sec:example}, we illustrate the sharpness of our theorems with an example.
 
\medskip

Essentially the same technique from the proof of \Cref{thm:ATtocarried} (performing a Fried blowup to remove perfect fits) can be used to extend Mosher's transverse surface theorem (for general pseudo-Anosov flows) to manifolds with boundary.

\begin{theorem}[Transverse surface theorem with boundary]
\label{th:mosher_with_boundary}
Given a transitive pseudo-Anosov flow $\phi$ on $M$ (possibly with boundary), any integral first cohomology class pairing nonnegatively with all closed orbits of $\phi$ is Poincar\'e dual to a surface almost transverse to $\phi$.
\end{theorem}

\begin{proof}
Let $\kappa$ be any collection of regular orbits of $\phi$ that kill perfect fits; this exists by \Cref{prop:tsang}.
Following the approach of \Cref{thm:ATtocarried}, perform a Fried blowup on all orbits of $\kappa$ to obtain $M^*$, $\phi^*$, and $\tau^*$. In this case, $\tau^*$ is the associated veering triangulation of $M^*$ since $\phi^*$ has no perfect fits. 

Any class $\eta \in H^1(M)$ that pairs nonnegatively with all closed orbits of $\phi$ pulls back to a class $\eta^* \in H^1(M^*)$ with the same property for $\phi^*$. 
Hence, $\eta^*$ can be realized by a surface $S^*$ in $M^*$ that is carried by $\tau^*$ by \cite[Proposition 5.11]{LMT20}. Since $\eta^*$ is the pullback of the class $\eta$ in $M$, $S^*$ can be extends across $U$ to a surface $S$ representing $\eta$ and relatively carried by $\tau$. Hence, after an isotopy $S$ is almost transverse to $\phi$ by \Cref{th:dynamic blowup for carried} and the proof is complete.
\end{proof}

The point of our \emph{strong} transverse surface theorems is to realize a surface in (almost) transverse position, up to isotopy (rather than just homology).
When the flow is without perfect fits, the strongest possible result is the following:

\begin{theorem}[Strong transverse surface theorem]
\label{th:stst}
Let $\phi$ be a pseudo-Anosov flow with no perfect fits on $M$, possibly with boundary. Let $S$ be a taut surface pairing nonnegatively with all closed orbits of $\phi$. Then $S$ is almost transverse to $\phi$ up to isotopy.
\end{theorem}

\begin{proof}
Since $\phi$ has no perfect fits, there is an associated veering triangulation $\tau$ in $M \ssm \Pi_s$.
By the results of \cite{LMT21}, the cone in $H_1(M)$ generated by closed orbits is equal to $\cone_1(\Gamma)$, so $[S]\in \cone_1^\vee(\Gamma)$. By \Cref{combiSTST}, $S$ is carried by $\hbs$ up to isotopy, so is almost transverse to $\phi$ up to isotopy by \Cref{th:dynamic blowup for carried}. 
\end{proof}

Combining the above arguments give the following extension to the general case.

\begin{theorem}
\label{th:general_stst}
Let $\phi$ be a transitive pseudo-Anosov flow on $M$, and let $\kappa$ be any finite collection of closed regular orbits of $\phi$ that kills perfect fits. Then an oriented surface $S$ in $M$ is isotopic to a surface that is almost transverse to $\phi$ if and only if
\begin{enumerate}
\item $S$ is taut,
\item $S$ has nonnegative (homological) intersection with each orbit of $\phi$,
and 
\item  the algebraic and geometric intersection numbers of $S$ with $\kappa$ are equal.
\end{enumerate}
\end{theorem}

Note that if $\varphi$ has no perfect fits (e.g. when $\phi$ is circular), then we can take $\kappa = \emptyset$, and in this case the third condition is vacuous. Hence \Cref{th:stst} is a special case of \Cref{th:general_stst}.

\begin{proof}
First, isotope $S$ so that condition $3$ holds. Using the previous method, we can Fried blowup $\phi$  at $\kappa$, just as in the proof of \Cref{thm:ATtocarried}, to obtain a pseudo-Anosov flow $\phi^\sharp$ without perfect fits on $M^\sharp$ and a properly embedded surface $S^\sharp$ that blows down to $S$. From condition $(3)$, it follows that $S^\sharp$ is still taut in $M^\sharp$. Indeed, the condition implies that the blowup boundary components of $S^\sharp$ are oriented coherently on each blowup component of $\partial M^\sharp$. In particular, any homologous surface needs to have at least that number of boundary components along $\partial M^\sharp$, and so tautness of $S$ implies tautness of $S^\sharp$.

Moreover, as in \Cref{th:mosher_with_boundary}, $S^\sharp$ has nonnegative intersection number with each orbit of $\phi^\sharp$. Hence we can apply \Cref{th:stst} to conclude that $S^\sharp$ is almost transverse to $\phi^\sharp$ after an isotopy. By blowing back down (so that each component of $\partial S^\sharp$ on a blowdown torus is blown down to a point), this implies that $S$ is almost transverse to $\phi$, as required. 
\end{proof}

\subsection{Realizing Birkhoff surfaces}

Our technique of blowing up to obtain flows on manifolds with boundary has applications to
Birkhoff surfaces for transitive pseudo-Anosov flows. Recall from the introduction that a \define{Birkhoff surface}
 of $\phi$ is an immersed surface in $M$ whose boundary
components cover closed orbits of $\phi$ and
whose interior is embedded transversely to $\phi$. 
We call the closed orbits of $\phi$ that are covered by the boundary components of a Birkhoff surface its \define{boundary orbits}. More generally, a Birkhoff surface of a dynamic blowup of $\phi$ will be called an \define{almost Birkhoff surface} of $\phi$.
Finally, a \define{Birkhoff section} is a Birkhoff surface that meets every flow segment of some definite (uniform) length. 

Let $\kappa$ be a collection of closed orbits, possibly containing singular orbits. Let $U_\kappa$ be a standard neighborhood of $\kappa$. Then using excision and Poincar\'e duality we make the identification
\[
H_2(M, \kappa) 
= H^1(M \ssm \kappa). 
\]
Meanwhile, if $\pi \colon M^\sharp \to M$ is the Fried blowup of $\phi$ to $\phi^\sharp$ at $\kappa$, then $\pi$ induces a homeomorphism from $M^\sharp \ssm U^\sharp \to M \ssm U_\kappa$, where $U^\sharp = \pi^{-1}(U_\kappa)$ is a standard neighborhood of the components of $\partial M^\sharp$ that blow down to orbits in $\kappa$.  

The next theorem generalizes Fried's result on Birkhoff sections \cite[Theorem N]{fried1982geometry} and seems new even for geodesic flow on hyperbolic surfaces.

\begin{theorem}[Representing relative classes with Birkhoff surfaces]
\label{th:Birk}
Let $\phi$ be a transitive pseudo-Anosov flow and let $\kappa$ be any collection of closed orbits. Then a class $\eta \in H_2(M, \kappa) = H^1(M \ssm \kappa)$ is represented by an almost Birkhoff surface in $M$ if and only if it is nonnegative on closed orbits of $M \ssm \kappa$.

Moreover, if each singular orbit of $\phi$ is either in $\kappa$ or positive under $\eta$, then $\eta$ is represented by a (honest) Birkhoff surface.
\end{theorem}

Note that the `moreover' statement is automatic when $\phi$ is Anosov.

\begin{proof}
Note that when $\kappa =\emptyset$, the result reduces to \Cref{th:mosher_with_boundary}. We show that the general case also reduces to what we have already shown as follows: 

As above, let $\pi \colon M^\sharp \to M$ with $\phi^\sharp$ be the Fried blowup of $\phi$ with respect to $\kappa$. 
The class $\eta$ pulls back to $\eta^\sharp = \pi^*\eta \in H^1(M^\sharp)$, and we claim
that $\eta^\sharp$ is nonnegative on the closed orbits of $\phi^\sharp$. This is obvious
for the closed orbits that are not contained in blown boundary components, i.e. do not
blow down to components of $\kappa$. If $\gamma$ is a closed orbit in a blown boundary component, then we can approximate arbitrarily high multiples $\gamma^k$ with a closed orbit $c_k$ that is \emph{not} contained a boundary component. More precisely, we can choose a periodic orbit $d$ in $\intr (M)$ such that $[c_k] = [\gamma^k] \cdot [d]$, where $d$ is a fixed orbit. This can be accomplished using a Markov partition for the flow, or the flow graph of any associated veering triangulation (e.g. \cite{LMT21}). Since $\eta^\sharp(c_k) = \eta(\pi(c_k)) \ge 0$, we must have $\eta^\sharp(\gamma) \ge 0$ as claimed.

Hence, we may apply \Cref{th:mosher_with_boundary} to produce a surface $\Sigma^\sharp$ representing $\eta^\sharp$ that is almost transverse to $\phi^\sharp$. Of course, if $\phi^\sharp$ has no interior singularities that are $\eta^\sharp$--null, then $\Sigma^\sharp$ can be chosen to be honestly transverse to $\phi^\sharp$. It only remains to position $\Sigma^\sharp$ so that it blows down to a Birkhoff surface $\Sigma$ in $M$. That is, each blowup boundary component $P$ of $M^\sharp$ is foliated by the fiber circles of $\pi_P \colon P \to k$ for $k \in \kappa$, and $\Sigma^\sharp$ needs to be arranged so that each of its boundary components is either one of these circles (in which case that component blows down to a point of transversality with $\phi$) or is transverse to this foliation (in which case that component covers the closed orbit of $\phi$). Given our local model for the blowup, it is not hard to see that this is possible. See \cite[Theorem M]{fried1982geometry} where this is done very carefully in the case of Birkhoff sections using essentially the same local model.
\end{proof}

\begin{corollary}
Let $\phi$ be a transitive pseudo-Anosov flow on a compact manifold $M$, and $\kappa$ a
finite collection of closed orbits.  Then there is an almost Birkhoff surface whose 
boundary orbits are contained in $\kappa$
 if and only if the closed orbits of $\phi$ are contained in a closed half-space of $H_1(M \ssm \kappa ; \mathbb{R})$. 
\end{corollary}

\subsection{An example}\label{sec:example}
The purpose of this subsection is to show by example 
that \Cref{th:stst} (\Cref{A}) does not hold without the `no perfect fits' assumption. That is, we will construct a manifold $M$, flow $\phi$, and a taut surface $Z$ that pairs nonnegatively with all closed orbits but cannot be made almost transverse to $\phi$ by isotopy. This will also show condition $(3)$ in \Cref{th:general_stst} (\Cref{B}) cannot be removed. 
Our example is based on the final example from \cite[Section 4]{mosher1992dynamical}, and assumes the reader is familiar with the basic language of sutured manifolds (see \cite{Gab83}).

\begin{figure}
    \centering
    \includegraphics[height=2.2in]{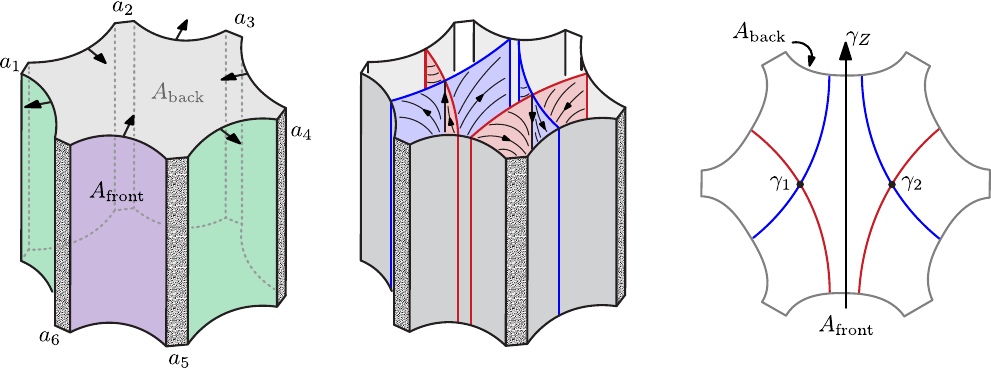}
    \caption{Left: a picture of the sutured solid torus $H$. Top and bottom are identified. Center: the flow $\phi$ in \Cref{lem:mosher} induces a semiflow in $H$ with precisely two periodic orbits, which are anti-homotopic. Right: projecting the orbit $\gamma_Z$ to a meridional disk of $H$.}
    \label{fig:roundhandle}
\end{figure}

\medskip

Let $H$ be a sutured solid torus with six longitudinal sutures labeled $a_1,\dots, a_6$ in cyclic order as shown in the left image in \Cref{fig:roundhandle}. Let $A_\back$ and $A_\front$ be the annuli in $\del H$ between $a_2$ and $a_3$ and between $a_5$ and $a_6$, respectively.
Let $S$ be a connected compact orientable surface with two boundary components and $\chi(S)<0$, and let $T$ be another such surface but with four boundary components. The products $(S\times I, \del S\times I)$ and $(T\times I,\del T\times I)$ are product sutured manifolds. Let $N$ be the sutured manifold obtained by gluing $\del S\times I$ to $a_1\cup a_4$ by a homeomorphism, and gluing $\del T\times I$ to $a_2\cup a_3\cup a_5\cup a_6$ by a homeomorphism. This book of $I$-bundles is shown schematically in \Cref{fig:mosherschematic}.

\begin{figure}
    \centering
    \includegraphics[]{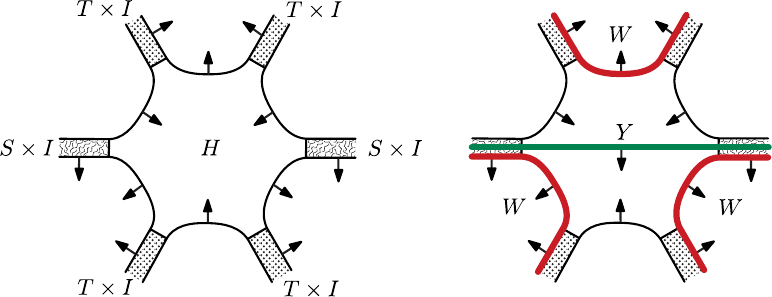}
    \caption{Left: a schematic picture of the gluing of $S\times I$, $T\times I$, and $H$ that yields $N$. Right: the surfaces $W$ (red) and $Y$ (green).}
    \label{fig:mosherschematic}
\end{figure}

When talking about $N$, we continue to refer to $S\times I$, $T\times I$, $a_1,\dots, a_6$ et cetera, and by these we mean the images of those objects under the gluings used to build $N$.
Let $Y$ be the surface in $N$ which is obtained by taking the union of $S\times \{\frac{1}{2}\}$ with an annulus in $H$ connecting $a_1$ and $a_4$. The orientation on $Y$ is chosen so that for each $s\in S$, the oriented segment $s\times I$ intersects $Y$ positively. Let $W=\del_+ N$, oriented compatibly with the coorientation of $\del_+N$. This is shown on the right side of  \Cref{fig:mosherschematic}.

Note that $\del_-N$ is homeomorphic to $\del_+N$. If we choose a homeomorphism between the two, we can glue to obtain a closed manifold. In such a situation we will continue to use the notations $W$ and $Y$ to denote the images of $W$ and $Y$ in the closed manifold.

We will take the following facts as given: 
\begin{lemma}[Mosher, Section 4 \cite{mosher1992dynamical}]\label{lem:mosher}
Let $N$ be as above.
\begin{enumerate}[label=(\alph*)]
\item There exists a hyperbolic 3-manifold $M$ obtained by gluing $\del_-N$ to $\del_+N$, and a pseudo-Anosov flow $\phi$ on $M$ positively transverse to $W$.
\item The surfaces $W,Y\subset M$ are taut and lie over the face of the Thurston norm ball determined by the Euler class of $\phi$.
\item Further, the pullback semiflow $\phi_N$ on $N$ has the following properties:
\begin{enumerate}[label=(\roman*)]
\item $\phi_N$ is compatible with the sutured structure of $N$ in that it points out of $N$ along $\del_+N$ and into $N$ along $\del_-N$
\item ``The semiflow is a product away from $H$": the flowlines of $\phi_N$ lying in $S\times I$ and $T\times I$ are all of the form $x\times I$ for some $x\in S\cup T$.
\item There are exactly 2 closed orbits of $\phi_N$ lying in $H$, $\gamma_1$ and $\gamma_2$. They are both freely homotopic to the core of $H$ and also have the property that $\gamma_1$ is freely homotopic to $\gamma_2$ with the orientation reversed.
\item Every other flowline  of $\phi_N$ in $H$ is either a properly embedded interval or has one end on $\del_\pm N$ and accumulates on $\gamma_1$ or $\gamma_2$ in the forward or backward direction.
\item There exists a periodic orbit $\gamma_Z$ of $\phi$ which pulls back to a single oriented flowline of $\phi_N$ in $H$ from $A_\front$ to $A_\back$.

\end{enumerate}
\end{enumerate}
\end{lemma}

The orbits $\gamma_1$ and $\gamma_2$, as well as parts of their stable and unstable leaves, are shown in the center image of \Cref{fig:roundhandle}; the orbit $\gamma_Z$ is shown in the righthand image.

We remark that the flow $\phi$ in the above lemma is constructed using the Gabai--Mosher construction of pseudo-Anosov flows transverse to sutured hierarchies in closed oriented atoroidal 3-manifolds, which is unpublished. However, Mosher laid much of the foundation in \cite{Mos96}, and Landry and Tsang are expositing the rest in detail (\cite{LandryTsangStep1} and work in progress).

Building on \Cref{lem:mosher}, we have:

\begin{theorem}\label{thm:example}
There exists a closed 3-manifold $M$, a transitive pseudo-Anosov flow $\phi$ on $M$, and a taut surface $Z\subset M$ such that $[Z]$ pairs nonnegatively with each closed orbit of $\phi$ but is not almost transverse to $\phi$ up to isotopy.
\end{theorem}

\begin{proof}
Let $M$ be as in \Cref{lem:mosher}; we will use the same notation as in that lemma. Note that since $W$ and $Y$ are taut and lie over the same face of the Thurston norm ball, the surface $Z=W\cup Y$ is taut. Note that $Z$ is not almost transverse to $\phi$ up to isotopy because this would imply the same for $Y$, which is impossible: by \Cref{lem:mosher}(c)(v), $Y$ has pairing $-1$ with the orbit $\gamma_Z$.

We now claim that for every closed orbit $\gamma$ of $\phi$, $[Z]$ pairs nonnegatively with $[\gamma]$; this will complete the proof. 

If $\gamma$ is equal to either $\gamma_1$ or $\gamma_2$ then $\gamma$ has null pairing with $Y$, since the $\gamma_i$'s are homotopic to the core of $H$, which can be homotoped into $Y$. Hence we may assume that $\gamma$ is any other closed orbit of $\phi$. In this case, $\gamma$ pulls back to finitely many oriented segments $\gamma_1,\dots, \gamma_n$ in $N$ beginning on $\del_-N$ and ending on $\del_+N$. Fix $1\le i\le n$ and consider $\gamma_i$. By \Cref{lem:mosher}(c)(ii) $\gamma_i$ either lies in $(S\cup T)\times I$ or in $H$. 

Suppose $\gamma_i\subset (S\cup T)\times I$. Then $\gamma_i$ contributes only positive intersections to the algebraic intersection number of $Z$ and $\gamma$ by \Cref{lem:mosher}(c)(ii).

Next, suppose $\gamma_i\subset H$. The absolute value of the algebraic intersection of $\gamma_i$ and $Y$ is at most 1 since $Y\cap H$ separates $H$, while the algebraic intersection of $\gamma_i$ with $W$ is equal to $1$. It follows that $\gamma_i$ contributes nonnegatively to the algebraic intersection of $Z$ and $\gamma$.
Summing over $i=1,\dots, n$, we see that $[Z]$ and $[\gamma]$ pair nonnegatively as claimed.
\end{proof}

One can modify the example above to give a manifold $M$, flow $\phi$, and surface $Z$ as in the statement of the theorem with $Z$ connected. In the construction, replace the surface $S\cup T$ with a connected surface $R$ with 6 boundary components as indicated on the left of \Cref{fig:connectedexample}. The same methods that Mosher uses to justify \Cref{lem:mosher} can be used to glue the inward and outward pointing boundaries of $N=H\cup (R\times I)$ to produce a hyperbolic manifold together with a pseudo-Anosov flow $\phi$ satisfying all the properties in \Cref{lem:mosher} (c).

Now cap off $R\times\{\frac{1}{2}\}$ inside $H$ as shown on the right of \Cref{fig:connectedexample} to form a connected surface $Z$; Mosher's arguments apply equally well to show that $Z$ is taut. The orbit $\gamma_Z$ furnished by \Cref{lem:mosher}(c) passes through the annulus in $Z$ running between $a_1$ and $a_4$ negatively and exactly once, up to an isotopy of this annulus supported in $H$.

It remains to show that the negative intersection point of $\gamma_Z$ with $Z$ cannot be removed by isotopy. To see this, let $\Sigma$ be the image in $M$ of the outward pointing boundary of $N$. 
Let $\wt \gamma_Z$ be a lift of $\gamma_Z$ to the universal cover $\wt M$ of $M$. Then $\gamma_Z$ passes through a bi-infinite sequence $(\wt \Sigma_n)_{n\in \Z}$ of distinct lifts of $\Sigma$ to $\wt M$.

In the region of $\wt M$ between $\wt \Sigma_0$ and $\wt \Sigma_1$, $\wt \gamma_Z$ has a negative point of intersection with a lift $\wt Z$ of $Z$. 
Let $\wt N$ be the universal cover of $N$ obtained as the component of the preimage of $N \subset M$ in $\wt M$ containing $\wt \Sigma_0$ and $\wt \Sigma_1$ in its boundary. Since $\wt N$ is assembled by gluing together copies of the universal covers of $H$ and $R \times I$, we can see that $\wt Z$ separates $\wt \Sigma_0$ and $\wt \Sigma_1$ by observing that $Z$ separates the positive and negative boundaries of $N$ within each of $H$ and $R \times I$.
Moreover, the intersection of $\wt Z$ with any copy of the universal cover of $H$ or $R \times I$ is connected; together, this implies that $\wt Z$ meets $\wt \gamma_Z$ in a single negative point of intersection.
We conclude that the negative intersection between $Z$ and $\gamma_Z$ cannot be removed by an isotopy and the example is complete. Note that in this example we could have replaced $\gamma_Z$ by any periodic orbit whose pullback to $N$ contains a segment from $A_\front$ to $A_\back$; there are many of these, since the periodic orbits of $\phi$ are dense.

\begin{figure}
\begin{center}
\includegraphics[]{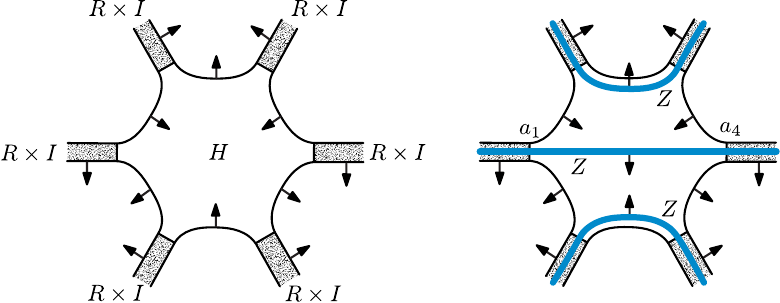}
\caption{Modifying the example of \Cref{thm:example} so that the surface $Z$ in the statement is connected.}
\label{fig:connectedexample}
\end{center}
\end{figure}

\section{Shadows of transverse surfaces}
\label{sec:generalshadows}

In \Cref{sec:flowequiv} we will show that almost transverse position is essentially unique. The key to  this is understanding the projections to the orbit space of lifts of compact connected surfaces transverse to an almost pseudo-Anosov flow. This is the goal of the present section. In the case where the ambient manifold has no boundary, this was essentially done by Fenley in \cite[Section 4]{Fen09}.

First we set some notations and terminology. 

\subsection{Notation and conventions}

Let $\orb^s$ and $\orb^u$ denote the stable and unstable foliations, respectively, of the orbit space $\orb$ of an almost pseudo-Anosov flow $\psi$. 
For $p\in \orb$, let $\orb^s(p)$ denote the leaf of $\orb^s$ through $p$, and similarly for $\orb^u(p)$. Recall that $\Theta \colon \wt M \to \orb$ is the projection to the flow space from the universal cover of $M$. If $p\in \orb$, we set $\gamma(p):=\Theta^{-1}(p)$ to be the orbit of $\wt \psi$ corresponding to $p$.

A \textbf{quadrant} of $p\in \orb$ is a component of $\orb\cut(\orb^s(p)\cup \orb^u(p))$. Points in $\del\orb$, as well as those connected to $\del \orb$ by
singular leaves of either  $\orb^s$ or $\orb^u$,
have infinitely many quadrants; all other points have $2n$ quadrants for some $n\ge 2$. Two quadrants are \textbf{adjacent} if they meet along a noncompact half-leaf.

A \textbf{stable} or \textbf{unstable slice leaf} is a properly embedded copy of $\R$
contained in a leaf of $\orb^s$ or $\orb^u$ respectively. A \textbf{stable} or
\textbf{unstable boundary slice ray} is a ray embedded in $\orb$ meeting $\del\orb$
exactly at its basepoint, obtained as the union of finitely many blown segments with a noncompact half-leaf of $\orb^s$ or $\orb^u$, respectively.

We observe that each component of $\del\orb$ is a slice leaf of both $\orb^s$ and
$\orb^u$. Each slice leaf which is not an entire component of $\del\orb$ separates $\orb$
into either $2$ or $3$ components. The $3$ component case arises when a slice leaf meets $\del\orb$ in finitely many, necessarily contiguous, blown segments.
Similarly, a boundary slice ray separates $\orb$ into $2$
components.

Now suppose $S$ is a properly embedded compact surface transverse to $\psi$. Lift $S$ to $\wt S\subset \wt M$, and let $\Omega=\Theta(\wt S)\subset \orb$. (Here $\Omega$ stands for \emph{ombre}, or shadow.)
Note that the restriction of $\Theta$ to $\wt S$ is a homeomorphism onto $\Omega$ since $\wt S$ separates $\wt M$ and orbits intersect $\wt S$ positively.
If $a$ is a blown annulus of $\psi$ and $S\cap a$ is nonempty, then it is either
$(1)$ a collection of nonseparating arcs or $(2)$ a collection of circles homotopic to the core of $a$.  If $\wt a$ is a lift of $a$ to $\wt M$ that intersects $\wt S$, we see that $\Theta(\wt a) \cap \Omega$ is  equal to the entire blown segment $\Theta(\wt a)$ in case $(1)$, and equal to $\Theta(\wt a)$ minus its endpoints in case $(2)$. If $a$ is not contained in the boundary of $M$ and its intersection with $S$ is a (possibly empty) collection of arcs, then $a$ can be collapsed to obtain a new almost pseudo-Anosov flow transverse to $A$ with fewer blown annuli.

Hence for simplicity we assume that if $S$ intersects a blown annulus in a collection of arcs,
then that annulus lives in $\del M$; we say $S$ has the \textbf{clean annulus property} with respect to $\psi$. Note that assuming $S$ has the clean annulus property is slightly weaker than $\psi$ being a minimal blowup with respect to $S$, because we are allowing the existence of blown annuli not contained in $\del M$ which do not meet $S$. The advantage of this relaxation is that it allows us to individually consider the components of a disconnected surface that is minimally transverse to $\psi$. Such a surface may not be componentwise minimally transverse; however, each component has the clean annulus property.

Without the clean annulus property, the entire discussion in this section goes through with minor modifications.

\subsection{Executive summary}

The main results of this section are fairly easily stated and we do so now. Their proofs appear in \Cref{sec:shadowarg}.

As before, let $\psi$ be an almost pseudo-Anosov flow in a compact 3-manifold $M$ which is transverse to a compact, connected, properly embedded surface $S$. Let $\wt S$ be a lift of $S$ to the universal cover $\wt M$ of $M$. Let $\Omega=\Theta(\wt S)\subset \orb$; this set is connected and disjoint from its frontier $\fr(\Omega)$ because it is open. 

\begin{proposition}[Shadows of transverse surfaces]
\label{prop:shadow_front}
The frontier of $\Omega= \Theta(\wt S)$ is a disjoint union of stable/unstable slice leaves and boundary slice rays, each isolated from the others and bounding $\Omega$ to one side.
\end{proposition}

This is an immediate consequence of the more technical \Cref{lem:front_prop}.

A \textbf{blowup tree} is defined to be either a maximal connected union of blown segments or a singular point not contained in a blown segment.

\begin{proposition}\label{prop:localshadows}
Let $\psi$ be an almost pseudo-Anosov flow on a compact manifold $M$ and let $S$ be a 
properly embedded
compact surface transverse to $\psi$. Let $\wt S$ be a lift of $S$ to $\wt M$ and let $\Omega= \Theta(\wt S)$ be its image in the flow space. Let $T$ be a blowup tree that intersects $\cl(\Omega) \subset \orb$. 

If $T \cap\Omega$ meets the interior of a blown segment, then either:
\begin{enumerate}[label=(\alph*)]
\item $T\cap \Omega=\intr(s)$ for a blown segment $s$ that meets two nonadjacent quadrants,
\item $T\cap \Omega=\intr(s)$ for a blown segment $s\subset\del \orb$, or
\item $T$ is an entire component of $\del \orb$ contained in $\Omega$.
\end{enumerate}

If $T$ is disjoint from $\Omega$, then:
\begin{enumerate}[resume, label=(\alph*)]
\item $T\cap \fr(\Omega)$ is either a single point or an arc comprised of finitely many blown segments. Moreover, $\Omega$ meets precisely two quadrants of $T$, and these quadrants are adjacent.
\end{enumerate}
\end{proposition}

\begin{figure}
\centering
\includegraphics[width=6in]{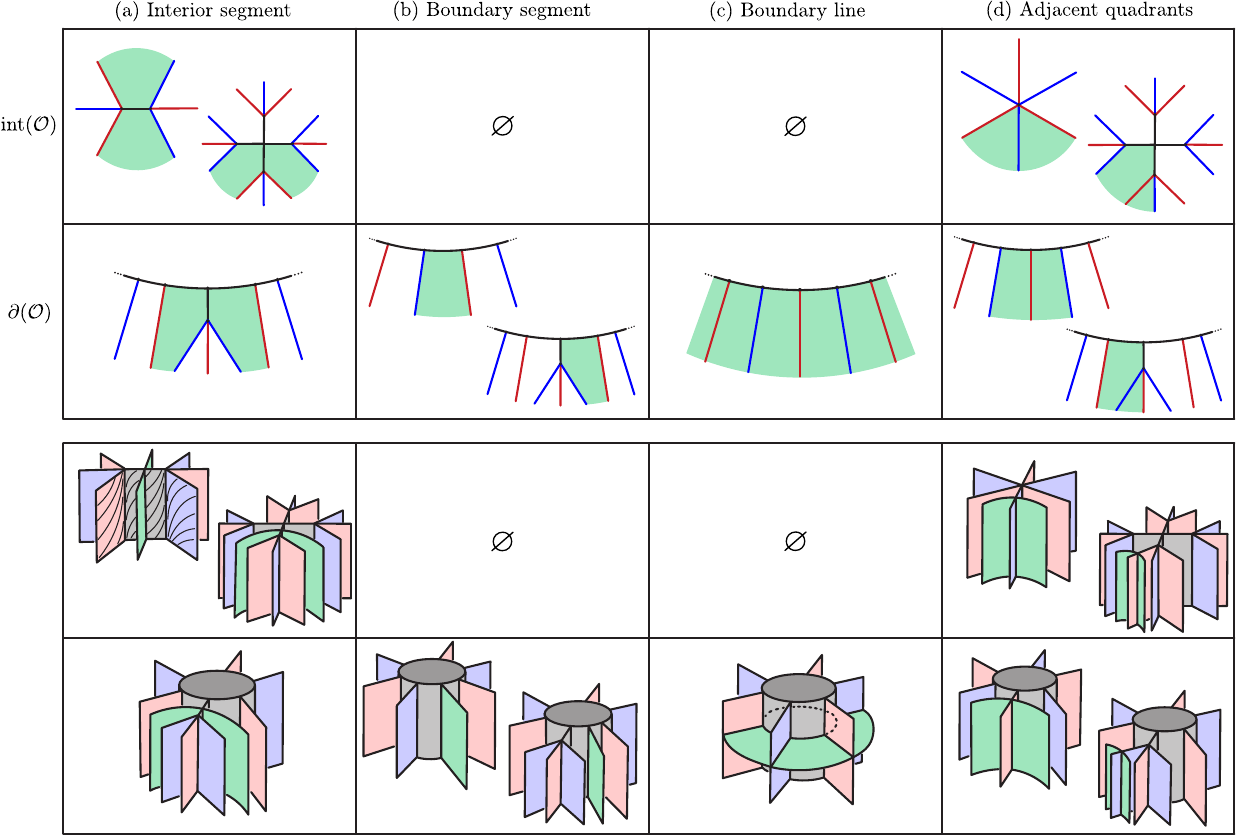}
\caption{\emph{Upper chart:} simple examples of the local pictures described in cases (a)-(d) of \Cref{prop:localshadows}, with examples showing possible interaction with $\del\orb$ in the second row. The set $\Omega$ is indicated in green. Cases (b) and (c) only concern boundary segments, which is why two spaces are empty. \emph{Lower chart}: pictures of the corresponding pieces of the surface (green) in $M$.}
\label{fig:shadows}
\end{figure}

See the top chart in \Cref{fig:shadows} for a visual companion to \Cref{prop:localshadows}.

\begin{remark}\label{remark:3Dinterpretation}
The situations described \Cref{prop:localshadows} translate to straightforward statements about $S$ in $M$.
Recall that $S$ has the clean annulus property. 
Thus $S$ intersects each  blown annulus in a collection of closed curves or not at all, unless the blown annulus lies in $\del M$, in which case $S$ may intersect a blown annulus $a$ in a family of arcs connecting the two boundary components of $a$. From this one sees that cases (a)-(c) in \Cref{prop:localshadows} correspond to the following situations in $M$:
\begin{enumerate}[label=(\alph*)]
\item $S$ passes through a blown annulus that lies in the interior of $M$.

\item A boundary component of $S$ lies interior to a blown annulus $a\subset \del M$.

\item A boundary component of $S$ meets each blown annulus that lies on  the boundary component of $M$ corresponding to the given component of $\del\orb$.
\end{enumerate}
For the final case, note that the two quadrants in (d) of \Cref{prop:localshadows} share a single periodic half-leaf of $\orb^u$ or $\orb^s$, corresponding to a half-leaf $\wt \lambda$ of $\wt W^u$ or $\wt W^s$. By analyzing the intersection of $\langle g\rangle\cdot \wt S$ with $\wt \lambda$, where $\langle g\rangle\le \pi_1(M)$ is the stabilizer of $\wt \lambda$, 
one can show without too much work that $S$ intersects the projection $\lambda$ of $\wt \lambda$ to $M$ in a circle. 
(This also follows from the proof of \Cref{lem:bdyperiodic}). Thus we have:
\begin{enumerate}[resume, label=(\alph*)]
\item There is a periodic half-leaf $\lambda$ of $W^u$ or $W^s$ such that $S \cap \lambda$
contains a circle. 
\end{enumerate}
The corresponding pictures are shown in the lower chart of \Cref{fig:shadows}.
\end{remark}

\subsection{Shadows of transverse surfaces: details}\label{sec:shadowarg}

Note that $\wt S$ separates $\wt M$.
If $B$ is a subset of $\orb$, we say $B$ \textbf{lies above} or \textbf{lies below} $\wt S$ if $\Theta^{-1}(B)$ lies entirely above or below $\wt S$ in $\wt M$, respectively (with respect to the coorientation of $\wt S$).

\smallskip

The following is a basic observation about asynchronously proximal orbits.

\begin{lemma}\label{lem:proximal2}
Let $S$ be a compact, connected surface in an irreducible 3-manifold $M$ which is transverse to a flow $\phi$, and let $\wt S$ be a lift of $S$ to $\wt M$. Suppose that the orbits passing through $a,b\in \wt M$ are asynchronously proximal. Then for all $t\gg 0$, $\wt \phi_t(a)$ and $\wt \phi_t(b)$ lie on the same side of $\wt S$. The same is true for all $t\ll 0$ if $a$ and $b$ are backward asynchronously proximal.
\end{lemma}

\begin{proof}
We shall prove the forward statement; the backward one is similar. 
Let $N$ be an $\epsilon$-tubular neighborhood of $S$ for some $\epsilon>0$ such that $N$ can be identified with $S\times I$ where the $I$-fibers are orbit segments of $\phi$. Let $\wt N$ be the lift of $N$ containing $\wt S$. 
If $\wt \phi_t(a)$ and $\wt \phi_t(b)$ lie on different sides of $\wt S$ for large $t$, then they will be at least  $2\epsilon$ apart for large $t$, hence are not asynchronously proximal.
\end{proof}

Let $\ell$ be a periodic half-leaf based at the periodic point $x \in \orb$.  We say that $\ell$ is \define{locally expanding at} $x$ if the local first return map to a transverse disk in $M$ through a point along the closed orbit corresponding to $x$ expands along the periodic leaf corresponding to $\ell$. Note that any blown segment is locally expanding at one of its endpoints and locally contracting (similarly defined) at its other. We also observe that a periodic half-leaf which is not a blown segment is locally expanding/contracting depending on whether it belongs to $\orb^u$ or $\orb^s$, respectively, and that the half-leaves terminating at a periodic point alternate between locally expanding and locally contracting in the circular order at that point. Each periodic half-leaf has at least one endpoint and we endow it with a unique \textbf{orientation} defined by requiring half-leaves to be oriented toward locally contracting endpoints and away from locally expanding endpoints.

\begin{lemma}\label{lem:blownsegs}
Let $\psi$ be an almost pseudo-Anosov flow on a compact manifold $M$ and let $S$ be a properly embedded compact surface transverse to $\psi$. Let $\wt S$ be a lift of $S$ to $\wt M$ and let $\Omega= \Theta(\wt S)$ be its image in the flow space. Let $s$ be a blown segment that intersects $\Omega$. 

If $s$ is not a subset of $\del\orb$, then:
\begin{enumerate}[label=(\alph*)]
\item $\intr(s)\subset \Omega$ and $\del s\subset \fr(\Omega)$.
\end{enumerate}

If $s\subset \del\orb$ then either:
\begin{enumerate}[resume, label=(\alph*)]
\item $\intr(s)\subset \Omega$ and $\del s\subset\fr(\Omega)$, or
\item the component of $\del \orb$ containing $s$ lies entirely in $\Omega$.
\end{enumerate} 

In cases (a) and (b), the locally expanding and contracting endpoints of $s$ lie under and over $\wt S$, respectively.
\end{lemma}

\begin{proof}

Let $A$ be a blown annulus. If $S\cap A$ is a family of circles, then each circle must be cooriented toward the attracting boundary component of $\del A$. Choosing a lift $\wt A$ of $A$ that intersects $\wt S$, we see that $\wt S\cap \wt A$ is a line intersecting every orbit in $\wt A$, cooriented toward the attracting side of $\wt A$. Letting $s=\Theta(\wt A)$, we see that $\intr(s)\subset \Omega$ and $\del s \subset \fr(\Omega)$, with the locally expanding and locally contracting endpoints of $s$ lying below and above $\wt S$, respectively. This describes situations (a) and (b).
Now suppose that $A$ lies in a component $T$ of $\del M$ and $S\cap A$ is a family of intervals. Then each component of $S\cap T$ intersects each blown annulus in $T$ in a family of intervals. Lifting to $A$ to $\wt A\subset \wt M$, we see the entire component of $\del \orb$ containing $\wt A$ is contained in $\Omega$. This is case (c).

Since $S$ has the clean annulus property there are no other types of nontrivial intersections between $S$ and blown annuli, so the proof is complete.
\end{proof}

\begin{lemma}[Local frontier]\label{lem:localboundary}
Let $x\in \fr (\Omega)$ lie below $\wt S$. If $x\notin \del \orb$, then there is an open interval $I\subset \orb^s(x)$ such that $x\in I$ and $I\subset \fr(\Omega)$.
If $x\in \del \orb$, then the same is true unless $x$ is a limit of points in $\Omega\cap \del \orb$, in which case $I$ can be taken to be a half-closed interval with boundary point $x$ such that $I-\{x\}\subset \orb \ssm \partial \orb$.

If $x$ lies above $\wt S$, the same is true after replacing below with above and $\orb^s$ with $\orb^u$.
\end{lemma}

\begin{figure}
\begin{center}
\includegraphics[height=1in]{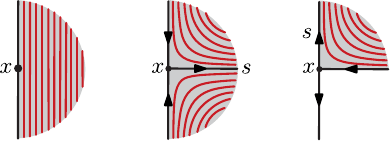}
\caption{Left: the local picture around a point $x$ in $\fr(\Omega)$ in the setting case 1 of the proof of \Cref{lem:localboundary}. Center and right: the local pictures around $x$ in cases 2 and 3 of the same proof, respectively. In the righthand picture the vertical (in the page) line through $x$ lies in $\del \orb$, while in the center and left pictures it may or may not lie in $\del \orb$.}
\label{fig:boundarylocal}
\end{center}
\end{figure}

\begin{proof}
We will assume that $x$ lies below $\wt S$; the proof in the other case is symmetric.

Let $D$ be a disk in $\wt M$ transverse to $\wt \psi$ intersecting the lifted orbit $\gamma(x) = \Theta^{-1}(x)$ 
 in a point. Because $\gamma(x)$ is disjoint from $\wt S$, we can choose $D$ small enough so that it is also disjoint from $\wt S$.

Let $N(x)$ be a neighborhood of $x$ contained in $\Theta(D)$ and small enough so that it contains no singular points or endpoints of blown segments, except possibly $x$. 

\medskip

\noindent\textbf{Case 1:} $x$ is not an endpoint of a blown segment that meets $\Omega$. 

Let $x'$ be any point in $N(x)\cap \orb^s(x)$. Since $x$ and $x'$ are not separated by a singular point by the definition of $N(x)$,
$\gamma(x')$ is asynchronously proximal to $\gamma(x)$ (\Cref{lem:proximal}). Since $\gamma(x)$ lies below $\Omega$, \Cref{lem:proximal2}
implies that $\gamma(x')$ lies below $\Omega$ so $x'\notin \Omega$. This shows that $N(x)\cap \orb^s(x)$ is disjoint from $\Omega$ and lies below $\wt S$.

Now take a sequence of points $(x_n)$ in $N(x)\cap \Omega$ converging to $x$ such that each $x_i$ lies in a single complementary component $Q$ of $\orb^s(x)$. 
Let $d_i$ be the intersection of $\gamma(x_i)$ with $D$. Since $D$ lies entirely below $\wt S$ and all the points in $W^s(d_i)\cap D$ are asynchronously proximal, \Cref{lem:proximal2} gives that all points in $W^s(d_i)\cap D$ intersect $\wt S$ in forward time. Hence $\orb^s(x_i)\cap N(x)\subset \Omega$. Letting $i\to \infty$ these segments limit on  $\del Q\cap N(x)$, so $\del Q\cap N(x)\subset \cl (\Omega)$. 
Since $\del Q\subset \orb^s(x)$ , $\del Q\cap N(x)\subset \fr (\Omega)$ and $\del Q\cap N(x)$ lies below $\wt S$. The local picture is shown on the left of \Cref{fig:boundarylocal}.

\medskip
\noindent\textbf{Case 2:} $x$ is an endpoint of a blown segment $s$ meeting $\Omega$, and $\intr s\cap \del \orb=\varnothing$.

By \Cref{lem:blownsegs}, $\intr(s)\subset \Omega$ and $s$ is locally expanding at $x$ since $x$ lies below $\wt S$.

Let $Q_1$ and $Q_2$ be the two complementary components of $\orb^s(x)$ which meet along
$s$. Let $\ell_1$ and $\ell_2$ be the two periodic locally contracting half-leaves at $x$
adjacent to $s$, indexed so that $\ell_i\subset \del Q_i$. As in case 1,
\Cref{lem:proximal2} implies that $\Omega$ is disjoint from both $\ell_1\cap N(x)$ and $\ell_2\cap N(x)$, and that $\gamma(x')$ lies below $\wt S$ for all $x'\in \ell_1\cup \ell_2$.  

Since $\intr(s)\subset \Omega$ and $\Omega$ is open we can choose a sequence of points $(x_n)$ in $\Omega$ converging to $x$ so that each $x_i$ lies in, say, $Q_1\cap N(x)$. 
As in case 1, $\orb^s(x_i)\cap N(x)\subset \Omega$; these segments accumulate on $(\ell_1\cup s)\cap N(x)$ so $\ell_1\cap N(x)\subset \fr(S)$.
We can argue the same way in $Q_2$ to complete the proof of this case. The local picture is as shown in the center of \Cref{fig:boundarylocal}.

\medskip
\noindent \textbf{Case 3:} $x$ is in the boundary of a blown segment $s\subset\del \orb$ intersecting $\Omega$. 

The argument for this case is the same as that of case 2, except now $s$ is adjacent to only a single complementary component of $\orb^s(p)$ since it lies in $\del \orb$. This produces a half-closed interval $I$ in $\orb^s(x)\cap \fr(\Omega)$, with boundary point $x$. Every point in $I-\{x\}$ must lie in $\intr(\orb)$ because $x$ is a 3-pronged singularity and $s\subset\del \orb$. The local picture is shown on the right of \Cref{fig:boundarylocal}.
\end{proof}

The next lemma turns the previous local statement into a global one.

\begin{lemma}\label{lem:localtoglobal}
Let $x\in \fr(\Omega)$. If $x$ lies below $\wt S$, then $x$ is contained in a unique slice leaf or boundary slice ray contained in $\fr(\Omega)\cap \orb^s(x)$, which lies below $\wt S$. The symmetric statement is true if $x$ lies above $\wt S$, with $\orb^u$ replacing $\orb^s$.
\end{lemma}

\begin{proof}
Let $\ell$ be the component of $\orb^s(x) \ssm (\Omega \cap \partial \orb)$ containing $x$ and let $\lambda$ be the component of $\ell \cap \fr(\Omega)$ that contains $x$. We will show that $\lambda$ lies below $\wt S$ and is a slice leaf if $\ell = \orb^s(x)$ and a boundary slice ray otherwise.

Observe that $\ell$ is a closed, connected subspace of $\orb^s(x)$ and hence a properly embedded tree in $\orb$ with at most one degree $1$ vertex on $\partial \orb$ (since a leaf meets at most one component of $\partial \orb$). By \Cref{lem:localboundary}, $\lambda$ contains an interval around each of its points; this interval is open for points in $\intr(\orb)$ and half open otherwise.
We also have that $\lambda$ is closed as a subspace of $\ell$ since $\fr(\Omega)$ is closed.
Hence, $\lambda$ is a properly embedded subtree of $\orb^s(x)$ with at most one degree $1$ vertex on $\partial \orb$ (exactly in the case when $\ell \subsetneq\orb^s(x)$). 
Moreover, $\lambda$ lies below $\wt S$: it is connected and disjoint from $\Omega$, contains $x$ which lies below $\wt S$, and points lying below $\wt S$ cannot be accumulated by points lying above $\wt S$ because ``not lying above $\wt S$" is an open condition.

It remains to show that $\lambda$ is a line or half line. 
Otherwise, it has a vertex $v$ of degree $n\ge 3$, but only one of the complementary regions of $\lambda$ incident to $v$ contains $\Omega$, contradicting the fact that $\lambda\subset \fr(\Omega)$.
\end{proof}

The following proposition establishes \Cref{prop:shadow_front} and more.
\begin{proposition}
\label{lem:front_prop}
Let $\ell$ be a component of $\fr(\Omega)$. Then the following hold:
\begin{enumerate}[label=(\alph*)]
\item $\ell$ is a slice leaf or boundary slice ray of $\orb^s$ or $\orb^u$, and $\ell$ is stable (unstable) if and only if $\ell$ lies below (above) $\wt S$.
\item $\ell$ has a neighborhood disjoint from all other components of $\fr(\Omega)$.
\item if $\ell$ contains periodic points, then there is a unique one $p$ with the property that there is a half-leaf $\lambda$ terminating at $p$ that intersects $\Omega$. Moreover, $\lambda$ is the unique such half-leaf, and $\lambda$ is locally expanding or contracting at $p$ depending on whether $\ell$ lies below or above $\wt S$, respectively.
\end{enumerate}
\end{proposition}

\begin{proof}
Let $x\in \fr(\Omega)$ and suppose without loss of generality that $x$ lies below $\wt S$. Using
\Cref{lem:localtoglobal} we can find a boundary slice leaf in $\orb^s(x)\cap \fr(\Omega)$, or a
slice ray in $\orb^s(x)\cap \fr (\Omega)$ terminating in a point accumulated by points in $\del \orb\cap \Omega$. 

Let $\lambda$ be this slice leaf or boundary
slice ray. \Cref{lem:localtoglobal} also tells us that $\lambda$ lies below $\wt S$. Note
that $\Omega$ is connected, so it only accumulates on $\lambda$ from one side of $\lambda$. This
implies that $\lambda$ is a maximal connected subset of $\orb^s(x)$ contained in $\fr(\Omega)$.

Let $\ell_1$ and $\ell_2$ be two slice leaves or boundary slice rays in $\fr(\Omega)$. We claim $\ell_1\cap \ell_2=\varnothing$. If $\ell_1\cap \ell_2$ is nonempty, it could consist of a single point or some union of half-leaves. In either case, $\ell_1\cup \ell_2$ separates $\orb$ into a number of components, none of which contains all of $\ell_1\cup \ell_2$ in its frontier. Since $\Omega$ lies in a single component of $\ell_1\cup\ell_2$, this is a contradiction.

It follows that $\fr(\Omega)$ is a disjoint union of slice leaves or boundary slice rays. Let $C$ be a component of $\fr( \Omega)$.
If $C$ consists of more than one slice leaves or boundary slice rays then it would have nonempty interior, a contradiction.
We conclude that $C$
is a single slice leaf or boundary slice ray. If the component contains any point lying under $\wt S$, we have seen that it is stable and lies entirely under $\wt S$. Conversely suppose there is a stable component $\ell^s$ containing a point $p$ lying over $\wt S$. Then by \Cref{lem:localtoglobal} there is an unstable slice leaf $\ell^u$ passing through $x$ contained in $\fr (\Omega)$, contradicting that $\fr(\Omega)$ is a disjoint union of slice leaves or boundary slice rays. This proves (a). 

As noted above, $\Omega$ must lie to one side of each slice or slice ray in $\fr(\Omega)$, and we know now these are precisely the components of $\fr(\Omega)$. This implies that boundary components do not accumulate on each other, proving (b).

Turning our attention to (c), suppose that $p$ is a periodic point in $\ell$ lying below $\Omega$. If $p$ is the boundary point of blown segment intersecting $\Omega$, we have already seen in the proof of \Cref{lem:localboundary} (cases 2 and 3) that this blown segment is locally expanding at $p$ and is the only half-leaf terminating at $p$ that intersects $\Omega$. Otherwise we are in case 1 of the proof of \Cref{lem:localboundary}, and there is an open interval $I\subset \fr(\Omega)$ containing $x$ that is the limit of segments of leaves of $\orb^s$, accumulating on $I$ from the $\Omega$ side. In this situation there can be at most one half-leaf terminating at $p$ lying on the $\Omega$ side of $I$. If there is such a half-leaf it is transverse to $\orb^s$, and hence is an unstable half-leaf which is not a blown segment. In particular it is locally expanding at $p$. 

Next we prove there is at most one periodic point $p$ in $\ell$ which is an endpoint of a periodic half-leaf intersecting $\Omega$. If $p\in \ell$ is a periodic point incident to two half-leaves in $\ell$ which are locally contracting at $p$, then there must be at least one (and hence only one by above) locally expanding half-leaf emanating from $p$ into $\Omega$. This is because the half-leaves meeting $p$ alternate between locally expanding and contracting. Call such a point $p$ a \emph{sink} of $\ell$. If $\ell$ is a slice leaf then the two noncompact half-leaves in $\ell$ are stable, and hence oriented toward their endpoints. This forces the existence of at least one sink in $\ell$. If $\ell$ is a boundary slice ray containing no sinks, then all its half-leaves must be consistently oriented, so the half-leaf of $\ell$ meeting $\del \orb$ is locally contracting at its endpoint $q$. There is then a blown segment in $\del\orb\cap \Omega$ meeting $q$ which is locally expanding at $q$. 

Now suppose that $\ell$ contains at least two periodic points which are endpoints of periodic half-leaves intersecting $\Omega$. Choose two such points $p$ and $p'$ which are adjacent along $\ell$ and are the endpoints of two such half-leaves, $r$ and $r'$ respectively. By above, $r$ and $r'$ are both locally expanding. This means that the union of $r$, $r'$, and the segment of $\ell$ between $p$ and $p'$ is not consistently oriented. Since this is part of the boundary of a quadrant, we have a contradiction (see \Cref{fig:repeated-blowup} and related discussion in \Cref{sec:surfaceblowups}). This proves (c) when $\ell$ lies below $\wt S$, and the other case is symmetric.
\end{proof}

\begin{proof}[Proof of \Cref{prop:localshadows}]
Suppose $T$ is a blowup tree containing a blown segment $s$ that intersects $\Omega$. By
\Cref{lem:blownsegs}, it is either the case that (1) $\intr(s)\subset\Omega$ and $\del
s\subset \fr(\Omega)$, or (2) $s$ is contained in a component $\ell$ of $\del \orb$
that lies entirely in $\Omega$. In case (1), \Cref{lem:front_prop} implies that $T\cap
\Omega$ equals $\intr(s)$. In case (2) we need only check that there are no blown
segments in $\Omega$
incident to $\ell$, but this is true since $S$ has the clean annulus property. Hence $T\cap \Omega=\ell$, completing the analysis when $T$ intersects $\Omega$.

To finish the proof, suppose $T$ is disjoint from $\Omega$ and suppose $p\in T\cap \fr(\Omega)$. Suppose without loss of generality that $p$ lies below $\wt S$. By \Cref{lem:front_prop}, $p$ is contained in $\ell$ where $\ell$ is a slice leaf or boundary slice ray of $\orb^s$, and there is a single periodic point in $\ell$ from which a half-leaf $h$ (necessarily locally expanding) emanates into $\Omega$. Since $T$ and $\Omega$ are disjoint, $h$ is not a blown segment. Hence $T\cap \fr(\Omega)$ consists precisely of the blown segments contained in $\ell$. Since the two quadrants of $T$ intersecting $\Omega$ meet along $h$, which is not a blown segment, they are adjacent.
\end{proof}

The following corollary will be used in the next section. A \emph{regular open set} is one that is the interior of its closure.
\begin{corollary}\label{cor:int_cls}
$\Omega$ is regular.
\end{corollary}
\begin{proof}
First note that $\fr( \Omega)= \fr (\closure(\Omega))$. Indeed, the fact that each boundary component bounds $\Omega$ to only one side gives that $\fr( \Omega)\subset \fr (\closure(\Omega))$, and the reverse inclusion follows from the openness of $\Omega$. Now $\Omega=\intr (\Omega)=\cl(\Omega)-\fr(\Omega)=\cl(\Omega)-\fr(\cl(\Omega))=\intr(\cl(\Omega))$.
\end{proof}

\section{Uniqueness of almost transverse position}\label{sec:flowequiv} 

The following results, which are the goals of this section, comprise our uniqueness statement from \Cref{D}. 

\begin{theorem}\label{thm:combinatorial equivalence}
Let $\phi$ be a pseudo-Anosov flow on $M$ and let $S_1$ and $S_2$ be isotopic properly embedded surfaces which are minimally transverse to generalized dynamic blow ups $\phi_1^\sharp$ and $\phi_2^\sharp$, respectively. Then $\phi_1^\sharp$ and $\phi_2^\sharp$ are combinatorially equivalent.
\end{theorem}

Recall the minimally transverse condition means each blown annulus of $\phi_i^\sharp$ not contained in $\del M$ is intersected by $S_i$ in a positive number of core curves.

\begin{theorem}\label{thm:same shadow}
Let $S_1$ and $S_2$ be two isotopic properly embedded surfaces minimally transverse to an
almost pseudo-Anosov flow $\phi^\sharp$. Then any two compatible lifts of isotopic components of $S_1$ and $S_2$ to $\wt
M$ have the same projection to the orbit space. Further, $S_1$ and $S_2$ are isotopic along flowlines. 
\end{theorem}

The main tool for proving \Cref{thm:combinatorial equivalence} and \Cref{thm:same shadow} is the structure theory of shadows described in \Cref{sec:generalshadows}.

\subsection{Shadows of isotopic surfaces are unique}\label{sec:proj_unq}

To begin, we need to recall and extend the
notation for flow spaces from \Cref{sec:flowspace}. We let $\orb$ denote the flow space of
the original flow $\phi$, lifted to the universal cover $\wt M$, and let $\Theta:\wt
M\to \orb$ denote the quotient map.  

In the case where $M$ has boundary, $\orb$ will have lines that are the images of the planar components of $\boundary \wt M$. We let $\orb^*$ denote the image of $\orb$ under the quotient map $p$ that collapses each blown annulus to a point, which simply collapses all boundary lines
The resulting space is not locally compact at the images of the lines. If $M$ has no boundary, then $\orb^* = \orb$.

Suppose, as in the statement of \Cref{thm:combinatorial equivalence}, that $S_1,S_2\subset M$ are isotopic surfaces which are respectively minimally transverse to (generalized) dynamic blowups $\phi_1^\sharp$ and $\phi_2^\sharp$ of $\phi$. We denote by $\orb_i$ the flow space of $\phi_i^\sharp$, and by $p_i\colon \orb_i\to \orb^*$ the map collapsing all blown segments to points (note we are identifying the result of this collapsing with $\orb^*$ in the obvious way).

Since $\phi_1^\sharp$ and $\phi_2^\sharp$ are dynamic blowups of $\phi$, there are homeomorphisms
\[
g_i \colon M \ssm ( \text{blowup locus of } \phi_i^\sharp  ) \to M \ssm (\text{blowup locus of }\phi )
\]
which are orbit equivalences from the restriction of $\phi^\sharp_i$ to that of $\phi$, 
and homotopic to the identity through maps $M \ssm ( \text{blowup locus of } \phi_i^\sharp  ) \to M$. The following diagram describes some of the  spaces and maps at play, but need not commute, because $\phi_1^\sharp$ and $\phi_2^\sharp$ are distinct flows.
\begin{equation} \label{arrows}
\begin{tikzcd}
&\wt M\arrow{ld}[swap]{\Theta_1}\arrow{d}{\Theta}\arrow{rd}{\Theta_2}&\\
\orb_1\arrow{rd}[swap]{p_1}&\orb\arrow{d}{p}&\orb_2\arrow{ld}{p_2}\\
&\orb^*&
\end{tikzcd}
\end{equation}

\medskip

Now, we choose compatible lifts $\wt S_1$ and $\wt S_2$ of components of $S_1$ and $S_2$.
For brevity we denote $\Theta_i(\wt S_i)$ by $\Omega_i$. 
Let $\mr\orb$, $\mr\orb_i$, $\mr\orb^*$ respectively denote the spaces $\orb$, $\orb_i$, $\orb^*$ minus all blown segments and/or singular points (these are all canonically homeomorphic via the maps in \Cref{arrows}). Applied to a set or map, the hovering circle notation denotes restriction to the corresponding set. Here are some salient properties of these objects:
\begin{itemize}
\item Each $\mr p_i\colon \mr\orb_i\to \mr \orb^*$ is a homeomorphism, so it respects interiors and closures in its domain and range.
\item Regularity of open sets (recall this means $A=\intr(\cl(A))$) is preserved upon passage to open subspaces. Each $\Omega_i$ is regular in $\orb_i$ by \Cref{cor:int_cls}, so $\mr\Omega_i$ is regular in $\mr\orb_i$.
\end{itemize}

\begin{lemma} \label{lem:puncture_images_equal}
With the above notation, $p_1(\Omega_1) = p_2(\Omega_2)$.
\end{lemma}

\begin{proof}
We first consider the sets of {\em regular} periodic points $P_1$ in $\Omega_1$ and $P_2$ in $\Omega_2$, that is
periodic points which correspond to regular orbits of the flows. We will show
that $p_1(P_1)=p_2(P_2)$ in  $\orb^*$.

Let $y_1\in P_1$ be a regular point. 
Its preimage $\wt \gamma_1=\Theta_1^{-1}(y_1)$ projects to a periodic
$\phi^\sharp_1$-orbit
$\gamma_1$ in $M$.  Let $\gamma_2=g_2^{-1}(g_1(\gamma_1))$. This is a periodic orbit of
$\phi^\sharp_2$, which is homotopic to $\gamma_1$ because $g_1$ and $g_2$ are homotopic to
the identity.

We can lift a homotopy from $\gamma_1$ to $\gamma_2$ to produce $\wt \gamma_2$, a lift of $\gamma_2$ to $\wt M$.
We know that $\wt\gamma_1$ passes from the negative side of $\wt S_1$ to the positive side, intersecting it once. 
We claim that $\wt \gamma_2$ intersects $\wt S_1$ finitely many times, starting on the
negative side and ending on the positive side.
This follows from the fact that, because $\gamma_1$ intersects $ S_1$ positively, the
distance from $\wt\gamma_1$ to $\wt S_1$ is a proper function, and on the other hand the
homotopy from $\wt\gamma_1$ to $\wt\gamma_2$ has bounded tracks.

Similarly, since the lifted homotopy between $\wt S_1$ and $\wt S_2$ has bounded tracks,
we then have that $\wt \gamma_2$ also passes from the negative side of $\wt S_2$ to the
positive side (note there can be at most one point of intersection). Thus
$\Theta_2(\wt\gamma_2)$ is in $P_2$.

Now $\Theta_1(\wt\gamma_1)$ and $\Theta_2(\wt\gamma_2)$ will have the same image in
$\orb^*$ by construction.
Because the roles of $S_1$ and $S_2$ are symmetric, we conclude that $P_1$ and $P_2$
 have the same images in $\orb^*$. Note that $P_i\subset \mr\orb_i$, and $P_i$ is dense in $\mr\Omega_i$. Using this, and the properties itemized before the lemma statement, we have (where interiors and closures are taken in the \emph{punctured} spaces $\mr\orb_i$ and $\mr\orb^*$):

\begin{align*}
\mr p_1(\mr\Omega_1)&= \mr p_1(\intr(\cl(\mr\Omega_1)))=\mr p_1(\intr(\cl(P_1)))=\intr(\cl(\mr p_1(P_1)))\\
&=\intr(\cl(\mr p_2(P_2)))=\mr p_2(\intr(\cl(P_2)))=\mr p_2(\intr(\cl(\mr\Omega_1)))\\
&=\mr p_2(\mr\Omega_2).
\end{align*}
The following lemma implies that $p_1(\Omega_1)=p_2(\Omega_2)$, completing the proof.
\end{proof}

\begin{lemma}
The set $p_i(\Omega_i)$ is uniquely determined by $\mr p_i(\mr \Omega_i)$ in the following sense: 

Let $q$ be a singular point of $\orb^*$ that is contained in the frontier of $\mr p_i(\mr
\Omega_i)\subset \orb^*$. Then $q\in p_i(\Omega_i)$ unless $\mr p_i(\mr\Omega_i)$ meets
exactly two adjacent quadrants of $q$. 
\end{lemma}

\begin{proof}
We have $q\in p_i(\Omega_i)$ if and only if $\Omega_i$ intersects the tree of blown
segments $p_i^{-1}(q)$, so this follows immediately from the local structure described by
\Cref{prop:localshadows}. The excluded case is column (d) of \Cref{fig:shadows}.
\end{proof}

\subsection{Uniqueness of combinatorial type}

We can now prove that the isotopy class of a surface almost transverse to a pseudo-Anosov flow determines a unique combinatorial type of minimal dynamic blowup transverse to the surface. 

\begin{proof}[Proof of \Cref{thm:combinatorial equivalence}]
Let $A$ be a blown annulus complex of $\phi_1^\sharp$, and let $U$ be a standard neighborhood for $A$. Let $D$ be a meridional disk or annulus of $U$, depending on whether $U$ is solid or hollow. Then there are singular stable and unstable foliations of $D$; as described in \Cref{global blowups}, the combinatorial type of $\phi_1^\sharp$ is determined by which quadrants of this bifoliation meet along blown segments (the intersections of blown annuli with $D$), as $A$ varies over all blown annulus complexes.

Suppose $Q_1$ and $Q_2$ are two quadrants meeting along a blown segment $s$ corresponding to the immersed annulus $A_s\subset A$. By minimality, there is an arc of $S_1\cap D$ lying in $Q_1\cup Q_2$ and intersecting $s$ in a point. This point corresponds to a closed curve $\alpha\subset S_1\cap A_s$. 
Let $S_1^\alpha$ be the component of $S_1$ containing $\alpha$. Choose intersecting lifts $\wt S_1^\alpha$, $\wt A_s$ to $\wt M$. Then $p_1(\Omega_1)$ meets two nonadjacent quadrants of $p_1(\wt A_s)$.

Conversely if $p_1(\Omega_1)$ meets two nonadjacent quadrants of a singular point corresponding to a lift of the complex $A$ to $\wt M$, then $\Omega_1$ meets their preimage quadrants in $\orb_1$, and \Cref{remark:3Dinterpretation} explains that this corresponds to a curve of intersection between $A$ and a blown annulus, meaning that the two corresponding quadrants in $D$ must meet along a blown annulus.

Reasoning in this way for all blown annulus complexes, we see the combinatorial type of $\phi_1^\sharp$ is uniquely determined by $p_1(\Omega_1)$. The same is true for $\phi_2^\sharp$. \Cref{lem:puncture_images_equal} finishes the proof of combinatorial equivalence.
\end{proof}

As a consequence, whether or not a dynamic blowup is necessary to achieve transversality can be read off from the data of a relative veering triangulation and a relatively carried surface:

\begin{corollary}\label{cor:no_lad}
Let $\phi$ be a transitive pseudo-Anosov flow on $M$ and let $\tau$ be any compatible veering triangulation. Then, up to isotopy, $S$ is (honestly) transverse to $\phi$ if and only if $S$ is carried by $\tau$ without ladderpole annuli.
\end{corollary}

\begin{proof}
If $S$ is carried by $\tau$ without ladderpole annuli, then the fact that $\tau$ is positive transverse to $\phi$ and the picture within each tube show that $S$ is indeed transverse to $\varphi$.

Let $S$ be transverse to $\varphi$ and suppose that $S$ is isotopic to a surface $S_1$ that is efficiently carried by a veering triangulation with ladderpole annuli. Then $S_1$ is minimally transverse to a nontrivial dynamic blowup $\phi_1^\sharp$ of $\phi$, contradicting \Cref{thm:combinatorial equivalence}.
\end{proof}

\subsection{Uniqueness of transverse position}

\begin{proof}[Proof of \Cref{thm:same shadow}]
Recall the setting: we have two isotopic surfaces $S_1$ and $S_2$ minimally transverse to a single almost pseudo-Anosov flow $\phi^\sharp$. 

We first wish to show that compatible lifts of any pair of isotopic components of $S_1$ and $S_2$ to $\wt M$ have the same projection to the orbit space $\orb$ of $\phi^\sharp$; for notational simplicity we assume that $S_1$ and $S_2$ are connected. Let $\wt S_1$ and $\wt S_2$ be compatible lifts to $\wt M$. By \Cref{lem:puncture_images_equal}, their images contain the same regular points and intersect exactly the same quadrants at each blowup tree in $\orb$. Hence they are equal.

Now we show that $S_1$ and $S_2$ are isotopic along flow lines. The above implies
that compatible lifts of components of $S_1$ and $S_2$ are sections of the bundle of
lifted flow lines over their common shadow in $\orb$. The difference of these sections is a
$\pi_1(S_i)$--invariant function, 
and by adding $t\in [0,1]$ times this function to one section, and projecting to
$M$ we obtain a homotopy through transverse surfaces between the components. 

In fact, we can see that this homotopy is an isotopy as follows:
The preimage $q^{-1}(S_i)$ (where $q \colon \wt M \to M$ is the covering map) divides
$\wt M$ into regions, and its coorientation induces an orientation on the dual tree $T_i$ of
this subdivision. Because $S_1$ and $S_2$ are isotopic, there is an isomorphism between
these oriented trees that preserves the labeling of the edges by conjugates of
$\pi_1(S_1)=\pi_1(S_2)$ in $\pi_1(M)$. Now a flow line $\ell$ in $\wt M$, because it is
positively transverse to $S_i$, determines a subset in $T_i$ which is either an oriented line, ray, or segment or a single vertex.
 Because compatible
lifts of $S_1$ and $S_2$ have the same shadows, the two lines are the same (up to the
isomorphism). In other words, the intersection points $\ell\intersect q^{-1}(S_1)$
and $\ell\intersect q^{-1}(S_2)$ appear in the same order along $\ell$, as labeled by
their corresponding subgroups (or shadows).
The linear interpolation given above must therefore preserve the embeddedness of the
intersection points, and hence is an isotopy.
\end{proof}

\subsection{An application: boundary leaves are periodic}\label{sec:bdyperiodic}
The goal of this subsection is to prove the following proposition. The case where $M$ is closed and $\varphi$ is pseudo-Anosov was proven by Fenley \cite{fenley1999surfaces}.

\begin{proposition}\label{lem:bdyperiodic}
Let $S$ be a connected surface transverse to a transitive almost pseudo-Anosov flow $\phi$. Let $\wt S$ be a lift of $S$ to $\wt M$, and let $\Omega$ be its projection to the flow space. Then any component of $\fr (\Omega)$ is periodic. 
\end{proposition}

For the almost pseudo-Anosov flow $\varphi$ on $M$, let $\kappa$ be a finite collection of orbits that kills its perfect fits (\Cref{prop:tsang}), and let $\tau$ be the associated veering triangulation on the fully-punctured manifold  $M \ssm \kappa_s$ (\Cref{th:AG+LMT}). Let $U$ be a standard neighborhood of $\kappa_s$, so that $\tau$ is a veering triangulation of $M$ relative to the tube system $U$.

We begin with the following additional setup. 
By \Cref{thm:ATtocarried}, since $S$ is transverse to $\phi$, we can isotope $S$
so that it is relatively carried by $\tau$. Further, we may choose an \emph{efficient} carried position. In such a position, whenever $S$ meets a component of $U$, it intersects the corresponding component of $\kappa_s$. Note that since $S$ is still transverse to $\varphi$, we have not changed its shadow $\Omega$ by \Cref{thm:same shadow}.

Let $\wt S$ be a lift of $S$ to $\wt M$.
The surface $\wt S$ traverses various sectors of the preimage of $\tau_U$ in $\wt M$, and each sector of this branched surface is contained in a face of the preimage $\wt \tau$ of $\tau$ in $\wt M$. We say such a face is \emph{traversed} by $\wt S$ if it contains a sector in the preimage of $\tau_U$ traversed by $\wt S$. Recall the definitions of $\mr \orb$ and $\mr \Omega$ from \Cref{sec:proj_unq} (the circle notation denotes removal of all non-regular points).

\begin{lemma}\label{lem:image_tri}
The punctured image $\mr\Omega=\mr \Theta(\wt S)$ in $\mr \orb$ is ideally triangulated by the images of the faces of $\wt \tau$ that are traversed by $\wt S$.
\end{lemma}

\begin{proof}
We begin by noting that the faces of $\wt \tau$ traversed by $\wt S$ triangulate their image in $\mr \orb$. To see this, note that the union of these faces in $\wt M$ determine a properly embedded surface $\wt S_\triangle$ in $\wt M \ssm \wt \kappa_s$ whose image in $M \ssm \kappa_s$ is a surface carried by $\tau$. It follows that the projection $\mr \Theta$ restricted to $\wt S_\triangle$ is a homeomorphism onto its image in $\mr \orb$. We denote this image by $\mr\Omega_\triangle$. For details, see \cite[Lemma 9.4]{LMT21}) where it is also established that  $\fr(\mr\Omega_\triangle)$ is a disjoint union of vertical/horizontal leaves \underline{in $\mr \orb$}. 

We now show that if a face $f$ of $\wt \tau$ is traversed by $\wt S$ then its image $\wh f$ in $\mr \orb$ is contained in $\mr \Omega$. If 
$\wh f$ meets $\mr \Omega$ but is not contained in $\mr \Omega$, then it is crossed by a component $\ell$ of $\fr (\wt S)$. By \Cref{lem:front_prop}, $\ell$ is stable/unstable slice leaf or slice ray. Hence there is a neighborhood $\wh V$ of an end of $\wh f$ 
(i.e. a deleted neighborhood of a vertex) whose preimage has closure $\wt V$ in $\wt M$ that is a closed neighborhood of a component of $\wt \kappa_s$ that does not meet $\wt S$. But as noted above, efficient position implies that $\wt S$ does not meet the corresponding component of $\wt U$, contradicting that $f$ is traversed by $\wt S$. This gives $\mr\Omega_\triangle\subset \mr \Omega$.

We finish by demonstrating the reverse inclusion. Suppose $p\in \mr \Omega$. The orbit $\mr\Theta^{-1}(p)$ passes through $\wt S$. If it does so outside of $\wt U$ then clearly $p\in \mr\Omega_\triangle$, so suppose $\gamma_p$ passes through $\wt S$ inside a component of $\wt U$ corresponding to an ideal vertex $v$ of $\mr \orb$.
There are two adjacent quadrants of $v$ such that $p$ lies in the interior of their union. Moreover, since $\wt S$ intersects the corresponding component of $U$, it traverses a face $\wt f$ of $\wt \tau$ whose projection $\wh f$ lies in the interior of the two quadrants. By definition, $\wh f$
lies in $\mr\Omega_\triangle$. Since $\fr(\mr\Omega_\triangle)$ is a union of vertical and horizontal leaves of $\mr \orb$, we have $p\in \mr\Omega_\triangle$.
\end{proof}
 
We now turn to the proof of the proposition.
\begin{proof}[Proof of \Cref{lem:bdyperiodic}]
Continuing with the discussion above, let $\lambda$ be a component of $\fr( \Omega)$. Since every singular leaf is periodic, we may assume that $\lambda$ is a regular stable or unstable leaf and contains no points of $\wh \kappa$. Without loss of generality, we assume $\lambda$ is unstable.

We recall the notion of an \emph{upward flip move} on a surface relatively carried by $\hbs_U$. These moves were essentially first considered in \cite{Minsky2017fibered} but the version we use here is from \cite{Landry_norm}. If $S$ is carried with positive weights on the two bottom faces of a tetrahedron $\tet$, there is an isotopy sweeping the corresponding portion of $S$ across $\tet$ into a neighborhood of the two top faces of $\tet$. See \Cref{fig:flipmove}.

\begin{figure}
    \centering
    \includegraphics[height=1in]{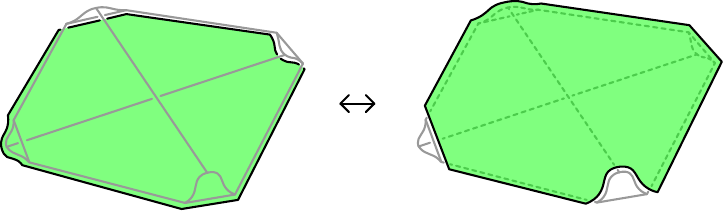}
    \caption{A picture of a flip move.}
    \label{fig:flipmove}
\end{figure}

Note that $S$ is not a cross section to $\phi$, since in that case we would have $\Omega=\orb$. By \cite[Proposition 4.5]{landry2019stable}, we may perform only finitely many upward flips before arriving at a relatively carried position for $S$ in which no upward flips are possible. We fix this position of $S$ and consider the associated triangulation of $\mr \Omega$ in $\mr \orb$ given by \Cref{lem:image_tri}.
By the definition of $B^s$, the train track $B^s\cap S$ has no large branches at this point. As a consequence there is a  finite collection of simple closed curves in $S$ carried by $B^s\cap S$ which we will call \emph{stable loops}. We observe that any infinite train route in $B^s\cap S$ must eventually circle around one of these stable loops forever.

Let $t_1$ be an ideal triangle of the ideal triangulation of $\mr \Omega$ which is near enough to $\lambda$ so that there exists a leaf $\ell^s$ of the stable foliation passing through the interior of $t_1$ and also intersecting $\lambda$. If we orient $\ell^s$ from $t_1$ toward $\lambda$, it determines an infinite path of ideal triangles $(t_1,t_2,t_3,\dots)$ in $\mr\Omega$. 
If $e_i$ is the edge shared by $t_i$ and $t_{i+1}$, we see that the sequence $(e_n)$ limits on $\lambda$. Indeed, this follows from the proof of \cite[Lemma 9.4]{LMT21}. A picture of this situation is shown in \Cref{fig:shadowseq}.

\begin{figure}[h]
    \centering
    \includegraphics[height=1.6in]{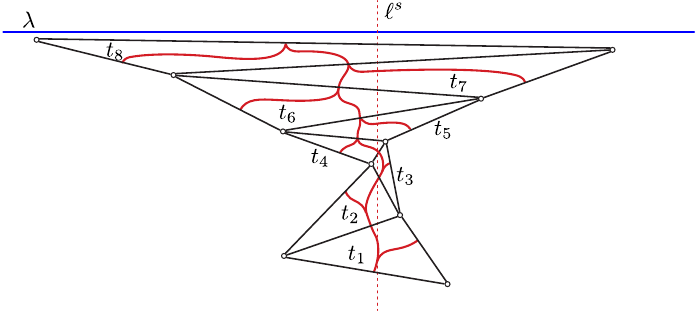}
    \caption{Notation from the proof of \Cref{lem:bdyperiodic}.}
    \label{fig:shadowseq}
\end{figure}

Each $t_i$ has a ``widest" edge, by which we mean an edge that intersects every leaf of $\orb^s$ that intersects $t_i$. 
The projection of $\wt S\cap \wt B^s$ to $\orb$ gives a train track in $\Omega$ which is dual to our ideal triangulation, and endows each ideal triangle in $t_i$ with a switch of a train track that points toward the widest edge of $t_i$. In particular there is a train route from the widest edge of $t_i$ to each of the other edges of $t_i$. Note that since $\ell^s$ is stable, it passes through the widest edge of each $t_i$. Therefore there is an infinite train route $\wh \rho$ in the train track $\wt S\cap \wt B^s$ passing through the triangles $t_1, t_2, t_3,\dots$ in that order. Let $\wt \rho$ be the lift of $\wh\rho$ to $\wt S\subset \wt M$, and let $\rho$ be the projection of $\wt \rho$ to $S\subset M$. As observed above, the route $\rho$ is eventually periodic, winding around a stable loop $\gamma\subset S$. By truncating an initial subsequence of $(t_n)$ and relabeling, we can assume that $\rho$ is actually periodic and that $\wt \rho$ is contained in a lift $\wt \gamma$ of $\gamma$. We orient $\gamma$ so that the orientations of $\wt\rho$ and $\wt \gamma$ are compatible.

Let $g$ be a deck transformation of $\wt M$ which corresponds to $\gamma\in \pi_1(S) \le \pi_1(M)$ and preserves $\wt \gamma$, translating it in the positive direction. The action of $g$ on $\wh S$ is cellular and in particular acts as translation on the (truncated) sequence $(e_n)$, sending each edge to one which is closer to $\lambda$. Since $(e_n)$ limits on $\lambda$ we conclude that $\lambda$ is preserved by the action of $g$, hence is periodic.
\end{proof}

\bibliography{vp3.1_strongtst.bbl}

\def\cprime{$'$} \def\cprime{$'$}
\providecommand{\bysame}{\leavevmode\hbox to3em{\hrulefill}\thinspace}
\providecommand{\MR}{\relax\ifhmode\unskip\space\fi MR }
\providecommand{\MRhref}[2]{%
  \href{http://www.ams.org/mathscinet-getitem?mr=#1}{#2}
}
\providecommand{\href}[2]{#2}
\begin{thebibliography}{LMT23b}

\bibitem[AT24]{AgolTsang}
Ian Agol and Chi~Cheuk Tsang, \emph{Dynamics of veering triangulations:
  infinitesimal components of their flow graphs and applications}, Algebraic
  and Geometric Topology (to appear) (2024).

\bibitem[Bru95]{brunella1995surfaces}
Marco Brunella, \emph{Surfaces of section for expansive flows on
  three-manifolds}, Journal of the Mathematical Society of Japan \textbf{47}
  (1995), no.~3, 491--501.

\bibitem[CC99]{cantwell1999isotopies}
John Cantwell and Lawrence Conlon, \emph{Isotopies of foliated 3-manifolds
  without holonomy}, Advances in Mathematics \textbf{144} (1999), no.~1,
  13--49.

\bibitem[CLR94]{cooper1994bundles}
Daryl Cooper, DD~Long, and Alan~W Reid, \emph{Bundles and finite foliations},
  Inventiones mathematicae \textbf{118} (1994), no.~1, 255--283.

\bibitem[Fen98]{fenley1998structure}
S{\'e}rgio~R. Fenley, \emph{The structure of branching in {A}nosov flows of
  3-manifolds}, Comment. Math. Helv. \textbf{73} (1998), no.~2, 259--297.

\bibitem[Fen99a]{fenley1999foliations}
\bysame, \emph{Foliations with good geometry}, J. Amer. Math. Soc. \textbf{12}
  (1999), no.~3, 619--676.

\bibitem[Fen99b]{fenley1999surfaces}
\bysame, \emph{Surfaces transverse to pseudo-anosov flows and virtual fibers in
  3-manifolds}, Topology \textbf{38} (1999), no.~4, 823--859.

\bibitem[Fen09]{Fen09}
\bysame, \emph{Geometry of foliations and flows {I}: {A}lmost transverse
  pseudo-{A}nosov flows and asymptotic behavior of foiations}, Journal of
  Differential Geometry \textbf{81} (2009), 1--89.

\bibitem[Fen16]{Fen16}
\bysame, \emph{Quasigeodesic pseudo-{A}nosov flows in hyperbolic 3-manifolds
  and connections with large scale geometry}, Advances in Mathematics
  \textbf{303} (2016), 192--278.

\bibitem[FG11]{futgue11}
David Futer and Fran{\c{c}}ois Gu{\'e}ritaud, \emph{From angled triangulations
  to hyperbolic structures}, Contemp. Math. \textbf{541} (2011), 159--182.

\bibitem[FG13]{futer2013explicit}
\bysame, \emph{Explicit angle structures for veering triangulations}, Algebraic
  \& Geometric Topology \textbf{13} (2013), no.~1, 205--235.

\bibitem[FM01]{fenley2001quasigeodesic}
S{\'e}rgio~R. Fenley and Lee Mosher, \emph{Quasigeodesic flows in hyperbolic
  3-manifolds}, Topology \textbf{40} (2001), no.~3, 503--537.

\bibitem[FO84]{floyd1984incompressible}
William Floyd and Ulrich Oertel, \emph{Incompressible surfaces via branched
  surfaces}, Topology \textbf{23} (1984), no.~1, 117--125.

\bibitem[Fri82]{fried1982geometry}
David Fried, \emph{The geometry of cross sections to flows}, Topology
  \textbf{21} (1982), no.~4, 353--371.

\bibitem[Fri83]{fried1983transitive}
\bysame, \emph{Transitive anosov flows and pseudo-anosov maps}, Topology
  \textbf{22} (1983), no.~3, 299--303.

\bibitem[FSS19]{Schleimer2019veering}
Steven Frankel, Saul Schleimer, and Henry Segerman, \emph{From veering
  triangulations to link spaces and back again}, arXiv preprint
  arXiv:1911.00006 (2019).

\bibitem[Gab83]{Gab83}
David Gabai, \emph{Foliations and the topology of 3-manifolds}, Journal of
  Differential Geometry \textbf{18} (1983), 445--503.

\bibitem[GO89]{GO89}
David Gabai and Ulrich Oertel, \emph{Essential laminations in 3-manifolds},
  Annals of Mathematics \textbf{130} (1989), no.~1, 41--73.

\bibitem[HRST11]{hodgson2011veering}
Craig~D Hodgson, J~Hyam Rubinstein, Henry Segerman, and Stephan Tillmann,
  \emph{Veering triangulations admit strict angle structures}, Geom. Topol.
  \textbf{15} (2011), no.~4, 2073--2089.

\bibitem[Lac00]{lackenby2000taut}
Marc Lackenby, \emph{Taut ideal triangulations of 3--manifolds}, Geom. Topol.
  \textbf{4} (2000), no.~1, 369--395.

\bibitem[Lan18]{landry2018taut}
Michael~P. Landry, \emph{Taut branched surfaces from veering triangulations},
  Algebr. Geom. Topol. \textbf{18} (2018), no.~2, 1089--1114.

\bibitem[Lan22]{Landry_norm}
\bysame, \emph{Veering triangulations and the {T}hurston norm: homology to
  isotopy}, Advances in Mathematics \textbf{396} (2022), 108102.

\bibitem[Lan23]{landry2019stable}
\bysame, \emph{Stable loops and almost transverse surfaces}, Groups Geom. Dyn.
  \textbf{17} (2023), no.~1, 35--75.

\bibitem[Li02]{Li02}
Tao Li, \emph{Laminar branched surfaces in 3-manifolds}, Geometry and Topology
  \textbf{6} (2002), 153--194.

\bibitem[LMT23a]{landry2023endperiodic}
Michael~P. Landry, Yair~N. Minsky, and Samuel~J. Taylor, \emph{Endperiodic maps
  via pseudo-{A}nosov flows}, arXiv preprint arXiv:2304.10620 (2023).

\bibitem[LMT23b]{LMT21}
\bysame, \emph{Flows, growth rates, and the veering polynomial}, Ergodic Theory
  and Dynamical Systems \textbf{43} (2023), no.~9, 3026--3107.

\bibitem[LMT24]{LMT20}
\bysame, \emph{A polynomial invariant for veering triangulations}, Journal of
  the European Mathematical Society (online first) \textbf{6} (2024), no.~2,
  731--788.

\bibitem[LT23]{LandryTsangStep1}
Michael~P. Landry and Chi~Cheuk Tsang, \emph{Endperiodic maps, splitting
  sequences, and branched surfaces}, arXiv:2304.14481 (2023).

\bibitem[Mos90]{mosher1990correction}
Lee Mosher, \emph{Correction to `equivariant spectral decomposition for flows
  with a $\mathbb{Z}$-action'}, Ergodic Theory and Dynamical Systems
  \textbf{10} (1990), no.~4, 787--791.

\bibitem[Mos92a]{Mos92}
\bysame, \emph{Dynamical systems and the homology norm of a $3$ -manifold {I}:
  efficient intersection of surfaces and flows}, Duke Math. J. \textbf{65}
  (1992), no.~3, 449--500.

\bibitem[Mos92b]{mosher1992dynamical}
\bysame, \emph{Dynamical systems and the homology norm of a 3-manifold {II}},
  Invent. Math. \textbf{107} (1992), no.~1, 243--281.

\bibitem[Mos96]{Mos96}
\bysame, \emph{Laminations and flows transverse to finite depth foliations},
  Preprint (1996).

\bibitem[MT17]{Minsky2017fibered}
Yair~N. Minsky and Samuel~J. Taylor, \emph{Fibered faces, veering
  triangulations, and the arc complex}, Geom. Funct. Anal. \textbf{27} (2017),
  no.~6, 1450--1496.

\bibitem[Oer86]{oertel1986homology}
Ulrich Oertel, \emph{Homology branched surfaces: {T}hurston's norm on ${H}_2
  ({M}^3)$}, Low-dimensional topology and Kleinian groups (1986), 253--272.

\bibitem[Sha20]{Shannon20}
Mario Shannon, \emph{Dehn surgeries and smooth structures on 3-dimensional
  transitive {Anosov} flows.}, PhD Thesis, 2020.

\bibitem[Thu86]{thurston1986norm}
William~P Thurston, \emph{A norm for the homology of 3-manifolds}, Mem. Amer.
  Math. Soc. \textbf{59} (1986), no.~339, 99--130.

\bibitem[Tsa23]{tsang2022constructing}
Chi~Cheuk Tsang, \emph{Constructing {B}irkhoff sections for pseudo-{A}nosov
  flows with controlled complexity}, Ergodic Theory and Dynamical Systems
  (2023), 1--53.

\bibitem[Tsa24]{Tsang_geodesic}
\bysame, \emph{Veering branched surfaces, surgeries, and geodesic flows}, New
  York Journal of Mathematics (to appear) (2024).

\end{thebibliography}
\bibliographystyle{amsalpha}

\end{document}